\documentclass[reqno]{amsart}

\usepackage{amsmath,amsthm,amssymb,amsfonts}
\usepackage{mathrsfs}
\usepackage{enumerate}
\usepackage{multirow} % for tabular option

\newtheoremstyle{mystyle}% % Name
  {}%                      % Space above
  {}%                      % Space below
  {\normalfont}%           % Body font
  { }%                  % Indent amount
  {\bfseries}%             % Theorem head font
  {}%                      % Punctuation after theorem headｘ
  {10pt}%                  % Space after theorem head, ' ', or \newline
  { }%                     % Theorem head spec (can be left empty, meaning `normal')
\theoremstyle{mystyle}

\newtheorem{theorem}{Theorem}
\newtheorem{proposition}{Proposition}
\newtheorem{lemma}{Lemma}
\newtheorem{remark}{Remark}

\DeclareMathOperator*{\argmin}{arg\,min}

\usepackage[dvipdfmx]{graphicx}
\usepackage[dvipdfmx]{color} 

\usepackage[foot]{amsaddr}
\usepackage[top=3cm,bottom=3cm,left=2.5cm,right=2.5cm,marginparwidth=1.75cm]{geometry}

\usepackage{comment}

\title[Adaptive inference for small diffusion processes]{Adaptive inference for small diffusion processes
based on sampled data}
\author[T Kawai]{Tetsuya Kawai $^{1}$}
\author[M Uchida]{Masayuki Uchida $^{1,2}$}
\address{$^{1}$Graduate School of Engineering Science, Osaka University}
\address{$^{2}$Center for Mathematical Modeling and Data Science (MMDS), Osaka University and JST CREST}

\begin{document}

%%%%%%%%%%%%%%%% abstract  %%%%%%%%%%%%%%%%%%%
\begin{abstract}
We consider parametric estimation and tests for multi-dimensional diffusion processes 
with a small dispersion parameter $\varepsilon$ from discrete observations. For parametric estimation of diffusion processes, 
the main target is to estimate 
the drift parameter and the diffusion parameter. In this paper, we propose two types of adaptive estimators for both parameters and show their asymptotic properties under $\varepsilon\to0$, $n\to\infty$ 
and the balance condition that $(\varepsilon n^\rho)^{-1} =O(1)$ for some $\rho>0$. 
Using these adaptive estimators, we also introduce consistent adaptive testing methods and prove that test statistics for adaptive tests have asymptotic distributions under null hypothesis.
In simulation studies, we examine and compare asymptotic behaviors of the two kinds of adaptive estimators and test statistics. Moreover, we treat the SIR model which describes a simple epidemic spread for a biological application.
\end{abstract}


\maketitle

%%%%%%%%%%%%%% section1 %%%%%%%%%%%%%%%%%%%%%%
\section{Introduction}\label{sec1}
We consider a $d$-dimensional small diffusion process satisfying the following stochastic differential equation (SDE):
\begin{equation}\label{sde}
\begin{cases}
dX_t = b(X_t,\alpha) dt + \varepsilon \sigma(X_t,\beta)dW_t, \quad t\in[0,T],\\
X_0 = x_0,
\end{cases}
\end{equation}
where $W_{t}$ is the $r$-dimensional standard Wiener process, $\alpha\in\Theta_{\alpha}\subset\mathbb{R}^{p},\ \beta \in \Theta_{\beta} \subset \mathbb{R}^{q}$, $\theta =(\alpha,\beta)$, $\Theta := \Theta_{\alpha} \times \Theta_{\beta}$ being compact and convex parameter space, $b:{\mathbb{R}}^{d} \times {\Theta}_{\alpha} \to {\mathbb{R}}^{d}$ and $\sigma:{ \mathbb{R}}^{d} \times {\Theta}_{\beta} \to {\mathbb{R}}^{d}\otimes {\mathbb{R}}^{r}$ are known except for the parameter $\theta$, and the initial value $x_0\in\mathbb{R}$ and the small coefficient $\varepsilon >0$ are known.
We assume the true parameter $\theta_0=(\alpha_0,\beta_0)$ belongs to $\mathrm{Int}(\Theta)$, and
the data are discrete observations $(X_{t^n_k})_{k=0,\ldots, n}$, where $t^n_k=k h_n$ and $h_n=T/n$.

A family of small diffusion processes defined by \eqref{sde} is an important class and called dynamical systems with small perturbations, see Azencott \cite{Azencott_1982}, Freidlin and Wentzell \cite{Freidlin_1998} and Yoshida \cite{yoshida1992Malliavin}. For applications of small diffusion processes to mathematical finance and mathematical biology, see Yoshida \cite{yoshida1992asymptotic}, Uchida and Yoshida \cite{Uchida_Yoshida_2004}, Guy et al  
\cite{Guy_2014, Guy_2015} and references therein.

Asymptotic theory of parametric inference for small diffusion processes 
has been well-developed. 
For continuous-time observations, see Kutoyants \cite{Kutoyants_1984, Kutoyants_1994} and Yoshida \cite{yoshida1992Malliavin,yoshida2003}. 
As for discrete observations, Genon-catalot \cite{genon_1990} studied minimum contrast estimation for the drift parameter and proved that this estimator has
asymptotic efficiency under the assumption $\varepsilon\sqrt{n}=O(1)$. 
Laredo \cite{laredo1990} investigated the asymptotically efficient estimator by using interpolated process under the assumption $(\varepsilon n^2)^{-1}\to 0$. S\o rensen and Uchida \cite{Sorensen_Uchida_2003} studied 
the joint estimation for both drift and diffusion parameters based on minimum contrast estimators. 
They proved that the estimator for drift parameter is asymptotically efficient 
and the estimator for diffusion parameter  is asymptotically normal 
under $(\varepsilon\sqrt{n})^{-1} =O(1)$. 
Uchida \cite{Uchida_2004} investigated the asymptotically efficient estimator for drift parameter 
by using the approximate martingale estimating function 
under $(\varepsilon n^l)^{-1}\to0$, where $l$ is a positive integer. 
For the asymptotically efficient estimator of drift parameter based on 
an approximate martingale estimating function
of a one-dimensional small diffusion process
under $\epsilon \rightarrow 0$ and  $n \rightarrow \infty$,
see  Uchida \cite{Uchida_2008}.
Gloter and S\o rensen \cite{Gloter_2009} generalized the results of S\o rensen and Uchida \cite{Sorensen_Uchida_2003} and Uchida \cite{Uchida_2004}. 
They proposed the minimum contrast estimators for both drift and diffusion parameters 
whose asymptotic covariance matrix equals to that of the estimators in S\o rensen 
and Uchida \cite{Sorensen_Uchida_2003} under
 $(\varepsilon n^\rho)^{-1}\ =O(1)$, where $\rho>0$. 

The adaptive inference for diffusion processes has been studied by many researchers. 
Since the adaptive method can divide the inference for $(\alpha,\beta)$ 
into that for the drift parameter $\alpha$ and that for diffusion parameter $\beta$, 
we expect that the adaptive parametric inference for diffusion processes 
is dealt with more accurately and quickly from the viewpoint of numerical analysis.
For adaptive parametric estimation for ergodic diffusion processes, many researchers studied 
and obtained the asymptotic results, see Prakasa Rao \cite{prakasa1983, prakasa1988}, 
Yoshida \cite{yoshida1992estimation}, Kessler \cite{Kessler_1995}, and Uchida and Yoshida \cite{Uchida_2012}.
Nomura and Uchida \cite{nomura2016adaptive} and Kaino and Uchida \cite{kaino2018hybrid} proposed 
the adaptive Bayes type estimators and the hybrid estimators for both drift and diffusion parameters 
in small diffusion processes. 
They proved the asymptotic efficiency for the estimator of the drift parameter 
and the asymptotic normality for the estimator of diffusion parameter 
whose asymptotic variance equals to that of the estimator in S\o rensen and Uchida \cite{Sorensen_Uchida_2003}.
Moreover their estimators  have convergence of moments under $(\varepsilon \sqrt{n})^{-1}=O(1)$.
For adaptive tests, Kawai and Uchida \cite{Kawai_2020_arxiv} proposed adaptive testing method in ergodic diffusion processes, which can test 
drift parameters and diffusion parameters separately. To construct test statistics of these adaptive tests, adaptive estimators in ergodic diffusions were used and  asymptotic properties of these test statistics were proved. 
Nakakita and Uchida \cite{nakakita2019adaptive} investigated the adaptive test for noisy ergodic diffusion processes. 
They derived the asymptotic null distribution of the adaptive test statistics based on the local means 
and consistency of the tests under alternatives.

In this paper, we utilize the results of Gloter and S\o rensen \cite{Gloter_2009} and propose 
the two types of adaptive estimators for both drift and diffusion parameters in small diffusion processes 
under the assumption $(\varepsilon n^\rho)^{-1}=O(1)$, where $\rho>0$. 
In small diffusions, the convergence rates of estimators for the drift parameter and the diffusion parameter 
are $\varepsilon^{-1}$ and $\sqrt{n}$, respectively. We first estimate drift parameter $\alpha$ and next estimate diffusion parameter $\beta$ in our adaptive estimation methods for arbitrary $\rho>0$. Moreover, if the balance coefficient $\rho$ satisfies $0<\rho<1/2$, we can also utilize simpler adaptive estimation methods than the above methods. In these simpler methods, we first estimate diffusion parameter since it holds that $\varepsilon^{-1}/\sqrt{n}\to0$.

The main result of this paper is that our proposed two kinds of adaptive estimators have asymptotic efficiency 
for $\alpha$ and asymptotic normality for $\beta$ whose asymptotic covariance matrix is the same as 
that of the estimators proposed in S\o rensen and Uchida \cite{Sorensen_Uchida_2003} under the milder assumption than $(\varepsilon \sqrt{n})^{-1}=O(1)$. 
We also give the estimator for drift parameter $\alpha$ 
which can be estimated independently from diffusion parameter $\beta$ 
and show that the estimator has asymptotic normality under the assumption 
$(\varepsilon n^\rho)^{-1} =O(1)$. For adaptive tests, we introduce the two kinds of likelihood ratio type test statistics, which are constructed by the above proposed adaptive estimators. These test statistics have asymptotic distribution under null hypothesis, and these tests are consistent under alternatives.

The paper is organized as follows. 
In Section \ref{sec2}, notation and assumptions are introduced. 
The infinitesimal generator of the small diffusion process 
and its approximation used in constructing the contrast functions for adaptive estimators are defined. 
In Section \ref{sec3}, we propose the two kinds of adaptive estimators and state the asymptotic properties
of the proposed estimators. 
Note that the drift parameter can be estimated 
independently of the diffusion parameter, 
and adaptive estimators can also treat the case of having 
the same parameters in the drift and diffusion coefficients.
In Section \ref{sec4}, adaptive testing method for small diffusion processes 
are introduced and two kinds of test statistics are constructed. 
We show the asymptotic properties of  these test statistics.
In Section \ref{sec5}, 
we give some examples and simulation results of the asymptotic performance for
two types of adaptive estimators and test statistics for multi-dimensional small diffusion processes. 
In model 1, we compare the adaptive estimators 
with the joint estimator proposed in Gloter and S\o rensen \cite{Gloter_2009}, 
and an adaptive test is also conducted.
In model 2, we treat the numerical simulation of the SIR model. 
In model 3, the difference of the asymptotic performances between the two adaptive estimators is examined.
Section \ref{sec6} is devoted to the proofs of the results presented in Sections \ref{sec3} and \ref{sec4}.
%%%%%%%%%%%%%% section2 %%%%%%%%%%%%%%%%%%%% %%

\section{Notation and assumptions}\label{sec2}
In this paper, we set $\partial_{\alpha_i}:=\partial/\partial{\alpha_i},\ \partial_{\beta_i}:=\partial/\partial{\beta_i},\ \partial_{\alpha}:=(\partial_{\alpha_1},\ldots,\partial_{\alpha_{p}})^\top,\ \partial_{\beta}:=(\partial_{\beta_1},\ldots,\partial_{\beta_{q}})^\top$,\ $\partial^2_{\alpha}:=\partial_\alpha\partial_\alpha^\top,\ \partial^2_{\beta}:=\partial_\beta\partial_\beta^\top,\ \partial^2_{\alpha\beta}:=\partial_\alpha\partial_\beta^\top$, where $\top$ is the transpose of a matrix. 
%For ($m\times n$)-matrices $A$ and $B,$ 
%it is defined that $A^{\otimes 2 } :=AA^\top$,\ $\|A\|^2:=\mathrm{tr}(A^{\otimes 2})$,\ $B[A] := \mathrm{tr}(B A^\top)$. 
The symbols $\overset{P}{\to}$ and $\overset{d}{\to}$ indicate convergence in probability and convergence in distribution, respectively. 

Let $(X_t^0)$ be the solution of the following ordinary differential equation (ODE) which is the case of $\varepsilon=0$ and $\alpha=\alpha_0$ in the SDE \eqref{sde}:
\begin{equation}\label{ode}
\begin{cases}
dX_t^0 = b(X_t^0,\alpha_0) dt, \quad t\in[0,T],\\
X_0^0 = x_0,
\end{cases}
\end{equation}
and $I(\theta_0)$ be the $(p+q) \times (p+q)$-matrix defined as 
\begin{align*}
I(\theta_0)&:=\begin{pmatrix}
\left(I_{b}^{i,j}(\theta_0)\right)_{1\le i,j\le p}&0\\
0&\left(I_\sigma^{i,j}(\beta_0)\right)_{1\le i,j \le q}
\end{pmatrix}, 
\end{align*}
where
\begin{align*}
I_{b}^{i,j}(\theta_0)&=\int_{0}^{T}\left(\partial_{\alpha_i}b(X_s^0,\alpha_0)\right)^\top[\sigma\sigma^\top]^{-1}(X_s^0,\beta_0)\left(\partial_{\alpha_j}b(X_s^0,\alpha_0)\right)ds,\\
I_\sigma^{i,j}(\beta_0)&=\frac{1}{2T}\int_{0}^{T}\mathrm{tr}\left[\left(\left(\partial_{\beta_i}[\sigma\sigma^\top]\right)[\sigma\sigma^\top]^{-1}\left(\partial_{\beta_j}[\sigma\sigma^\top]\right)[\sigma\sigma^\top]^{-1}\right)(X_s^0,\beta_0)\right]ds.
\end{align*}
Moreover, we define $p\times p$-matrices $J_b(\alpha_0)=\left(J_b^{i,j}(\alpha_0)\right)_{1\le i,j\le p}$ and $K_b(\theta_0)=\left(K_b^{i,j}(\theta_0)\right)_{1\le i,j\le p}$ as
\begin{align*}
J_{b}^{i,j}(\alpha_0)&=\int_{0}^{T}\left(\partial_{\alpha_i}b(X_s^0,\alpha_0)\right)^\top\left(\partial_{\alpha_j}b(X_s^0,\alpha_0)\right)ds,\\
K_{b}^{i,j}(\theta_0)&=\int_{0}^{T}\left(\partial_{\alpha_i}b(X_s^0,\alpha_0)\right)^\top[\sigma\sigma^\top](X_s^0,\beta_0)\left(\partial_{\alpha_j}b(X_s^0,\alpha_0)\right)ds.
\end{align*}

We make the following assumptions.
\begin{enumerate}
\item[{\bf[A1]}] For all $\varepsilon>0,$ SDE \eqref{sde} with the true value of the parameter has a unique strong solution on some probability space $(\Omega,\mathcal{F},P)$, and ODE \eqref{ode} has a unique solution.
\item[{\bf[A2]}] $b\in \mathcal{C}^{^\infty}(\mathbb{R}^d\times \Theta_{\alpha})$, $\sigma\in \mathcal{C}(\mathbb{R}^d\times \Theta_{\beta})$, and there exists an open convex subset $\mathcal{U}\in\mathbb{R}^d$ such that $X_t^0\in\mathcal{U}$ for all $t\in[0,1]$, and $\sigma\in\mathcal{C}^{\infty}(\mathcal{U}\times \Theta_{\beta})$. Moreover $[\sigma\sigma^\top](x,\beta)$ is invertible on $\mathcal{U}\times \Theta_{\beta}$.
\item[{\bf[A3]}]
\begin{enumerate}[]
\item $b(X_t^0,\alpha) = b(X_t^0, \alpha_{0} )$ for all $t\in[0,1]\ \Rightarrow\ \alpha = \alpha_{0},$
\item $\sigma(X_t^0,\beta) = b(X_t^0, \beta_{0} )$ for all $t\in[0,1]\ \Rightarrow\ \beta = \beta_{0}.$
\end{enumerate}
\item[{\bf[A4]}] 
\begin{enumerate}[(i)]
\item $I_b (\theta_{0})$, $J_b(\alpha_0)$ and $K_b(\theta_0)$ are non-singular,
\item $I_\sigma (\theta_{0})$ is non-singular.
\end{enumerate}
\item[{\bf[B]}] $\varepsilon=\varepsilon_n\to0$ as $n\to\infty$ and there exists $\rho>0$ such that $\displaystyle{\varlimsup_{n\to\infty}}\ (\varepsilon n^\rho)^{-1}< +\infty$.
\end{enumerate}
\begin{remark}\label{rem:1}
(i) Assumption [{\bf A2}] is derived from a localization argument and hence [{\bf A2}] is the mild condition. The strict assumption of [{\bf A2}] and the relationship between the two assumptions are introduced in Section \ref{sec6}.

(ii) For [{\bf A4}], if we assume a condition for $\sigma$, then we can clarify the relationship among the regularity for $I_b (\theta_{0})$, $J_b (\alpha_{0})$ and $K_b (\theta_{0})$. 
Under [{\bf A2'}] in Section \ref{sec6}, in particular, it holds true that if one is regular, then the others are also regular.
\end{remark}

For [{\bf B}], we set the approximation degree $v$ as the integer such that $v=\lceil \rho+\frac{1}{2}\rceil$, 
where $\lceil y\rceil:=\min \{z\in\mathbb{Z}|\ y\le z\}$.
Here the two kinds of operators $\mathcal{L}_{\alpha_0}^\varepsilon$ and $\mathcal{L}_\alpha^0$ are introduced. 
We denote by $\mathcal{L}_{\alpha_0}^\varepsilon$ the infinitesimal generator of the diffusion process $X$: 
for any smooth function $f$,
\begin{align*}
\mathcal{L}_{\alpha_0}^\varepsilon(f)(x)&:=\sum_{i=1}^{d}{b^i(x,\alpha_0)\frac{\partial}{\partial{x_i}}f(x)}+\frac{\varepsilon^2}{2}\sum_{i,j=1}^{d}{[\sigma\sigma^\top]^{i,j}(x,\beta_0)\frac{\partial^2}{\partial{x_i}\partial{x_j}}f(x)},
\end{align*}
and define the simple approximation of the generator $\mathcal{L}_\alpha^0$ as
\begin{align*}
\mathcal{L}_\alpha^0(f)(x)&:=\sum_{i=1}^{d}{b^i(x,\alpha)\frac{\partial}{\partial{x_i}}f(x)}.
\end{align*}
Using the operator $\mathcal{L}_\alpha^0$, 
we set that for $2\le l\le v$,
\begin{align*}
P_{1,k}(\alpha)&:=X_{t_k^n}-X_{t_{k-1}^n}-h_nb(X_{t_{k-1}^n},\alpha), \\
P_{l,k}(\alpha)&:=P_{1,k}(\alpha)-Q_{l,k}(\alpha), \\
Q_{l,k}(\alpha)&:=\sum_{j=1}^{l-1}{\frac{h_n^{j+1}}{(j+1)!}(\mathcal{L}_\alpha^0)^j b(X_{t_{k-1}^n},\alpha)}.
\end{align*}
In particular, 
\begin{align*}
P_{2,k}(\alpha)&=X_{t_k^n}-X_{t_{k-1}^n}-h_nb(X_{t_{k-1}^n},\alpha)-\frac{h_n^2}{2}\sum_{i=1}^{d}{b^i(X_{t_{k-1}^n},\alpha)\frac{\partial}{\partial x_i}b(X_{t_{k-1}^n},\alpha)}, \\
Q_{2,k}(\alpha)&=\frac{h_n^2}{2}\sum_{i=1}^{d}{b^i(X_{t_{k-1}^n},\alpha)\frac{\partial}{\partial x_i}b(X_{t_{k-1}^n},\alpha)}.
\end{align*}
%%%%%%%%%%%%%% section3 %%%%%%%%%%%%%%%%%%%%%%
\section{Adaptive estimation}\label{sec3}
In this section, we propose the following two types of adaptive estimators. 

%%%%%%%%%%%%%%%%%%%%%%%%%%%%%%%%
\subsection{Type I estimator}\label{sec:type1}
\ 
First, we introduce the adaptive estimators which can 
be divided the parameter optimization for $\theta=(\alpha, \beta)$ into the optimization of $\alpha$ and that of $\beta$.
\begin{enumerate}[Step 1.]
\item  Set
\begin{align*}
U_{\varepsilon,n,v}^{(1)}(\alpha)&:=\varepsilon^{-2}h_n^{-1}\sum_{k=1}^nP_{v,k}(\alpha)^\top P_{v,k}(\alpha),
\end{align*}
and $\tilde{\alpha}_{\varepsilon,n}^{(1)}$ is defined as
\begin{align*}
\tilde{\alpha}_{\varepsilon,n}^{(1)}:=\argmin_{\alpha\in\Theta_{\alpha}}U_{\varepsilon,n,v}^{(1)}(\alpha).
\end{align*}

\item Set
\begin{align*}
U_{\varepsilon,n,v}^{(2)}(\beta|\bar{\alpha})&:=\sum_{k=1}^n\left\{\log\det[\sigma\sigma^\top](X_{t_{k-1}^n},\beta)+\varepsilon^{-2}h_n^{-1}P_{v,k}(\bar{\alpha})^\top [\sigma\sigma^\top]^{-1}(X_{t_{k-1}^n},\beta)P_{v,k}(\bar{\alpha})\right\},
\end{align*}
and $\tilde{\beta}_{\varepsilon,n}=\tilde{\beta}_{\varepsilon,n}^{(1)}$ is defined as
\begin{align*}
\tilde{\beta}_{\varepsilon,n}:=\argmin_{\beta\in\Theta_{\beta}}U_{\varepsilon,n,v}^{(2)}(\beta|\tilde{\alpha}_{\varepsilon,n}^{(1)}).
\end{align*}

\item Set
\begin{align*}
U_{\varepsilon,n,v}^{(3)}(\alpha|\bar{\beta})&:=\varepsilon^{-2}h_n^{-1}\sum_{k=1}^nP_{v,k}(\alpha)^\top[\sigma\sigma^\top]^{-1}(X_{t_{k-1}^n},\bar{\beta}) P_{v,k}(\alpha),
\end{align*}
and $\tilde{\alpha}_{\varepsilon,n}=\tilde{\alpha}_{\varepsilon,n}^{(2)}$ is defined as
\begin{align*}
\tilde{\alpha}_{\varepsilon,n}:=\argmin_{\alpha\in\Theta_{\alpha}}U_{\varepsilon,n,v}^{(3)}(\alpha|\tilde{\beta}_{\varepsilon,n}).
\end{align*}

\end{enumerate}
In Step 1, we have the following asymptotic properties.
\begin{lemma}\label{lem:1}
Assume {\bf{[A1]}}-{\bf{[A3]}}, {\bf{[A4]}}-(i) and {\bf{[B]}}. Then it holds that
\begin{align*}
\varepsilon^{-1}(\tilde{\alpha}^{(1)}_{\varepsilon,n}-\alpha_0)=O_P(1).
\end{align*}
In particular,  
\begin{align*}
\varepsilon^{-1}(\tilde{\alpha}^{(1)}_{\varepsilon,n}-\alpha_0)\overset{d}{\to}
N_p(0,J_b(\alpha_0)^{-1}K_b(\theta_0)J_b(\alpha_0)^{-1})
\end{align*}
as $\varepsilon\to0$ and $n\to\infty$.
\end{lemma}
\begin{remark}\label{rem:2}
In order to prove the asymptotic normality of the estimator in Theorem  \ref{thm:1} below, 
we need to show that $\varepsilon^{-1}(\tilde{\alpha}^{(1)}_{\varepsilon,n}-\alpha_0)=O_P(1)$.
%However, using the estimator $\tilde{\alpha}_{\varepsilon,n}^{(1)}$ and the results of Lemma \ref{rem:1}, 
%we can estimate the drift parameter $\alpha$ independent of the diffusion parameter $\beta$
%and ensure that this estimator has asymptotic normality. 
Using $\tilde{\alpha}^{(1)}_{\varepsilon,n}$,
we can estimate the drift parameter $\alpha$ independent of the diffusion parameter $\beta$
and it follows from Lemma \ref{rem:1} that this estimator has asymptotic normality. 
Note that this estimator $\tilde{\alpha}^{(1)}_{\varepsilon,n}$ is not asymptotic efficient in general.
\end{remark}

The main result for Type I method is as follows.
\begin{theorem}\label{thm:1}
Assume {\bf{[A1]}}-{\bf{[A4]}} and {\bf{[B]}}. Then it follows that
\begin{align*}
\tilde{\theta}_{\varepsilon,n}\overset{P}{\to}\theta_0.
\end{align*}
Moreover,  
\begin{align*}
\begin{pmatrix}
\varepsilon^{-1}(\tilde{\alpha}_{\varepsilon,n}-\alpha_0)\\
\sqrt{n}(\tilde{\beta}_{\varepsilon,n}-\beta_0)
\end{pmatrix}
\overset{d}{\to}
N_{p+q}(0,I(\theta_0)^{-1})
\end{align*}
as $\varepsilon\to0$ and $n\to\infty$.
\end{theorem}

%%%%%%%%%%%%%%%%%%%%%%%%%%%%%%%%
\subsection{Type II estimator}\label{sec:type2}
\ 
In the Type II method, we divide the Step 1 of Type I method into 
%many steps. 
several steps.
By this split, we expect that 
the estimators in Type II method are computed more quickly than 
those in the Type I method from the viewpoint of computation time.
\begin{enumerate}[Step 1.]
\item Set
\begin{align*}
V_{\varepsilon,n,v}^{(1)}(\alpha)&:=\varepsilon^{-2}h_n^{-1}\sum_{k=1}^nP_{1,k}(\alpha)^\top P_{1,k}(\alpha),
\end{align*}
and $\hat{\alpha}_{\varepsilon,n}^{(1)}$ is defined as
\begin{align*}
\hat{\alpha}_{\varepsilon,n}^{(1)}:=\argmin_{\alpha\in\Theta_{\alpha}}V_{\varepsilon,n,v}^{(1)}(\alpha).
\end{align*}

\end{enumerate}

\begin{enumerate}[Step]
\item $l=2$ to $v$. Set 
\begin{align*}
V_{\varepsilon,n,v}^{(l)}(\alpha|\bar{\alpha})&:=\varepsilon^{-2}h_n^{-1}\sum_{k=1}^n\left(P_{1,k}(\alpha)-Q_{l,k}(\bar{\alpha})\right)^\top \left(P_{1,k}(\alpha)-Q_{l,k}(\bar{\alpha})\right),
\end{align*}
and $\hat{\alpha}_{\varepsilon,n}^{(l)}$ is defined as
\begin{align*}
\hat{\alpha}_{\varepsilon,n}^{(l)}:=\argmin_{\alpha\in\Theta_{\alpha}}V_{\varepsilon,n,v}^{(l)}(\alpha|\hat{\alpha}_{\varepsilon,n}^{(l-1)}).
\end{align*}
\end{enumerate}

\begin{enumerate}[Step]
\item $v+1$. Set
\begin{align*}
V_{\varepsilon,n,v}^{(v+1)}(\beta|\bar{\alpha})&:=\sum_{k=1}^n\left\{\log\det[\sigma\sigma^\top](X_{t_{k-1}^n},\beta)+\varepsilon^{-2}h_n^{-1}P_{v,k}(\bar{\alpha})^\top [\sigma\sigma^\top]^{-1}(X_{t_{k-1}^n},\beta)P_{v,k}(\bar{\alpha})\right\},
\end{align*}
and $\hat{\beta}_{\varepsilon,n}=\hat{\beta}_{\varepsilon,n}^{(1)}$ is defined as
\begin{align*}
\hat{\beta}_{\varepsilon,n}:=\argmin_{\beta\in\Theta_{\beta}}V_{\varepsilon,n,v}^{(v+1)}(\beta|\hat{\alpha}_{\varepsilon,n}^{(v)}).
\end{align*}
\end{enumerate}

\begin{enumerate}[Step]
\item $v+2$. Set
\begin{align*}
V_{\varepsilon,n,v}^{(v+2)}(\alpha|\bar{\alpha},\bar{\beta})&:=\varepsilon^{-2}h_n^{-1}\sum_{k=1}^n\left(P_{1,k}(\alpha)-Q_{v,k}(\bar{\alpha})\right)^\top[\sigma\sigma^\top]^{-1}(X_{t_{k-1}^n},\bar{\beta}) \left(P_{1,k}(\alpha)-Q_{v,k}(\bar{\alpha})\right),
\end{align*}
and $\hat{\alpha}_{\varepsilon,n}=\hat{\alpha}_{\varepsilon,n}^{(v+1)}$ is defined as
\begin{align*}
\hat{\alpha}_{\varepsilon,n}:=\argmin_{\alpha\in\Theta_{\alpha}}V_{\varepsilon,n,v}^{(v+2)}(\alpha|\hat{\alpha}_{\varepsilon,n}^{(v)},\hat{\beta}_{\varepsilon,n}).
\end{align*}
\end{enumerate}

In Step $v$, we have the following asymptotic properties.
\begin{lemma}\label{lem:2}
Assume {\bf{[A1]}}-{\bf{[A3]}}, {\bf{[A4]}}-(i) and {\bf{[B]}}. Then it holds that
\begin{align*}
\varepsilon^{-1}(\hat{\alpha}^{(v)}_{\varepsilon,n}-\alpha_0)=O_P(1).
\end{align*}
In particular, 
\begin{align*}
\varepsilon^{-1}(\hat{\alpha}^{(v)}_{\varepsilon,n}-\alpha_0)\overset{d}{\to}
N_p(0,J_b(\alpha_0)^{-1}K_b(\theta_0)J_b(\alpha_0)^{-1})
\end{align*}
as $\varepsilon\to0$ and $n\to\infty$.
\end{lemma}

The main result for the Type II method is as follows.
\begin{theorem}\label{thm:2}
Assume {\bf{[A1]}}-{\bf{[A4]}} and {\bf{[B]}}. Then it follows that
\begin{align*}
\hat{\theta}_{\varepsilon,n}\overset{P}{\to}\theta_0.
\end{align*}
Moreover,  
\begin{align*}
\begin{pmatrix}
\varepsilon^{-1}(\hat{\alpha}_{\varepsilon,n}-\alpha_0)\\
\sqrt{n}(\hat{\beta}_{\varepsilon,n}-\beta_0)
\end{pmatrix}
\overset{d}{\to}
N_{p+q}(0,I(\theta_0)^{-1})
\end{align*}
as $\varepsilon\to0$ and $n\to\infty$.
\end{theorem}

%%%%%%%%%%%%%%%%%%%%%%%%%%%%%
\subsection{Remarks for adaptive estimation}\label{sec:supp}
\

The Type I and Type II estimators work well for any $\rho>0$.  However, if we consider the case $\rho<\frac{1}{2}$, that is,  the case $\varepsilon^{-1}/\sqrt{n}\to0$, it is natural that we first estimate diffusion parameter instead of drift parameter. Therefore, we introduce another adaptive method 
which estimates diffusion parameter first: 

\begin{enumerate}[Step 1.]

\item Set 
\begin{align*}
W_{\varepsilon,n}^{(1)}(\beta)&:=\sum_{k=1}^n\left\{\log\det[\sigma\sigma^\top](X_{t_{k-1}^n},\beta)+\varepsilon^{-2}h_n^{-1}(X_{t_{k}^n}-X_{t_{k-1}^n})^\top [\sigma\sigma^\top]^{-1}(X_{t_{k-1}^n},\beta)(X_{t_{k}^n}-X_{t_{k-1}^n})\right\},
\end{align*}
and $\check{\beta}_{\varepsilon,n}$ is defined as
\begin{align*}
\check{\beta}_{\varepsilon,n}:=\argmin_{\beta\in\Theta_{\beta}}W_{\varepsilon,n}^{(1)}(\beta).
\end{align*}

\item Set
\begin{align*}
W_{\varepsilon,n}^{(2)}(\alpha|\bar{\beta})&:=\varepsilon^{-2}h_n^{-1}\sum_{k=1}^nP_{1,k}(\alpha)^\top[\sigma\sigma^\top]^{-1}(X_{t_{k-1}^n},\bar{\beta}) P_{1,k}(\alpha),
\end{align*}
and $\check{\alpha}_{\varepsilon,n}$ is defined as
\begin{align*}
\check{\alpha}_{\varepsilon,n}:=\argmin_{\alpha\in\Theta_{\alpha}}W_{\varepsilon,n}^{(2)}(\alpha|\check{\beta}_{\varepsilon,n}).
\end{align*}
\end{enumerate}

These adaptive estimators also have asymptotic normality with the same asymptotic variance in Theorems \ref{thm:1} and \ref{thm:2}.

\begin{proposition}\label{prop:1-small}
Assume {\bf{[A1]}}-{\bf{[A4]}} and {\bf{[B]}} with $\rho<\frac{1}{2}$. Then it follows that
\begin{align*}
\check{\theta}_{\varepsilon,n}:=(\check{\alpha}_{\varepsilon,n},\check{\beta}_{\varepsilon,n})\overset{P}{\to}\theta_0.
\end{align*}
Moreover,  
\begin{align*}
\begin{pmatrix}
\varepsilon^{-1}(\check{\alpha}_{\varepsilon,n}-\alpha_0)\\
\sqrt{n}(\check{\beta}_{\varepsilon,n}-\beta_0)
\end{pmatrix}
\overset{d}{\to}
N_{p+q}(0,I(\theta_0)^{-1})
\end{align*}
as $\varepsilon\to0$ and $n\to\infty$.
\end{proposition}
This proof is similar to that of Theorem \ref{thm:1}. We omit the proof.

Next, we consider the following special cases and 
show that we can also apply our adaptive methods. 
\begin{enumerate}[(i)]
\item Case of $\sigma(x,\beta)=\sigma(x)$ (see, for example, Model 3 in Section \ref{sec:model3} below) : 
We fix some algorithm and estimate only the drift parameter $\alpha$. 
\begin{enumerate}[Type I :]
\item Calculate only the estimator in Step 3, and output $\tilde{\alpha}_{\varepsilon,n}$.
\item Compute the estimators in Step 1 to $v$, and skip Step $v+1$. Next, run Step $v+2$. 
Finally, output $\hat{\alpha}_{\varepsilon,n}$.
\end{enumerate}
From these algorithms, estimators 
$\tilde{\alpha}_{\varepsilon,n}$ % $= \tilde{\alpha}_{\varepsilon,n}^{(1)}$
and 
$\hat{\alpha}_{\varepsilon,n}$ % $= \hat{\alpha}_{\varepsilon,n}^{(v)}$
are asymptotically efficient, respectively.

\item Case of $\alpha=\beta$ (see, for example, Model 2 in Section \ref{sec:model2} below) : 
We estimate only drift or diffusion parameter. This choice depends on the speed of convergence for both parameters.
\begin{enumerate}[(a)]
\item Case of $\varepsilon^{-1}/\sqrt{n}\to\infty$:  We estimate only drift parameter.
\begin{enumerate}[Type I :]
\item Calculate the estimator in Step 1, and skip Step 2.
Next, run Step 3 with $\tilde{\beta}_{\varepsilon,n}=\tilde{\alpha}_{\varepsilon,n}^{(1)}$. Finally, output $\tilde{\alpha}_{\varepsilon,n}$.
\item Compute the estimators in Step 1 to $v$, and skip Step $v+1$. 
Next, calculate Step $v+2$ with $\hat{\beta}_{\varepsilon,n}=\hat{\alpha}_{\varepsilon,n}^{(v)}$. Lastly, output $\hat{\alpha}_{\varepsilon,n}$.
\end{enumerate}
By these methods, estimators $\tilde{\alpha}_{\varepsilon,n}$ and $\hat{\alpha}_{\varepsilon,n}$ are asymptotically efficient, respectively.

\item Case of $\varepsilon^{-1}/\sqrt{n}\to0$:  We utilize the adaptive method for $\rho<\frac{1}{2}$ and estimate only diffusion parameter: 
Calculate only Step 1 in Section \ref{sec:supp}, and output $\check{\beta}_{\varepsilon,n}$.

\item Case of $\varepsilon^{-1}/\sqrt{n}\to M\neq\{0, \infty\}$: In this case, we need to consider the contrast functions both drift and diffusion simultaneously. Therefore, proposed adaptive estimators cannot be used. For estimating method in this case, see Uchida \cite{Uchida_2003}.

\end{enumerate}
\end{enumerate}

%%%%%%%%%%%%%% section4 %%%%%%%%%%%%%%%%%%%%%%
\section{Adaptive tests}\label{sec4}
In this section, we consider the following set of parametric tests:

\begin{equation}\label{ada.stest}
\begin{cases}
H_0^{(1)}:\ \alpha_1=\cdots=\alpha_{r}=0,\\
H_1^{(1)}:\ \text{not}\ H_0^{(1)},
\end{cases}
\quad
\begin{cases}
H_0^{(2)}:\ \beta_1=\cdots=\beta_{s}=0,\\
H_1^{(2)}:\ \text{not}\ H_0^{(2)}.
\end{cases}
\end{equation}
where $1\le r \le p$,  $1\le s \le q$. 
These set of tests give more information about parameters than the joint test and can provide the following four interpretations:
(i) $H_0^{(1)}$ is rejected and $H_0^{(2)}$ is rejected; \
(ii) $H_0^{(1)}$ is rejected and $H_0^{(2)}$ is not rejected; \
(iii) $H_0^{(1)}$ is not rejected and $H_0^{(2)}$ is rejected; \ 
(iv) $H_0^{(1)}$ is not rejected and $H_0^{(2)}$ is not rejected. 
In order to construct test statistics for above tests, we define restricted parameter spaces 
$\Theta_{\alpha}^{H_0}$ and $\Theta_{\beta}^{H_0}$ 
as $\Theta_{\alpha}^{H_0}:=\{\alpha\in\Theta_\alpha\ |\  \alpha\  \text{satisfies}\  H_0^{(1)} \}$, 
$\Theta_{\beta}^{H_0}:=\{\beta\in\Theta_\beta\ |\  \beta\  \text{satisfies}\ H_0^{(2)} \}$ 
and $\Theta^{H_0}:=\Theta_{\alpha}^{H_0}\times\Theta_{\beta}^{H_0}$,
respectively.
Under {\bf{[A4]}}, we set 
\begin{align*}
J_b^{-1}(\alpha_0)&=\begin{pmatrix}
\left(J_{b,1}^{-1,(i,j)}(\beta_0)\right)_{1\le i,j\le r}&\left(J_{b,2}^{-1,(i,j)}(\beta_0)\right)_{1\le i\le r,\  r+1\le j\le p}\\
\left(J_{b,2}^{-1,(i,j)}(\beta_0)\right)_{1\le i\le r,\  r+1\le j\le p}^\top&\left(J_{b,3}^{-1,(i,j)}(\beta_0)\right)_{r+1\le i,j\le p}
\end{pmatrix},\\
I_\sigma^{-1}(\beta_0)&=\begin{pmatrix}
\left(I_{\sigma,1}^{-1,(i,j)}(\beta_0)\right)_{1\le i,j\le s}&\left(I_{\sigma,2}^{-1,(i,j)}(\beta_0)\right)_{1\le i\le s,\  s+1\le j\le q}\\
\left(I_{\sigma,2}^{-1,(i,j)}(\beta_0)\right)_{1\le i\le s,\  s+1\le j\le q}^\top&\left(I_{\sigma,3}^{-1,(i,j)}(\beta_0)\right)_{s+1\le i,j\le q}
\end{pmatrix},
\end{align*}
and define $(p\times p)$-matrix $G_{1}^r(\alpha_0)$ and $(q\times q)$-matrix $G_{2}^s(\beta_0)$ as
\begin{align*}
G_{1}^r(\alpha_0)=\begin{pmatrix}
0&0\\
0&J_{b,3}^{-1}(\beta_0)
\end{pmatrix},\quad 
G_{2}^s(\beta_0)=\begin{pmatrix}
0&0\\
0&I_{\sigma,3}^{-1}(\beta_0)
\end{pmatrix},
\end{align*}
respectively.
Let $Z$ be a $p$-dimensional random vector which has the normal distribution $N_p(0,K_b(\theta_0))$ and $\pi_r$ be the distribution 
of the random variable $Z^\top \left(J_b^{-1}(\alpha_0)-G_1^r(\alpha_0)\right)Z$.
%that the random variable $Z^\top \left(J_b^{-1}(\alpha_0)-G_1^r(\alpha_0)\right)Z$ follows.
Moreover, we define an optimal parameters under null hypothesis $(\alpha_0^{H_0},\beta_0^{H_0})$ as 
\begin{align*}
\alpha_0^{H_0}:=\argmin_{\alpha\in\Theta_{\alpha}^{H_0}} U_1(\alpha;\alpha_0),\quad
\beta_0^{H_0}:=\argmin_{\beta\in\Theta_{\beta}^{H_0}} U_2(\beta;\beta_0)
\end{align*}
where, 
\begin{align}
U_1(\alpha,\alpha_0)&:=\int_{0}^{1}{\left(b(X_s^0,\alpha_0)-b(X_s^0,\alpha)\right)^\top\left(b(X_s^0,\alpha_0)-b(X_s^0,\alpha)\right)}ds,\label{U1-small}\\
U_2(\beta,\beta_0)&:=\int_{0}^{1}\left\{\log\det\left([\sigma\sigma^\top](X_s^0,\beta)[\sigma\sigma^\top]^{-1}(X_s^0,\beta_0)\right)+\mathrm{tr}\left([\sigma\sigma^\top]^{-1}(X_s^0,\beta)[\sigma\sigma^\top](X_s^0,\beta_0)\right)-d\right\}ds.\label{U2:def-small}
\end{align}
It is remarked that $\alpha_0^{H_0}\neq \alpha_0$ and $\beta_0^{H_0}\neq \beta_0$ under alternatives by the definition of $(\alpha_0^{H_0},\beta_0^{H_0})$.
Using restricted optimal parameters $(\alpha_0^{H_0},\beta_0^{H_0})$, we make the following condition.
\begin{enumerate}
\item[{\bf[C]}] 
\begin{enumerate}[(i)]
\item For any $\delta_1>0$,\quad
$\displaystyle{
\inf_{\left\{\alpha\in{\Theta}_\alpha^{H_0};|\alpha-\alpha_0^{H_0}|\geq\delta_1\right\}}\left({U}_1(\alpha,\alpha_0)-{U}_1(\alpha_0^{H_0},\alpha_0)\right)>0.
}
$
\item For any $\delta_2>0$,\quad
$\displaystyle{
\inf_{\left\{\beta\in{\Theta}_\beta^{H_0};|\beta-\beta_0^{H_0}|\geq\delta_2\right\}}\left({U}_2(\beta,\beta_0)-{U}_2(\beta_0^{H_0},\beta_0)\right)>0.
}
$
\end{enumerate}
\end{enumerate}
In this paper, we propose two kinds of likelihood ratio type test statistics (Type I, Type II) for each test using adaptive estimators.

%%%%%%%%%%%%%%%%%%%%%%%%%%%%%
\subsection{Type I tests}\label{sec:type1-test}
\ 
First, we construct the test statistics utilizing the Type I estimator. 
The restricted Type I estimators $(\tilde{\alpha}_{\varepsilon,n}^{(1),H_0},\tilde{\beta}_{\varepsilon,n}^{H_0})$ 
is defined as
\begin{align*}
\tilde{\alpha}_{\varepsilon,n}^{(1),H_0}:=\argmin_{\alpha\in\Theta_{\alpha}^{H_0}}U_{\varepsilon,n,v}^{(1)}(\alpha),\quad
\tilde{\beta}_{\varepsilon,n}^{H_0}:=\argmin_{\beta\in\Theta_{\beta}^{H_0}}U_{\varepsilon,n,v}^{(2)}(\beta|\tilde{\alpha}_{\varepsilon,n}^{(1)}),
\end{align*}
and likelihood ratio type test statistics $(\tilde{\Lambda}_{n}^{(1)},\tilde{\Lambda}_{n}^{(2)})$ as
\begin{align*}
\tilde{\Lambda}_{n}^{(1)}&:=U_{\varepsilon,n,v}^{(1)}\left(\tilde{\alpha}_{\varepsilon,n}^{(1),H_0}\right)-U_{\varepsilon,n,v}^{(1)}\left(\tilde{\alpha}_{\varepsilon,n}^{(1)}\right),\\
\tilde{\Lambda}_{n}^{(2)}&:=U_{\varepsilon,n,v}^{(2)}\left(\tilde{\beta}_{\varepsilon,n}^{H_0}|\tilde{\alpha}_{\varepsilon,n}^{(1)}\right)-U_{\varepsilon,n,v}^{(2)}\left(\tilde{\beta}_{\varepsilon,n}|\tilde{\alpha}_{\varepsilon,n}^{(1)}\right).
\end{align*}
The following theorem gives asymptotic distributions of these test statistics under null hypothesis.
From this theorem, we can calculate rejection regions for each test and conduct tests \eqref{ada.stest}.
\begin{theorem}\label{thm:3-small}
Assume {\bf{[A1]}}-{\bf{[A4]}} and {\bf{[B]}}. 
Then it follows that
%under null hypothesis that
\begin{align*}
\tilde{\Lambda}_{n}^{(1)}\overset{d}{\to}\pi_r \quad(\mathrm{under\  } H_0^{(1)}),\quad
\tilde{\Lambda}_{n}^{(2)}\overset{d}{\to}\chi^2_s \quad(\mathrm{under\  } H_0^{(2)})
\end{align*}
as $\varepsilon\to0$ and $n\to\infty$.
\end{theorem}

The followings ensure consistency of Type I adaptive tests. For a distribution $\nu$ and $\delta\in(0,1)$, we denote $\nu(\delta)$ as the upper $\delta$ point of $\nu$.
\begin{theorem}\label{thm:4-small}
Assume {\bf{[A1]}}-{\bf{[A4]}}, {\bf{[B]}} and {\bf{[C]}}. 
Then it follows 
%under alternatives 
that for any $\delta\in(0,1),$
\begin{align*}
P(\tilde{\Lambda}_{n}^{(1)}\ge\pi_r(\delta))\to1 \quad(\mathrm{under\  } H_1^{(1)}),\quad
P(\tilde{\Lambda}_{n}^{(2)}\ge\chi^2_s(\delta))\to1 \quad(\mathrm{under\  } H_1^{(2)})
\end{align*}
as $\varepsilon\to0$ and $n\to\infty$.
\end{theorem}

%%%%%%%%%%%%%%%%%%%%%%%%%%%%%%
\subsection{Type II tests}\label{sec:type2-test}
\ 

Next, we consider test statistics utilizing the Type II estimator. 
%In similar to Type I tests,  
In a similar way to Type I tests,  
we define restricted Type II estimators $(\hat{\alpha}_{\varepsilon,n}^{(v),H_0},\hat{\beta}_{\varepsilon,n}^{H_0})$ as
\begin{align*}
\hat{\alpha}_{\varepsilon,n}^{(v),H_0}:=\argmin_{\alpha\in\Theta_{\alpha}^{H_0}}V_{\varepsilon,n,v}^{(v)}(\alpha|\hat{\alpha}_{\varepsilon,n}^{(v-1)}),\quad
\hat{\beta}_{\varepsilon,n}^{H_0}:=\argmin_{\beta\in\Theta_{\beta}^{H_0}}V_{\varepsilon,n,v}^{(v+1)}(\beta|\hat{\alpha}_{\varepsilon,n}^{(v)}),
\end{align*}
and likelihood ratio type test statistics $(\hat{\Lambda}_{n}^{(1)},\hat{\Lambda}_{n}^{(2)})$ as
\begin{align*}
\hat{\Lambda}_{n}^{(1)}&:=V_{\varepsilon,n,v}^{(v)}\left(\hat{\alpha}_{\varepsilon,n}^{(1),H_0}|\hat{\alpha}_{\varepsilon,n}^{(v-1)}\right)-V_{\varepsilon,n,v}^{(v)}\left(\hat{\alpha}_{\varepsilon,n}^{(v)}|\hat{\alpha}_{\varepsilon,n}^{(v-1)}\right),\\
\hat{\Lambda}_{n}^{(2)}&:=V_{\varepsilon,n,v}^{(v+1)}\left(\hat{\beta}_{\varepsilon,n}^{H_0}|\hat{\alpha}_{\varepsilon,n}^{(v)}\right)-V_{\varepsilon,n,v}^{(v+1)}\left(\hat{\beta}_{\varepsilon,n}|\hat{\alpha}_{\varepsilon,n}^{(v)}\right).
\end{align*}
Asymptotic distributions of these test statistics under null hypothesis and consistency of tests can be proved in an analogous manner to theorem \ref{thm:3-small} and
\ref{thm:4-small}. We omit the proofs of the following theorems. 
\begin{theorem}\label{thm:5-small}
Assume {\bf{[A1]}}-{\bf{[A4]}} and {\bf{[B]}}. 
Then it follows 
%under null hypothesis 
that
\begin{align*}
\hat{\Lambda}_{n}^{(1)}\overset{d}{\to}\pi_r \quad(\mathrm{under\  } H_0^{(1)}),\quad
\hat{\Lambda}_{n}^{(2)}\overset{d}{\to}\chi^2_s \quad(\mathrm{under\  } H_0^{(2)})
\end{align*}
as $\varepsilon\to0$ and $n\to\infty$.
\end{theorem}

\begin{theorem}\label{thm:6-small}
Assume {\bf{[A1]}}-{\bf{[A4]}}, {\bf{[B]}} and {\bf{[C]}}. 
Then it follows 
%under alternatives 
that for any $\delta\in(0,1),$
\begin{align*}
P(\hat{\Lambda}_{n}^{(1)}\ge\pi_r(\delta))\to1 \quad(\mathrm{under\  } H_1^{(1)}),\quad
P(\hat{\Lambda}_{n}^{(2)}\ge\chi^2_s(\delta))\to1 \quad(\mathrm{under\  } H_1^{(2)})
\end{align*}
as $\varepsilon\to0$ and $n\to\infty$.
\end{theorem}

%%%%%%%%%%%%%%%%%%%%%%%%%%%%%%%%%
\subsection{Remarks for adaptive tests}\label{sec:supp-test}
\

We introduce the adaptive test statistics
for $\rho<\frac{1}{2}$. Using adaptive estimators proposed in Section \ref{sec:supp} $(\check{\alpha}_{\varepsilon,n},\check{\beta}_{\varepsilon,n})$, we define restricted estimators $(\check{\alpha}_{\varepsilon,n}^{H_0},\check{\beta}_{\varepsilon,n}^{H_0})$ as
\begin{align*}
\check{\alpha}_{\varepsilon,n}^{H_0}:=\argmin_{\alpha\in\Theta_{\alpha}^{H_0}}W_{\varepsilon,n}^{(2)}(\alpha|\check{\beta}_{\varepsilon,n}),\quad
\check{\beta}_{\varepsilon,n}^{H_0}:=\argmin_{\beta\in\Theta_{\beta}^{H_0}}W_{\varepsilon,n}^{(1)}(\beta),
\end{align*}
and likelihood ratio type test statistics $\check{\Lambda}_n^{(\alpha)}$ and $\check{\Lambda}_n^{(\beta)}$ are defined as
\begin{align*}
\check{\Lambda}_{n}^{(\alpha)}&:=W_{\varepsilon,n}^{(2)}\left(\check{\alpha}_{\varepsilon,n}^{H_0}|\check{\beta}_{\varepsilon,n}\right)-W_{\varepsilon,n}^{(2)}\left(\check{\alpha}_{\varepsilon,n}|\check{\beta}_{\varepsilon,n}\right),\\
\check{\Lambda}_{n}^{(\beta)}&:=W_{\varepsilon,n}^{(1)}\left(\check{\beta}_{\varepsilon,n}^{H_0}\right)-W_{\varepsilon,n}^{(1)}\left(\check{\beta}_{\varepsilon,n}\right).
\end{align*}
In this case, we first test 
the diffusion parameter $\beta$. 
Asymptotic distributions of these test statistics under null hypothesis and consistency are 
shown. Since $\check{\alpha}_{\varepsilon,n}$ is asymptotically efficient, 
the test statistic $\check{\Lambda}_{n}^{(\alpha)}$ 
converges to the chi-squared distribution under null hypothesis.
\begin{proposition}\label{prop:2-small}
Assume {\bf{[A1]}}-{\bf{[A4]}} and {\bf{[B]}} with $\rho<\frac{1}{2}$. 
Then it follows 
%under null hypothesis 
that
\begin{align*}
\hat{\Lambda}_{n}^{(\alpha)}\overset{d}{\to}\chi^2_r \quad(\mathrm{under\  } H_0^{(1)}),\quad
\hat{\Lambda}_{n}^{(\beta)}\overset{d}{\to}\chi^2_s \quad(\mathrm{under\  } H_0^{(2)})
\end{align*}
as $\varepsilon\to0$ and $n\to\infty$.
\end{proposition}

\begin{proposition}\label{prop:3-small}
Assume {\bf{[A1]}}-{\bf{[A4]}}, {\bf{[B]}} with $\rho<\frac{1}{2}$ and {\bf{[C]}}. 
Then it follows 
%under alternatives 
that for any $\delta\in(0,1),$
\begin{align*}
P(\hat{\Lambda}_{n}^{(\alpha)}\ge\chi^2_r(\delta))\to1 \quad(\mathrm{under\  } H_1^{(1)}),\quad
P(\hat{\Lambda}_{n}^{(\beta)}\ge\chi^2_s(\delta))\to1 \quad(\mathrm{under\  } H_1^{(2)})
\end{align*}
as $\varepsilon\to0$ and $n\to\infty$.
\end{proposition}
We omit the proofs of these propositions. 

In the testing method for $\alpha$ with the chi-squared distribution in general $\rho>0$, 
we can use the asymptotic efficient estimator $\tilde{\alpha}_{\varepsilon,n}$ 
in Type I or $\hat{\alpha}_{\varepsilon,n}$ in Type II. 
If the restrict estimators $\tilde{\alpha}_{\varepsilon,n}^{H_0}$ and $\hat{\alpha}_{\varepsilon,n}^{H_0}$ 
are defined as
\begin{align*}
\tilde{\alpha}_{\varepsilon,n}^{H_0}:=\argmin_{\alpha\in\Theta_{\alpha}^{H_0}}U_{\varepsilon,n,v}^{(3)}(\alpha|\tilde{\beta}_{\varepsilon,n}),\quad
\hat{\alpha}_{\varepsilon,n}^{H_0}:=\argmin_{\alpha\in\Theta_{\alpha}^{H_0}}V_{\varepsilon,n,v}^{(v+2)}(\alpha|\hat{\alpha}_{\varepsilon,n}^{(v)},\hat{\beta}_{\varepsilon,n}).
\end{align*}
and the likelihood ratio type test statistics $\tilde{\Lambda}_n^{(\alpha)}$ and $\hat{\Lambda}_n^{(\alpha)}$ 
are defined as
\begin{align*}
\tilde{\Lambda}_{n}^{(\alpha)}&:=U_{\varepsilon,n,v}^{(3)}\left(\tilde{\alpha}_{\varepsilon,n}^{H_0}|\tilde{\beta}_{\varepsilon,n}\right)-U_{\varepsilon,n,v}^{(3)}\left(\tilde{\alpha}_{\varepsilon,n}|\tilde{\beta}_{\varepsilon,n}\right),\\
\hat{\Lambda}_{n}^{(\alpha)}&:=V_{\varepsilon,n,v}^{(v+2)}\left(\hat{\alpha}_{\varepsilon,n}^{H_0}|\hat{\alpha}_{\varepsilon,n}^{(v)},\hat{\beta}_{\varepsilon,n}\right)-V_{\varepsilon,n,v}^{(v+2)}\left(\hat{\alpha}_{\varepsilon,n}|\hat{\alpha}_{\varepsilon,n}^{(v)},\hat{\beta}_{\varepsilon,n}\right),
\end{align*}
then $\tilde{\Lambda}_n^{(\alpha)}$ and $\hat{\Lambda}_n^{(\alpha)}$ converge to the chi-squared distribution 
with $r$ degree of freedom under null hypothesis.
In the case of $\alpha=\beta$ except for (c) in Section \ref{sec:supp}, we can also conduct the chi-squared test 
for drift or diffusion parameter using the estimators introduced in Section \ref{sec:supp}-(a) and (b).

%%%%%%%%%%%%%% section5 %%%%%%%%%%%%%%%%%%%%%%
\section{Examples and simulations}\label{sec5}
\subsection{Model 1 (Case of estimating both drift and diffusion parameters)}\label{sec:model1}
\ 

First, we examine the asymptotic performance of Theorems \ref{thm:1} and \ref{thm:2}.
Consider the following two-dimensional model:
\begin{equation}\label{model1}
\begin{cases}
dX_t = 
\begin{pmatrix}
-\alpha_1 X_{t,1}+2\cos(1+\alpha_2 X_{t,2})\\
2\sin(1+\alpha_3 X_{t,1})-\alpha_4 X_{t,2}
\end{pmatrix}
dt
+\varepsilon
\begin{pmatrix}
{\beta_1}(1+X_{t,1}^2)^{-1}&-0.1\\
0.1&{\beta_2}(1+X_{t,2}^2)^{-1}
\end{pmatrix}
dW_t,\quad t\in[0,1]\\
X_0 = \begin{pmatrix}
1\\
1
\end{pmatrix},
\end{cases}
\end{equation}
where $\theta=(\alpha_1,\alpha_2,\alpha_3,\alpha_4,\beta_1,\beta_2)$ are unknown parameters. The true parameter values are $\theta_0=(3,6,5,4,1,0.5)$, 
and the parameter space is assumed to be $\Theta=[0.01,50]^6$. 
We estimate these parameters by the joint estimation method in Gloter and S{\o}rensen \cite{Gloter_2009}, 
the Type I method, and the Type II method.  We choose the initial parameters $\theta=\theta_0$ or $\theta=(6,4,6,8,2,2)$ and treat the case of $(\varepsilon,n)=(0.05,100),(0.01,100),(0.01,1000)$. 
For the balance condition, we set $\rho=1$, that is, the approximation degree $v=2$. In the simulation, {\bf{optim()}} is used with the "L-BFGS-B" method in R Language, and 10000 independent sample paths are 
generated.
%estimated.

Tables \ref{model1:table1} and \ref{model1:table2} show the simulation results of parameter estimation with the two choice of the initial parameter values $\theta_{\mathrm{init}}=\theta_0=(3,6,5,4,1,0.5)$ or $\theta_{\mathrm{init}}=(6,4,6,8,2,1)$.
In Table \ref{model1:table1}, we see that all of the estimation methods have good performances 
and there is no notable difference among the three types of methods. In Table \ref{model1:table2}, however,  
the joint estimation method has  considerable biases 
while the Type I and Type II methods have good performances. 
This is because the joint estimation method needs a six-dimensional optimization 
while the Type I and Type II methods need at most four-dimensional optimization. 
Therefore, the adaptive estimation methods are useful for decreasing the dimensions of the parameter optimization. 
Figures \ref{model1:plot1-small} and \ref{model1:plot2-small}  
calculate $(\varepsilon^{-1}(\tilde{\alpha}_{\varepsilon,n}-\alpha_0),\sqrt{n}(\tilde{\beta}_{\varepsilon,n}-\beta_0))$ for 10000 times 
and create histogram, 
empirical distribution, and Q-Q plot for each parameter. 
These figures show that each parameter's estimator has asymptotic normality and its asymptotic variance equals $I(\theta_0)^{-1}$. 

Next, we consider the following adaptive tests: 
\begin{equation}\label{ada.stest-model1}
\begin{cases}
H_0^{(1)}:\ (\alpha_1,\alpha_4)=(3.0,4.0),\\
H_1^{(1)}:\ \text{not}\ H_0^{(1)},
\end{cases}
\quad
\begin{cases}
H_0^{(2)}:\ (\beta_1,\beta_2)=(1.0,0.5),\\
H_1^{(2)}:\ \text{not}\ H_0^{(2)}.
\end{cases}
\end{equation}
These tests derive the four kinds of results as follows:
\begin{enumerate}[{\bf{Case} 1.}]
\item  Neither $\alpha$ nor $\beta$ is rejected;
\item $\alpha$ is not rejected, but $\beta$ is rejected;
\item  $\alpha$ is rejected, but $\beta$ is not rejected;
\item  Both $\alpha$ and $\beta$ are rejected.
\end{enumerate}
We set true parameters $(\alpha_{2}^*,\alpha_{3}^*)=(6.0,5.0)$ and choose true parameters
 $=(\alpha_1^*,\alpha_4^*,\beta_1^*,\beta_2^*)$ from $\{(3.0,4.0,1.0,0.5),\ (3.0,4.0,1.1,0.6),\ (3.1,4.1,1.0,0.5),\ (3.1,4.1,1.1,0.6)\}$, 
which corresponds to the true parameters of Cases 1-4. 
In this adaptive test simulation, we consider the cases of $(\varepsilon,n)=(0.05,100),$ $(0.01,100),(0.01,1000)$ and treat only Type I method. 
The rest of the settings are the same as in the simulation of the estimation above. 
Let the significance level denote $\delta=0.05$ and each test is rejected when the realization of
test statistic $\tilde{\Lambda}_{n}^{(1)}$ or $\tilde{\Lambda}_{n}^{(2)}$ is greater than $\pi_2(0.05)$ or $\chi^2_2(0.05)$, respectively.
The simulation is repeated 10000 times.

Table \ref{model1_table-small} shows the number of counts of Cases 1-4 selected by the tests \eqref{ada.stest-model1}, 
where the true parameters $(\alpha^*_1,\alpha^*_4,\beta^*_1,$ $\beta^*_2)$ correspond to each case. 
%For all case, 
In all cases, 
the adaptive tests can judge the true case most often as $\varepsilon$ decreases and $n$ increases. 
Table \ref{model1_esize-small} shows the empirical sizes and powers 
for each parametric test in \eqref{ada.stest-model1}. 
The bolded letters in the table indicate that the results are similar to the theoretical results, that is, 
the empirical sizes take around the significance level $\delta=0.05$ or 
the empirical powers take 1.0000. 
When the true case is Case 1, in particular, 
Figure \ref{model1:plot3-small} shows the histogram 
and the empirical distribution of the Type I test statistics. 
This Figure implies that each test statistic follows the asymptotic distribution stated 
in Theorem \ref{thm:3-small} under null hypothesis.

\begingroup
\begingroup
\renewcommand{\arraystretch}{1.2}
\begin{table}[tbp]
\begin{center}
\caption{\ Mean (S.D.) of the simulated values with the true initial parameters $\theta_{\mathrm{init}}=\theta_0$.}
\label{model1:table1}
\begin{tabular}{cccccccccc}
\hline
$\varepsilon$&$n$&&Method&$\alpha_1(3)$&$\alpha_2(6)$&$\alpha_3(5)$&$\alpha_4(4)$&$\beta_1(1)$&$\beta_2(0.5)$\\
\hline\hline
\multirow{6}{*}{0.05}&\multirow{6}{*}{100}&&\multirow{2}{*}{Joint}&3.0104&6.0071&5.0005&4.0003&0.9731&0.4863\\
&&&&(0.0862)&(0.0892)&(0.0360)&(0.0689)&(0.0713)&(0.0360)\\
&&&\multirow{2}{*}{Type I}&3.0104&6.0071&5.0005&4.0003&0.9755&0.4868\\
&&&&(0.0862)&(0.0892)&(0.0360)&(0.0689)&(0.0715)&(0.0360)\\
&&&\multirow{2}{*}{Type II}&3.0092&6.0080&5.0004&4.0007&0.9755&0.4868\\
&&&&(0.0862)&(0.0895)&(0.0361)&(0.0690)&(0.0715)&(0.0360)\\
%\hline
%\multirow{6}{*}{0.05}&\multirow{6}{*}{1000}&&\multirow{2}{*}{Joint}&3.0003&6.0000&5.0000&4.0011&0.9965&0.4988\\
%&&&&(0.0836)&(0.0776)&(0.0340)&(0.0663)&(0.0213)&(0.0108)\\
%&&&\multirow{2}{*}{Type1}&3.0003&6.0000&5.0000&4.0011&0.9966&0.4988\\
%&&&&(0.0836)&(0.0776)&(0.0340)&(0.0663)&(0.0213)&(0.0108)\\
%&&&\multirow{2}{*}{Type2}&3.0003&6.0000&5.0000&4.0011&0.9966&0.4988\\
%&&&&(0.0836)&(0.0776)&(0.0340)&(0.0663)&(0.0213)&(0.0108)\\
\hline
\multirow{6}{*}{0.01}&\multirow{6}{*}{100}&&\multirow{2}{*}{Joint}&3.0110&6.0085&5.0008&3.9993&0.9777&0.4877\\
&&&&(0.0172)&(0.0178)&(0.0072)&(0.0138)&(0.0716)&(0.0360)\\
&&&\multirow{2}{*}{Type I}&3.0110&6.0085&5.0008&3.9993&0.9803&0.4881\\
&&&&(0.0172)&(0.0178)&(0.0072)&(0.0138)&(0.0720)&(0.0360)\\
&&&\multirow{2}{*}{Type II}&3.0099&6.0095&5.0006&3.9997&0.9846&0.4884\\
&&&&(0.0172)&(0.0179)&(0.0072)&(0.0138)&(0.0720)&(0.0360)\\\hline
\multirow{6}{*}{0.01}&\multirow{6}{*}{1000}&&\multirow{2}{*}{Joint}&3.0002&6.0015&5.0003&4.0008&0.9976&0.4986\\
&&&&(0.0171)&(0.0176)&(0.0072)&(0.0137)&(0.0226)&(0.0114)\\
&&&\multirow{2}{*}{Type I}&3.0002&6.0015&5.0003&4.0008&0.9978&0.4986\\
&&&&(0.0171)&(0.0176)&(0.0072)&(0.0137)&(0.0227)&(0.0114)\\
&&&\multirow{2}{*}{Type II}&3.0002&6.0015&5.0003&4.0009&0.9978&0.4986\\
&&&&(0.0171)&(0.0176)&(0.0072)&(0.0137)&(0.0227)&(0.0114)\\
\hline
\end{tabular}
\end{center}
\end{table}
\endgroup

\begingroup
\renewcommand{\arraystretch}{1.2}
\begin{table}[bp]
\begin{center}
\caption{\ Mean (S.D.) of the simulated values with the initial parameters $\theta_{\mathrm{init}}=(6,4,6,8,2,1)$.}
\label{model1:table2}
\begin{tabular}{cccccccccc}
\hline
$\varepsilon$&$n$&&Method&$\alpha_1(3)$&$\alpha_2(6)$&$\alpha_3(5)$&$\alpha_4(4)$&$\beta_1(1)$&$\beta_2(0.5)$\\
\hline\hline
\multirow{6}{*}{0.05}&\multirow{6}{*}{100}&&\multirow{2}{*}{Joint}&2.5696&48.172&42.142&3.8555&5.0627&3.8834\\
&&&&(0.1447)&(2.3443)&(2.7301)&(0.4780)&(0.2111)&(0.1693)\\
&&&\multirow{2}{*}{Type I}&3.0104&6.0071&5.0005&4.0003&0.9755&0.4868\\
&&&&(0.0862)&(0.0892)&(0.0360)&(0.0689)&(0.0715)&(0.0360)\\
&&&\multirow{2}{*}{Type II}&3.0070&6.1873&5.0003&4.0002&1.0008&0.4870\\
&&&&(0.0927)&(2.1529)&(0.0361)&(0.0694)&(0.3134)&(0.0361)\\
%\hline
%\multirow{6}{*}{0.05}&\multirow{6}{*}{1000}&&\multirow{2}{*}{Joint}&3.0012&5.9998&5.0004&4.0024&0.9965&0.4989\\
%&&&&(0.0839)&(0.0791)&(0.0343)&(0.0669)&(0.0213)&(0.0108)\\
%&&&\multirow{2}{*}{Type1}&2.9984&6.1955&4.9992&4.0015&1.0033&0.4998\\
%&&&&(0.0918)&(2.1250)&(0.0368)&(0.0681)&(0.0737)&(0.0148)\\
%&&&\multirow{2}{*}{Type2}&2.9987&6.1752&4.9994&4.0014&1.0026&0.4993\\
%&&&&(0.0911)&(2.0640)&(0.0362)&(0.0684)&(0.0691)&(0.0109)\\
\hline
\multirow{6}{*}{0.01}&\multirow{6}{*}{100}&&\multirow{2}{*}{Joint}&2.4918&49.999&49.933&3.9127&24.812&20.631\\
&&&&(0.0148)&(0.0530)&(0.1038)&(0.0424)&(0.1539)&(0.1132)\\
&&&\multirow{2}{*}{Type I}&3.0110&6.0085&5.0008&3.9993&0.9803&0.4881\\
&&&&(0.0172)&(0.0178)&(0.0072)&(0.0138)&(0.0720)&(0.0360)\\
&&&\multirow{2}{*}{Type II}&3.0099&6.0095&5.0006&3.9997&0.9846&0.4884\\
&&&&(0.0172)&(0.0179)&(0.0072)&(0.0138)&(0.0720)&(0.0360)\\
\hline
\multirow{6}{*}{0.01}&\multirow{6}{*}{1000}&&\multirow{2}{*}{Joint}&2.5730&30.681&19.359&3.5535&6.6510&6.4561\\
&&&&(0.0275)&(1.9874)&(3.7729)&(0.5760)&(0.0907)&(0.5269)\\
&&&\multirow{2}{*}{Type I}&3.0002&6.0015&5.0003&4.0008&0.9978&0.4986\\
&&&&(0.0171)&(0.0176)&(0.0072)&(0.0137)&(0.0227)&(0.0114)\\
&&&\multirow{2}{*}{Type II}&3.0002&6.0015&5.0003&4.0009&0.9978&0.4986\\
&&&&(0.0171)&(0.0176)&(0.0072)&(0.0137)&(0.0227)&(0.0114)\\
\hline
\end{tabular}
\end{center}
\end{table}
\endgroup
\endgroup

%\begingroup
%\renewcommand{\arraystretch}{1.2}
%\begin{table}[htbp]
%\begin{center}
%\caption{\ Computation time of estimation for 10000 sample paths}
%\begin{tabular}{c|ccc}
%Method&Joint&Type1&Type2\\
%\hline
%Time (h)&3.06&1.09&1.07
%\end{tabular}
%\end{center}
%\end{table}
%\endgroup

\begin{figure}[htbp]
\centering
\includegraphics[width=0.61\columnwidth]{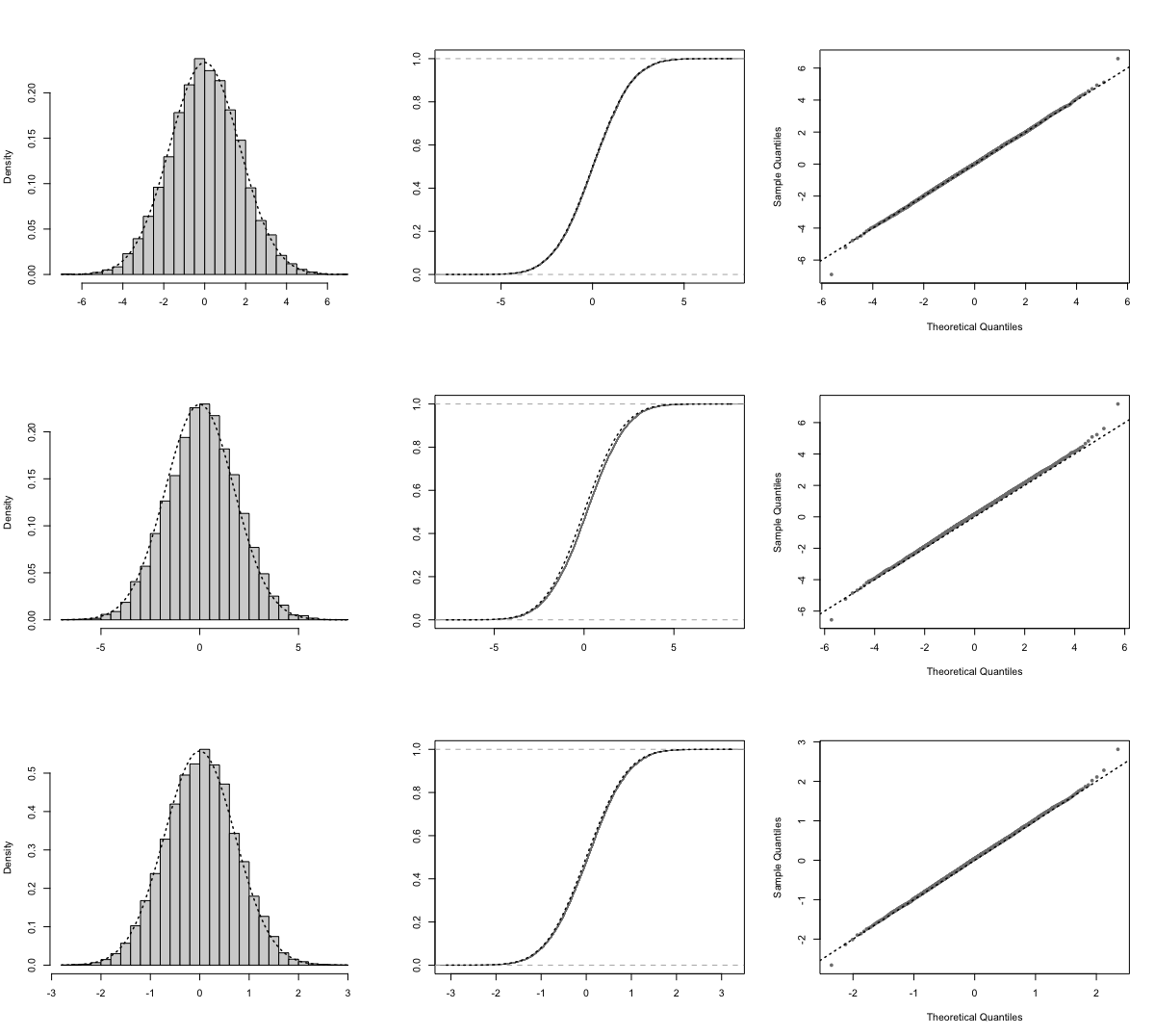}
\caption{Histogram (left), empirical distribution (middle), and Q-Q plot (right) for
%\\ 
estimating $\alpha_1$ (top), $\alpha_2$ (middle), and $\alpha_3$ (bottom). 
The dotted lines are theoretical curves.  
(Type I, true initial parameter, $\varepsilon=0.01$, $n=1000)$}
\label{model1:plot1-small}
\end{figure}
\begin{figure}[htbp]
\centering
\includegraphics[width=0.61\columnwidth]{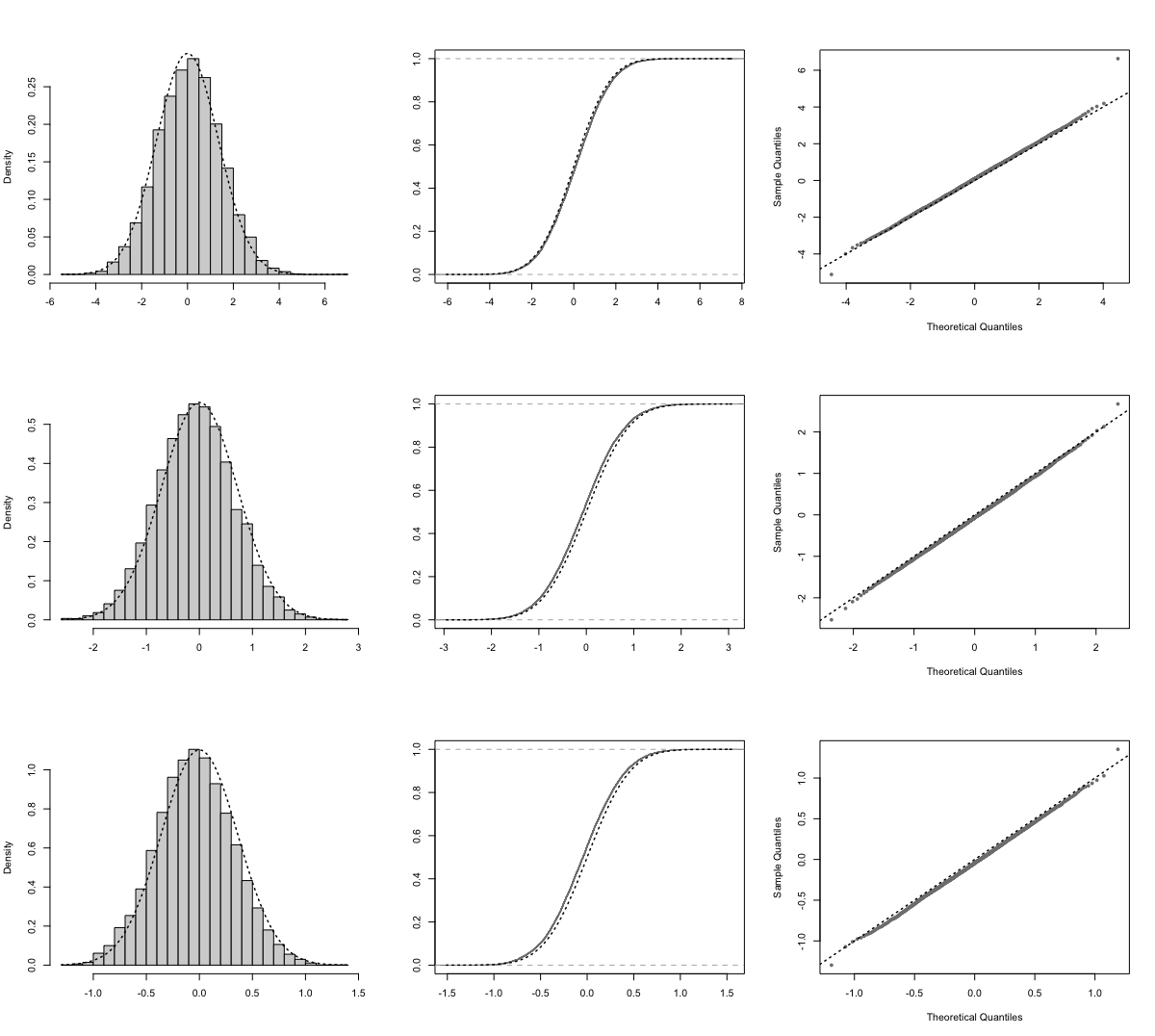}
\caption{Histogram (left), empirical distribution (middle), and Q-Q plot (right) for
%\\ 
estimating $\alpha_4$ (top), $\beta_1$ (middle), and $\beta_2$(bottom).
The dotted lines are theoretical curves.  
 (Type I, true initial parameter,  $\varepsilon=0.01$, $n=1000)$}
 \label{model1:plot2-small}
\end{figure}

\renewcommand{\arraystretch}{1.5}
\begin{table}[htbp]
\begin{center}
\caption{Results of Type I adaptive tests.}
\label{model1_table-small}
%\scalebox{0.88}{
\begin{tabular}{ccc|cccc}
\multirow{2}{*}{{\bf{True Case}}}&\multirow{2}{*}{$\varepsilon$}&\multirow{2}{*}{$n$}&\multicolumn{4}{c}{{\bf{Judgement}}}\\
&&&Case 1&Case 2&Case 3&Case 4\\
\hline\hline
\multirow{3}{*}{Case 1}&0.05&100				&{\bf{8879}}&690&401&30\\
&0.01&100									&{\bf{8808}}&640&520&32\\
&0.01&1000									&{\bf{9052}}&500&422&26\\
\hline
\multirow{3}{*}{Case 2}&0.05&100				&1703&{\bf{7866}}&81&350\\
&0.01&100									&1838&{\bf{7610}}&110&442\\
&0.01&1000									&0&{\bf{9552}}&0&448\\
\hline
\multirow{3}{*}{Case 3}&0.05&100				&{\bf{7075}}&556&2205&164\\
&0.01&100									&0&0&{\bf{9328}}&672\\
&0.01&1000									&0&0&{\bf{9474}}&526\\
\hline
\multirow{3}{*}{Case 4}&0.05&100				&1364&{\bf{6267}}&420&1949\\
&0.01&100									&0&0&1948&{\bf{8052}}\\
&0.01&1000									&0&0&0&{\bf{10000}}\\
\end{tabular}
%}
\end{center}
\end{table}

\renewcommand{\arraystretch}{1.5}
\begin{table}[htbp]
\begin{center}
\caption{Empirical sizes and powers of Type I test statistics.}
\label{model1_esize-small}
%\scalebox{0.88}{
\begin{tabular}{ccc|ccc}
\multirow{2}{*}{{\bf{True Case}}}&\multirow{2}{*}{$\varepsilon$}&\multirow{2}{*}{$n$}&\multicolumn{3}{c}{{\bf{Testing parameter}}}\\
&&&$\alpha$&&$\beta$\\
\hline\hline
\multirow{3}{*}{Case 1}&0.05&100				&{\bf{0.0431}}&&0.0720\\
&0.01&100									&{\bf{0.0552}}&&0.0672\\
&0.01&1000									&{\bf{0.0448}}&&{\bf{0.0526}}\\
\hline
\multirow{3}{*}{Case 2}&0.05&100				&{\bf{0.0431}}&&0.8216\\
&0.01&100									&{\bf{0.0552}}&&0.8052\\
&0.01&1000									&{\bf{0.0448}}&&{\bf{1.0000}}\\
\hline
\multirow{3}{*}{Case 3}&0.05&100				&0.2369&&0.0720\\
&0.01&100									&{\bf{1.0000}}&&0.0672\\
&0.01&1000									&{\bf{1.0000}}&&{\bf{0.0526}}\\
\hline
\multirow{3}{*}{Case 4}&0.05&100				&0.2369&&0.8216\\
&0.01&100									&{\bf{1.0000}}&&0.8052\\
&0.01&1000									&{\bf{1.0000}}&&{\bf{1.0000}}\\
%}
\end{tabular}
\end{center}
\end{table}

\begin{figure}[htbp]
\centering
\includegraphics[width=0.61\columnwidth]{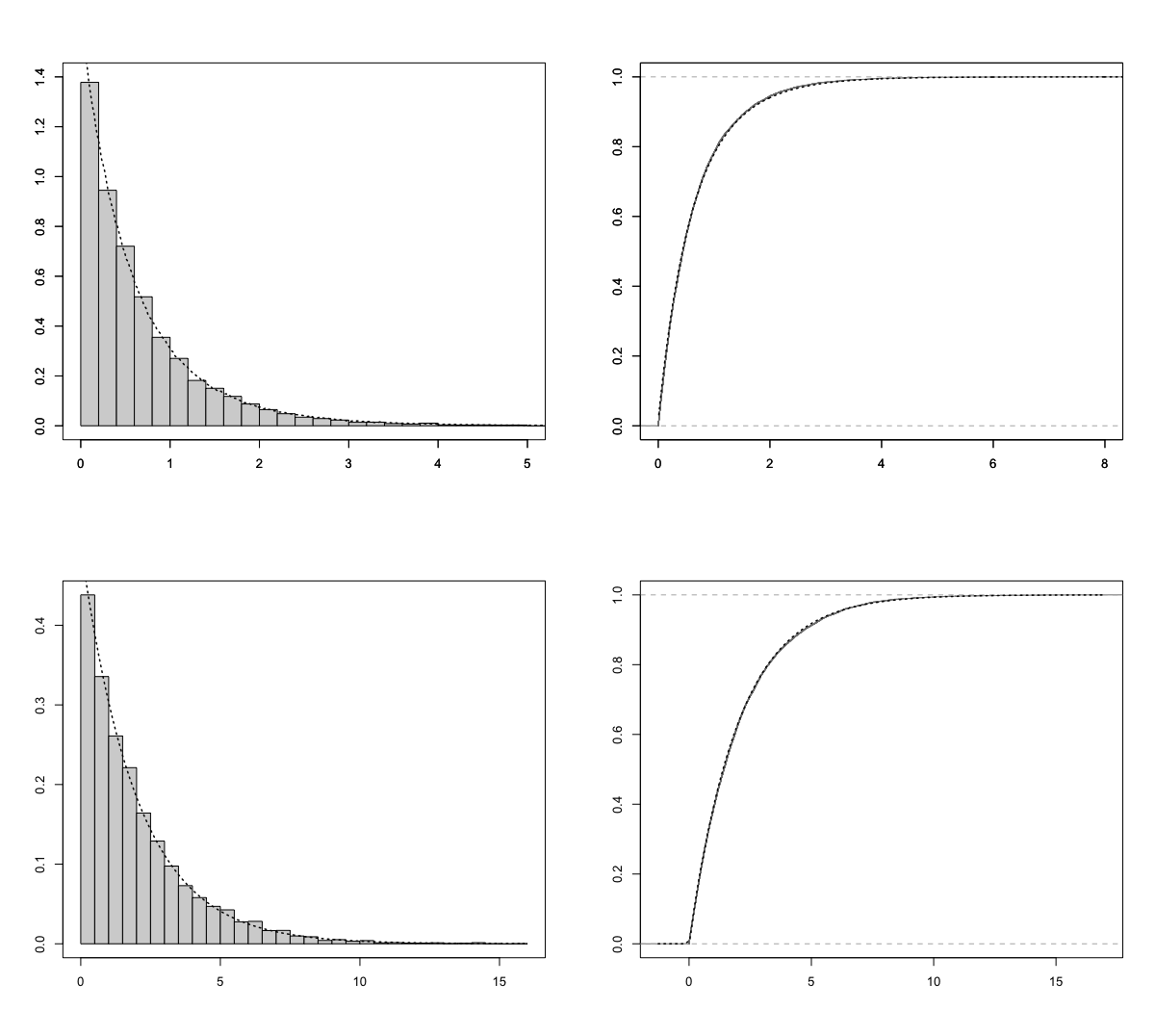}
\caption{Histogram and  empirical distribution for $\Lambda_n^{(1)}$ (top) and $\Lambda_n^{(2)}$ (bottom) under Case 1.
%\\
The dotted lines are theoretical curves of $\pi_2$ (top) and $\chi^2_2$ (bottom).  
  (Type I, true initial parameter, $\varepsilon=0.01$, $n=1000)$}
\label{model1:plot3-small}
\end{figure}

%\newpage
%%%%%%%%%%%%%%%%%%%%%%%%%%%%
\subsection{Model 2 (Case of having the same parameter)}\label{sec:model2}
\ 

Second, we consider the case that the model has the same parameter in the drift and diffusion coefficients,
that is, the case of $\alpha=\beta$ in the SDE \eqref{sde}.
For example, we introduce the following SIR model with the small diffusion coefficient
proposed in Guy et al \cite{Guy_2014, Guy_2015}:
let $X_t=(S_t,I_t),\ t\in[0,T]$, and
\begin{equation}\label{SIR}
\begin{cases}
dS_t = -\beta S_tI_t \ dt + \varepsilon \sqrt{\beta S_tI_t}\ dW_{t,1},  \\
dI_t = \left(\beta S_tI_t-\gamma I_t\right) \ dt 
+ \varepsilon \left(-\sqrt{\beta S_tI_t}\ dW_{t,1}+\sqrt{\gamma I_t}\ dW_{t,2}\right), \\
X_0 = (s_0,i_0)^\top,
\end{cases}
\end{equation}
where both of the drift and diffusion coefficients have $\theta=(\beta,\gamma)$ as unknown parameters 
and $(s_0,i_0)\in(0,1)^2$ is the fixed value. The SIR model describes the simple epidemic spread 
with the three mutually exclusive health states : Susceptible-Infectious-Removed from the infectious chain. 
The parameter $\beta$ implies the transmission rate and  the parameter $\gamma$ implies the recovery rate. Moreover, the basic reproduction number $R_0:=\beta/\gamma$ implies the average number of secondary cases generated by 
one infected person.
%one infectious. 
Therefore, when the value $R_0$ is greater than one, then the epidemic spreads and vice versa.   
We treat the cases of $\theta_0=(1.2,1.0)$ and 
$\theta_0=(0.9, 1.0)$. The first case 
corresponds to 
%the case of 
$R_0>1$, and the second case corresponds to $R_0<1$. In this simulation, we set $\varepsilon=10^{-4}\ (=1/\sqrt{N},\ N=10^8$: population size$)$, $(s_0,i_0)=(0.99999,0.00001)$ and the parameter space $\Theta=[0.01,100]^2$. 
In order to treat the 10 days, monthly and yearly data, we set $(n,T,h_n)=(10,1, 1/10),\ (30,1,1/30),\ (360,12,1/30)$ and determine the balance coefficient $\rho=4$. 
Therefore, the Type I and Type II methods described in Section \ref{sec:supp}-(ii)-(a) are used for estimation. 
It is remarked that this model does not have to consider the initial parameter problem. This is because the model has only two parameters, and the simulation does not fail the parameters optimization. 
Therefore, we set that the initial parameter is the true value. 
The simulation is repeated 10000 times for each setting.

Tables \ref{model2:table1} and \ref{model2:table2} show the simulation results 
with the two types of true parameters settings. In both tables, the sample means are close to the true value and the sample standard deviations are also close to the theoretical standard deviations. Overall, the simulations for both Type I and Type II methods have good behavior.

Next, we conduct the following parametric test: 
\begin{equation}\label{ada.stest-model2}
\begin{cases}
H_0:\ (\beta,\gamma)=(1.2,1.0),\\
H_1:\ \text{not}\ H_0^{(1)}.
\end{cases}
\end{equation}
In this test, we utilize Type II estimator and construct a test statistic 
by using the method proposed in Section \ref{sec:supp-test}. The test statistic is expected to converge in distribution to $\chi_2^2$ under $H_0$. Hence this test is rejected when the realization of test statistic is greater than $\chi_2^2(\delta)$.
We choose 
$(1.2,1.0),$ or $(1.3, 0.9)$ 
as 
true parameters $(\beta^*,\gamma^*)$, 0.05 as the significance level $\delta$ and other simulation settings are the same as the estimation case in this section. Table \ref{model2_table-small} shows the number of counts of 
$H_0$ or $H_1$ selected by the Type II test
 and empirical sizes and powers. 
 %As a whole, 
On the whole, the Type II test
 %this chi-squared test 
 has good performance.

\begingroup
\renewcommand{\arraystretch}{1.8}
\begin{table}[htbp]
\begin{center}
\caption{\ Mean and standard deviation (S.D.) of the estimators in the case of $\theta_0=(1.2,1.0)$.}
\label{model2:table1}
\begin{tabular}{cccccccccccc}
\hline
&&&&&$\beta$&&&$\gamma$&&\\
n&&$T$&&Method&Mean&S.D.&Theoretical S.D.&Mean&S.D.&Theoretical S.D.\\
\hline\hline
\multirow{2}{*}{10}&&\multirow{2}{*}{1}&&Type I&1.199507&0.035437&0.034884&0.999758&0.032008&0.031845\\
&&&&Type II&1.199570&0.035440&0.034884&0.999802&0.032010&0.031845\\
\hline
\multirow{2}{*}{30}&&\multirow{2}{*}{1}&&Type I&1.199555&0.033724&0.033545&1.000083&0.030760&0.030622\\
&&&&Type II&1.199578&0.033726& 0.033545&1.000099&0.030762&0.030622\\
\hline
\multirow{2}{*}{360}&&\multirow{2}{*}{12}&&Type I&1.199709&0.004906&0.004915&1.000173&0.004553&0.004486\\
&&&&Type II&1.199712&0.004906&0.004915&1.000176&0.004554&0.004486\\
\hline
\end{tabular}
\end{center}
\end{table}
\endgroup

\begingroup
\renewcommand{\arraystretch}{1.8}
\begin{table}[htbp]
\begin{center}
\caption{\ Mean and standard deviation (S.D.) of the estimators in the case of $\theta_0=(0.9,1.0)$.}
\label{model2:table2}
\begin{tabular}{cccccccccccc}
\hline
&&&&&$\beta$&&&$\gamma$&&\\
n&&$T$&&Method&Mean&S.D.&Theoretical S.D.&Mean&S.D.&Theoretical S.D.\\
\hline\hline
\multirow{2}{*}{10}&&\multirow{2}{*}{1}&&Type I&0.899630&0.032820&0.032337&0.999791&0.034273&0.034086\\
&&&&Type II&0.899663&0.032819&0.032337&0.999833&0.034279&0.034086\\
\hline
\multirow{2}{*}{30}&&\multirow{2}{*}{1}&&Type I&0.899652&0.031437&0.031253&1.000083&0.033082&0.032944\\
&&&&Type II&0.899659&0.031440&0.031253&1.000092&0.033078&0.032944\\
\hline
\multirow{2}{*}{360}&&\multirow{2}{*}{12}&&Type I&0.899130&0.011284&0.011358&1.000537&0.012076&0.011972\\
&&&&Type II&0.899144&0.011286&0.011358&1.000555&0.012076&0.011972\\
\hline
\end{tabular}
\end{center}
\end{table}
\endgroup

\renewcommand{\arraystretch}{1.6}
\begin{table}[htbp]
\begin{center}
\caption{Results of  the Type II test.}
\label{model2_table-small}
%\scalebox{0.88}{
\begin{tabular}{ccc|ccc}
\multirow{2}{*}{{\bf{True Case}}}&\multirow{2}{*}{$n$}&\multirow{2}{*}{$T$}&\multicolumn{2}{c}{{\bf{Judgement}}}&\multirow{1}{*}{{\bf{Empirical}}}\\
&&&$H_0$&$H_1$&\multirow{1}{*}{{\bf{Size or Power}}}\\
\hline\hline
\multirow{2}{*}{$H_0$}&10&1			&{\bf{9471}}&529&{\bf{0.0529}}\\
\multirow{2}{*}{$(\beta^*,\gamma^*)=(1.2,1.0)$}&30&1									&{\bf{9488}}&512&{\bf{0.0512}}\\
&360&12									&{\bf{9490}}&510&{\bf{0.0510}}\\
\hline
\multirow{2}{*}{$H_1$}&10&1				&202&{\bf{9708}}&{\bf{0.9708}}\\
\multirow{2}{*}{$(\beta^*,\gamma^*)=(1.3,0.9)$}&30&1									&133&{\bf{9867}}&{\bf{0.9867}}\\
&360&12									&0&{\bf{10000}}&{\bf{1.0000}}\\
%}
\end{tabular}
\end{center}
\end{table}

\newpage
%%%%%%%%%%%%%%%%%%%%%%%%%%%%%%
\subsection{Model 3 (Case of estimating only drift parameter)}\label{sec:model3}
\ 

Third, we consider the case of estimating only the drift parameter. 
In particular, we treat the case 
%of the identical diffusion coefficient: 
where the diffusion coefficient is the identity matrix.
\begin{equation}\label{model3}
dX_t = 
\begin{pmatrix}
1-\alpha_1 X_{t,1}-5\sin(\alpha_2 X_{t,2}^2)\\
2-\alpha_3 X_{t,2}-5\sin(\alpha_4 X_{t,3}^2)\\
3-\alpha_5 X_{t,3}-5\sin(\alpha_6 X_{t,1}^2)
\end{pmatrix}
dt
+\varepsilon
dW_t,\quad t\in[0,1],\quad
X_0 = \begin{pmatrix}
1\\
1\\
1
\end{pmatrix},
\end{equation}
where $\alpha=(\alpha_1,\alpha_2,\alpha_3,\alpha_4,\alpha_5,\alpha_6)$ are unknown parameters. Assume that the parameter space is $\Theta_\alpha=[0.01,30]^2$ and the true parameter values are $\alpha_0=(3,7,2,8,1,6)$. We treat the case of $(\varepsilon,n)=(0.01,100)$, $(0.001,100)$, $(0.001,1000)$ and choose $v=3$. We estimate the parameters 
with the Type I method and Type II method. 
Moreover, we estimate the initial parameters with 
the following uniform +  {\bf optim()} method:
\begin{enumerate}[Step 1.]
\item Generate 20000 uniform random numbers $\alpha_{0,m}$ ($m=1,\ldots, 20000$) on $[0.01,30]^6$.
\item Compute 
\begin{align*}
\tilde{\alpha}_{m}^{(1)}&=\argmin_{\alpha \in \Theta_\alpha} 
U_{\varepsilon,n,v}^{(1)}(\alpha),\quad\text{(case of Type I)},\\
\hat{\alpha}_{m}^{(1)}&=\argmin_{\alpha \in \Theta_\alpha}  V_{\varepsilon,n,v}^{(1)}(\alpha),\quad\text{(case of Type II)},\\
\end{align*}
by means of {\bf optim()} in the R language, where 
the uniform random numbers $\alpha_{0,m}$ are used 
as the initial value for optimization.
\item Define the initial estimator $\tilde{\alpha}_{\mathrm{init}}^{(1)}$ or $\hat{\alpha}_{\mathrm{init}}^{(1)}$ as
\begin{align*}
\tilde{\alpha}_{\mathrm{init}}^{(1)}&=\argmin_{\alpha \in \Theta_\alpha}  \left\{U_{\varepsilon,n,v}^{(1)}(\tilde{\alpha}_{1}^{(1)}),U_{\varepsilon,n,v}^{(1)}(\tilde{\alpha}_{2}^{(1)}),\ldots,U_{\varepsilon,n,v}^{(1)}(\tilde{\alpha}_{20000}^{(1)})\right\},\quad\text{(case of Type I)},\\
\hat{\alpha}_{\mathrm{init}}^{(1)}&=\argmin_{\alpha \in \Theta_\alpha}   \left\{V_{\varepsilon,n,v}^{(1)}(\hat{\alpha}_{1}^{(1)}),V_{\varepsilon,n,v}^{(1)}(\hat{\alpha}_{2}^{(1)}),\ldots,V_{\varepsilon,n,v}^{(1)}(\hat{\alpha}_{20000}^{(1)})\right\},\quad\text{(case of Type II)}.
\end{align*}
\end{enumerate}
The simulations are repeated 1000 times for each estimation method.

Table \ref{model3:table1} shows the sample means, standard deviations of the simulated estimator values and the computation times of estimation for one sample path. Regarding accuracy of the estimation, we see that 
both methods performed well for each setting. 
From the viewpoint of the computation time,
however,  the estimators of the Type II method 
are computed more quickly than those of the Type I method. 
This is because that the contrast function 
of the Type II method does not optimize the higher order term 
but put the estimated value $\bar{\alpha}$ into $Q_{l,k}(\bar{\alpha})$. 
As the result, we recommend using the Type II method.

\begingroup
\renewcommand{\arraystretch}{1.2}
\begin{table}[tbp]
\begin{center}
\caption{\ Mean (S.D.) of the simulated values and the computation times.}
\label{model3:table1}
\begin{tabular}{ccccccccccc}
\hline
$\varepsilon$&$n$&&Method&$\alpha_1(3)$&$\alpha_2(7)$&$\alpha_3(2)$&$\alpha_4(8)$&$\alpha_5(1)$&$\alpha_6(6)$&Time (m)\\
\hline\hline
\multirow{4}{*}{0.01}&\multirow{4}{*}{100}&&\multirow{2}{*}{Type I}&3.0032&7.0127&1.9983&7.9964&1.0049&5.9979&\multirow{2}{*}{58}\\
&&&&(0.1036)&(0.3040)&(0.0127)&(0.0047)&(0.0227)&(0.0131)&\\
&&&\multirow{2}{*}{Type II}&2.9891&7.0030&1.9987&7.9958&0.9989&6.0032&\multirow{2}{*}{11}\\
&&&&(0.0570)&(0.0127)&(0.0124)&(0.0068)&(0.0088)&(0.0090)&\\
\hline
\multirow{4}{*}{0.001}&\multirow{4}{*}{100}&&\multirow{2}{*}{Type I}&2.9997&6.9990&1.9986&7.9958&1.0083&6.0003&\multirow{2}{*}{53}\\
&&&&(0.0079)&(0.0028)&(0.0045)&(0.0052)&(0.0240)&(0.0121)&\\
&&&\multirow{2}{*}{Type II}&2.9977&7.0002&1.9987&7.9947&0.9994&6.0022&\multirow{2}{*}{9}\\
&&&&(0.0051)&(0.0024)&(0.0050)&(0.0081)&(0.0010)&(0.0014)&\\
\hline
\multirow{4}{*}{0.001}&\multirow{4}{*}{1000}&&\multirow{2}{*}{Type I}&2.99994&7.00001&1.99993&7.99997&0.99999&6.00005&\multirow{2}{*}{204}\\
&&&&(0.00182)&(0.00034)&(0.00108)&(0.00018)&(0.00082)&(0.00072)&\\
&&&\multirow{2}{*}{Type II}&2.99994&7.00001&1.99993&7.99997&0.99999&6.00005&\multirow{2}{*}{21}\\
&&&&(0.00182)&(0.00034)&(0.00108)&(0.00018)&(0.00082)&(0.00072)&\\
\hline
\end{tabular}
\end{center}
\end{table}
\endgroup

\newpage
%%%%%%%%%%%%%%% section6 %%%%%%%%%%%%%%%%%%%%%%
\section{Proofs}\label{sec6}
In this section, we treat the case of $T=1$ without loss of generality. 
Let the $\sigma$-field $\mathcal{G}_k^n:=\sigma(X_{t_k^n}:s\le t_k^n)$, and for any vector $u$ (matrix $A$), $u^i$ ($A^{i,j}$) denotes the $i$-th element ($(i,j)$-th element) of the vector (matrix). 
For any positive sequence $u_n$, $R:\mathbb{R}\times\mathbb{R}^d\to\mathbb{R}$ denotes a function with a constant $C>0$ such that for all $x\in\mathbb{R}^d$, $|R(u_n,x)|\leq u_nC(1+|x|)^C$, and $R_d$ denotes the $d$-dimensional vector whose element satisfies the definition of the function $R$.
In order to prove the theorems and the lemmas proposed in Section \ref{sec2}, 
we introduce the following restrictive condition of [{\bf{A2}}]:
\begin{enumerate}
\item[{\bf[A2']}] 
For all $(x,\beta)\in\mathbb{R}^d\times \Theta_{\beta}$, the matrix $[\sigma\sigma^\top](x,\beta)$ is positive definite. Moreover, the functions
$\sigma,\  [\sigma\sigma^\top]^{-1}$ (respectively $b$) are bounded and smooth with bounded derivatives of any order on $\mathbb{R}^d\times\Theta_\beta$ (respectively $\mathbb{R}^d\times\Theta_\alpha$).
\end{enumerate}
The following proposition enables us to prove 
the theorems and the lemmas under [{\bf A1}], [{\bf A2'}], [{\bf A3}], [{\bf A4}] and [{\bf B}].

\begin{proposition}\label{prop:1}
In order to show that the conclusions of Theorems \ref{thm:1}-\ref{thm:2} and Lemmas \ref{lem:1}-\ref{lem:2} hold under [{\bf A1}]-[{\bf A4}] and [{\bf B}],
it is enough to prove that they hold under [{\bf A1}], [{\bf A2'}], [{\bf A3}], [{\bf A4}] and [{\bf B}].
\end{proposition}
\begin{proof}[\bf Proof]
This result is obtained in an analogous manner to the proof of Proposition 1 in Gloter and  S\o rensen \cite{Gloter_2009}. We omit the detailed proof.
\end{proof}
%%%%%%%%%%%%%%%%%%%%%%%%%%%%%%%
\begin{proof}[\bf Proof of Lemma \ref{lem:1}]
First, we   show $\tilde{\alpha}_{\varepsilon, n}^{(1)}\overset{P}{\to}\alpha_0$. One deduces that
\begin{align*}
U_{\varepsilon,n,v}^{(1)}(\alpha)&=\varepsilon^{-2}n\sum_{k=1}^{n}{\left(P_{v,k}(\alpha)-P_{v,k}(\alpha_0)+P_{v,k}(\alpha_0)\right)^\top\left(P_{v,k}(\alpha)-P_{v,k}(\alpha_0)+P_{v,k}(\alpha_0)\right)}\\
&=\varepsilon^{-2}n\sum_{k=1}^{n}{\left(P_{v,k}(\alpha)-P_{v,k}(\alpha_0)\right)^\top\left(P_{v,k}(\alpha)-P_{v,k}(\alpha_0)\right)}\\
&\quad+ 2\varepsilon^{-2}n\sum_{k=1}^{n}{\left(P_{v,k}(\alpha)-P_{v,k}(\alpha_0)\right)^\top P_{v,k}(\alpha_0)}\\
&\quad+ \varepsilon^{-2}n\sum_{k=1}^{n}{P_{v,k}(\alpha_0)^\top P_{v,k}(\alpha_0)}.
\end{align*}
Hence, it follows from [{\bf A2'}] and Lemma 4 in Gloter and S\o rensen \cite{Gloter_2009} that

\begin{align}
\varepsilon^2\left(U_{\varepsilon,n,v}^{(1)}(\alpha)-U_{\varepsilon,n,v}^{(1)}(\alpha_0)\right)&=
n\sum_{k=1}^{n}{\left(P_{v,k}(\alpha)-P_{v,k}(\alpha_0)\right)^\top\left(P_{v,k}(\alpha)-P_{v,k}(\alpha_0)\right)}\notag\\
&\quad+ 2n\sum_{k=1}^{n}{\left(P_{v,k}(\alpha)-P_{v,k}(\alpha_0)\right)^\top P_{v,k}(\alpha_0)}\notag\\
&=\frac{1}{n}\sum_{k=1}^{n}{\left(b(X_{t_{k-1}^n},\alpha_0)-b(X_{t_{k-1}^n},\alpha)\right)^\top\left(b(X_{t_{k-1}^n},\alpha_0)-b(X_{t_{k-1}^n},\alpha)\right)}\notag\\
&\quad+ 2\sum_{k=1}^{n}{\left(b(X_{t_{k-1}^n},\alpha_0)-b(X_{t_{k-1}^n},\alpha)+R_d(n^{-1},X_{t_{k-1}^n})\right)^\top P_{v,k}(\alpha_0)}\notag\\
&\quad+\frac{1}{n}\sum_{k=1}^{n}{R(n^{-1},X_{t_{k-1}^n})}\notag\\
&\overset{P}{\to}U_1(\alpha,\alpha_0)\quad\text{uniformly in }\alpha,\label{U1_conv-small}
\end{align}
where $U_1$ is defined by \eqref{U1-small}:
\begin{align*}
U_1(\alpha,\alpha_0)=\int_{0}^{1}{\left(b(X_s^0,\alpha_0)-b(X_s^0,\alpha)\right)^\top\left(b(X_s^0,\alpha_0)-b(X_s^0,\alpha)\right)}ds.
\end{align*}
Let $\omega\in\Omega$ be fixed. It follows from the compactness of $\Theta_\alpha$ that for any sequence $(\varepsilon_m,n_m)$, there exists subsequence $(\varepsilon_m',n_m')$ such that
\begin{align}\label{conv:alpha}
\tilde{\alpha}_{\varepsilon_m',n_m'}^{(1)}(\omega)\to\alpha_\infty\in\Theta_\alpha\quad(\varepsilon_m'\to0,n_m'\to\infty).
\end{align}
From the continuity of $U_1$ and the definition of $\tilde{\alpha}_{\varepsilon,n}^{(1)}$, one deduces that
\begin{align*}
0\ge\varepsilon^2\left(U_{\varepsilon,n,v}^{(1)}(\tilde{\alpha}_{\varepsilon_m',n_m'}^{(1)}(\omega))-U_{\varepsilon,n,v}^{(1)}(\alpha_0)\right)(\omega)\to U_1(\alpha_\infty,\alpha_0)\ge0.
\end{align*}
Hence, we have $\alpha_\infty=\alpha_0$ from the 
identifiability condition [{\bf A3}], and \eqref{conv:alpha} means that $\tilde{\alpha}_{\varepsilon, n}^{(1)}\overset{P}{\to}\alpha_0$.

Second, we  prove $\varepsilon^{-1}(\tilde{\alpha}^{(1)}_{\varepsilon,n}-\alpha_0)=O_P(1)$. It follows from Taylor's theorem that
\begin{align}\label{taylor:U1}
-\varepsilon\partial_{{\alpha}}U_{\varepsilon,n,v}^{(1)}(\alpha_0)=\left(\varepsilon^2\int_{0}^{1}{\partial_{{\alpha}}^2U_{\varepsilon,n,v}^{(1)}(\alpha_0+u(\tilde{\alpha}_{\varepsilon,n}^{(1)}-\alpha_0))}du\right)\varepsilon^{-1}(\tilde{\alpha}^{(1)}_{\varepsilon,n}-\alpha_0).
\end{align}
For $1\le l\le p$ and $1\le l_1,l_2\le p$, we deduce from [{\bf A2'}] and Lemma 4 in Gloter and S\o rensen \cite{Gloter_2009} that
\begin{align}\label{taylor:U1:part1}
-\varepsilon\partial_{{\alpha}}U_{\varepsilon,n,v}^{(1)}(\alpha_0)&=2\varepsilon^{-1}\sum_{k=1}^{n}{\sum_{i=1}^{d}{\left(\partial_{{\alpha}_l}b^i(X_{t_{k-1}^n},\alpha_0)+R(n^{-1},X_{t_{k-1}^n})\right)P_{v,k}^i(\alpha_0)}}
=O_P(1),
\end{align}
and
\begin{align}\label{conv:B1}
\varepsilon^2\partial_{{\alpha}_{l_1l_2}}^2U_{\varepsilon,n,v}^{(1)}(\alpha)&=
-2\sum_{k=1}^{n}{\sum_{i=1}^{d}{\left(\partial_{{\alpha}_{l_1l_2}}^2b^i(X_{t_{k-1}^n},\alpha)-R(n^{-1},X_{t_{k-1}^n})\right)\left(P^i_{v,k}(\alpha)-P^i_{v,k}(\alpha_0)\right)}}\notag\\
&\quad+2n^{-1}\sum_{k=1}^{n}{\sum_{i=1}^{d}{\left(\partial_{{\alpha}_{l_1}}b^i(X_{t_{k-1}^n},\alpha)-R(n^{-1},X_{t_{k-1}^n})\right)\left(\partial_{{\alpha}_{l_2}}b^i(X_{t_{k-1}^n},\alpha)-R(n^{-1},X_{t_{k-1}^n})\right)}}\notag\\
&=\frac{2}{n}\sum_{k=1}^{n}{\sum_{i=1}^{d}{\partial_{{\alpha}_{l_1l_2}}^2b^i(X_{t_{k-1}^n},\alpha)\left(b^i(X_{t_{k-1}^n},\alpha)-b^i(X_{t_{k-1}^n},\alpha_0)\right)}}\notag\\
&\quad+\frac{2}{n}\sum_{k=1}^{n}{\sum_{i=1}^{d}{\partial_{{\alpha}_{l_1}}b^i(X_{t_{k-1}^n},\alpha)\partial_{{\alpha}_{l_2}}b^i(X_{t_{k-1}^n},\alpha)}}+\frac{1}{n}\sum_{k=1}^{n}{\sum_{i=1}^{d}{R(n^{-1},X_{t_{k-1}^n})}}\notag\\
&\overset{P}{\to}2B_1^{l_1,l_2}(\alpha,\alpha_0) \quad\text{uniformly in }\alpha,
\end{align}
where
\begin{align*}
B_1^{l_1,l_2}(\alpha,\alpha_0):&=\int_{0}^{1}{\left(\partial_{{\alpha}_{l_1l_2}}^2b(X_s^0,\alpha)\right)^\top\left(b(X_s^0,\alpha)-b(X_s^0,\alpha_0)\right)}ds+ J_{b}^{l_1,l_2}(\alpha).
\end{align*}
By noting that for all $\lambda\in\mathbb{R}^d\setminus \{0\}$, it follows from [{\bf A4}] that 
\begin{align}\label{ineq:reg}
\eta:=2\lambda^\top B_1(\alpha_0,\alpha_0)\lambda=2\lambda^\top J_b^{l_1,l_2}(\alpha_0)\lambda>0,
\end{align}
and one deduces that
\begin{align}
1&=P\left(2\lambda^\top B_1(\alpha_0,\alpha_0)\lambda>\frac{\eta}{2}\right)\notag\\
&\le P\left(\lambda^\top\left[\int_{0}^{1}{\left\{B_1(\alpha_0,\alpha_0)-B_1(\alpha_0+u(\tilde{\alpha}_{\varepsilon,n}^{(1)}-\alpha_0),\alpha_0)\right\}}du\right]\lambda>\frac{\eta}{12}\right)\label{ineq:Op1-1}\\
&\quad+ P\left(\lambda^\top\left[\int_{0}^{1}{\left\{2B_1(\alpha_0+u(\tilde{\alpha}_{\varepsilon,n}^{(1)}-\alpha_0),\alpha_0)-\varepsilon^2\partial_{{\alpha}}^2U_{\varepsilon,n,v}^{(1)}(\alpha_0+u(\tilde{\alpha}_{\varepsilon,n}^{(1)}-\alpha_0))\right\}}du\right]\lambda>\frac{\eta}{6}\right)\label{ineq:Op1-2}\\
&\quad+P\left(\lambda^\top\left(\varepsilon^2\int_{0}^{1}{\partial_{{\alpha}}^2U_{\varepsilon,n,v}^{(1)}(\alpha_0+u(\tilde{\alpha}_{\varepsilon,n}^{(1)}-\alpha_0))}du\right)\lambda>\frac{\eta}{6}\right)\notag. 
\end{align}
For a sequence $\{r_n\}_{n\in\mathbb{N}}$ such that $r_n\to0\ (n\to\infty)$, we define a set $N_{n,\alpha}$ and an event $A_{n,\alpha}$ as
\begin{align}\label{event:alpha}
N_{n,\alpha}:=\{\alpha\in\Theta_\alpha|\ |\alpha-\alpha_0|\le r_n\},\quad A_{n,\alpha}:=\left\{\tilde{\alpha}_{\varepsilon,n}^{(1)}\in N_{n,\alpha}\right\}.
\end{align}
We then obtain from $\tilde{\alpha}_{\varepsilon, n}^{(1)}\overset{P}{\to}\alpha_0$ that $P(A_{n,\alpha})\to1\ (\varepsilon\to0,n\to\infty)$. Therefore, for the right hand side of \eqref{ineq:Op1-1}, it follows from 
uniform continuity of $B_1$ that
\begin{align*}
&P\left(\lambda^\top\left[\int_{0}^{1}{\left\{B_1(\alpha_0,\alpha_0)-B_1(\alpha_0+u(\tilde{\alpha}_{\varepsilon,n}^{(1)}-\alpha_0),\alpha_0)\right\}}du\right]\lambda>\frac{\eta}{12}\right)\\
&\le P\left(\left\{\sup_{\alpha\in N_{n,\alpha}}\left| B_1(\alpha_0,\alpha_0)-B_1(\alpha,\alpha_0)\right| >\frac{\eta}{12|\lambda|^2}\right\}\cap A_{n,\alpha}\right)+P(A_{n,\alpha}^c)\\
&\le P\left(\sup_{\alpha\in N_{n,\alpha}}\left| B_1(\alpha_0,\alpha_0)-B_1(\alpha,\alpha_0)\right| >\frac{\eta}{12|\lambda|^2}\right)+P(A_{n,\alpha}^c)\\
&\to0\quad(\varepsilon\to0,n\to\infty),
\end{align*}
and for \eqref{ineq:Op1-2}, we deduce from the uniformly convergence \eqref{conv:B1} that
\begin{align*}
&P\left(\lambda^\top\left[\int_{0}^{1}{\left\{2B_1(\alpha_0+u(\tilde{\alpha}_{\varepsilon,n}^{(1)}-\alpha_0),\alpha_0)-\varepsilon^2\partial_{{\alpha}}^2U_{\varepsilon,n,v}^{(1)}(\alpha_0+u(\tilde{\alpha}_{\varepsilon,n}^{(1)}-\alpha_0))\right\}}du\right]\lambda>\frac{\eta}{6}\right)\\
&\le P\left(\sup_{\alpha\in\Theta_\alpha}\left|2B_1(\alpha,\alpha_0)-\varepsilon^2\partial_{{\alpha}}^2U_{\varepsilon,n,v}^{(1)}(\alpha)\right|>\frac{\eta}{6|\lambda|^2}\right)\\
&\to0\quad(\varepsilon\to0,n\to\infty).
\end{align*}
Consequently, we obtain
\begin{align}\label{taylor:U1:part2}
P\left(\lambda^\top\left(\varepsilon^2\int_{0}^{1}{\partial_{{\alpha}}^2U_{\varepsilon,n,v}^{(1)}(\alpha_0+u(\tilde{\alpha}_{\varepsilon,n}^{(1)}-\alpha_0))}du\right)\lambda>\frac{\eta}{6}\right)\to 1\quad(\varepsilon\to0,n\to\infty),
\end{align}
and hence, 
%in the equation \eqref{taylor:U1}, the results \eqref{taylor:U1:part1} 
%and \eqref{taylor:U1:part2} lead to 
it follows from \eqref{taylor:U1},  \eqref{taylor:U1:part1} 
and \eqref{taylor:U1:part2}
that one has
$\varepsilon^{-1}(\tilde{\alpha}^{(1)}_{\varepsilon,n}-\alpha_0)=O_P(1)$.

Third, we prove the asymptotic normality of $\tilde{\alpha}_{\varepsilon,n}^{(1)}$. In an analogous manner to S\o rensen and Uchida \cite{Sorensen_Uchida_2003}, it is sufficient to show the following two properties:
\begin{align}
&\sup_{u\in[0,1]}\left|\varepsilon^2\partial_{{\alpha}}^2U_{\varepsilon,n,v}^{(1)}(\alpha_0+u(\tilde{\alpha}_{\varepsilon,n}^{(1)}-\alpha_0))-2J_b(\alpha_0)\right|\overset{P}{\to}0,\label{main1:lem1}\\
&-\varepsilon\partial_{{\alpha}}U_{\varepsilon,n,v}^{(1)}(\alpha_0)\overset{d}{\to} N_{p}(0,4K_b(\theta_0))\label{main2:lem1}.
\end{align}
For \eqref{main1:lem1}, it follows from \eqref{conv:B1}, the consistency 
of $\tilde{\alpha}_{\varepsilon,n}^{(1)}$ and [{\bf A2'}] that for all $\delta>0$,
\begin{align*}
&P\left(\sup_{u\in[0,1]}\left|\varepsilon^2\partial_{{\alpha}}^2U_{\varepsilon,n,v}^{(1)}(\alpha_0+u(\tilde{\alpha}_{\varepsilon,n}^{(1)}-\alpha_0))-2J_b(\alpha_0)\right|>\delta\right)\\
&\le P\left(\sup_{u\in[0,1]}\left|\varepsilon^2\partial_{{\alpha}}^2U_{\varepsilon,n,v}^{(1)}(\alpha_0+u(\tilde{\alpha}_{\varepsilon,n}^{(1)}-\alpha_0))-2B_1(\alpha_0+u(\tilde{\alpha}_{\varepsilon,n}^{(1)}-\alpha_0),\alpha_0)\right|>\frac{\delta}{3}\right)\\
&\quad+ P\left(\sup_{u\in[0,1]}\left|2B_1(\alpha_0+u(\tilde{\alpha}_{\varepsilon,n}^{(1)}-\alpha_0),\alpha_0)-2B_1(\alpha_0,\alpha_0)\right|>\frac{\delta}{3}\right)\\
&\quad+ P\left(\sup_{u\in[0,1]}\left|2\int_{0}^{1}{\left(\partial_{{\alpha}_{l_1l_2}}^2b(X_s^0,\alpha_0+u(\tilde{\alpha}_{\varepsilon,n}^{(1)}-\alpha_0))\right)^\top\left(b(X_s^0,\alpha_0+u(\tilde{\alpha}_{\varepsilon,n}^{(1)}-\alpha_0))-b(X_s^0,\alpha_0)\right)}ds\right|>\frac{\delta}{3}\right)\\
&\le P\left(\sup_{\alpha\in\Theta_\alpha}\left|\varepsilon^2\partial_{{\alpha}}^2U_{\varepsilon,n,v}^{(1)}(\alpha)-2B_1(\alpha,\alpha_0)\right|>\frac{\delta}{3}\right)+P\left(\sup_{\alpha\in N_{n,\alpha}}\left|2B_1(\alpha,\alpha_0)-2B_1(\alpha_0,\alpha_0)\right|>\frac{\delta}{3}\right)\\
&\quad+P\left(\sup_{\alpha\in N_{n,\alpha}}\left|b(X_s^0,\alpha)-b(X_s^0,\alpha_0)\right|>\frac{\delta}{6C}\right)+3P(A_{n,\alpha}^c)\\
&\to0\quad(\varepsilon\to0,n\to\infty).
\end{align*}
Regarding \eqref{main2:lem1}, we have for $1\le l\le p$,
\begin{align}\label{asym:lem1}
-\varepsilon\partial_{{\alpha}_l}U_{\varepsilon,n,v}^{(1)}(\alpha_0)&=\sum_{k=1}^{n}{\zeta_{k,1}^l(\alpha_0)}+\sum_{k=1}^{n}{R_d(\varepsilon^{-1}n^{-1},X_{t_{k-1}^n})^\top P_{v,k}(\alpha_0)},
\end{align}
where
\begin{align}\label{xi:1}
\zeta_{k,1}^l(\alpha_0):= 2\varepsilon^{-1} \left(\partial_{{\alpha}_l}b(X_{t_{k-1}^n},\alpha_0)\right)^\top P_{v,k}(\alpha_0).
\end{align}
It follows from Lemma 1 in Gloter and S\o rensen \cite{Gloter_2009} that
\begin{align*}
\mathbb{E}\left[\left|R_d(\varepsilon^{-1}n^{-1},X_{t_{k-1}^n})^\top P_{v,k}(\alpha_0)\right| |\mathcal{G}_{k-1}^n\right]
&=\mathbb{E}\left[\left|\sum_{i=1}^{d}R(\varepsilon^{-1}n^{-1},X_{t_{k-1}^n}) P_{v,k}^i(\alpha_0)\right| |\mathcal{G}_{k-1}^n\right]\\
&\le \sum_{i=1}^{d}{\mathbb{E}\left[|P_{v,k}^i(\alpha_0)||\mathcal{G}_{k-1}^n\right]} R(\varepsilon^{-1}n^{-1},X_{t_{k-1}^n})\\
&\le \sum_{i=1}^{d}{\mathbb{E}\left[|P_{v,k}^i(\alpha_0)|^2|\mathcal{G}_{k-1}^n\right]}^{\frac{1}{2}} R(\varepsilon^{-1}n^{-1},X_{t_{k-1}^n})\\
&\le R(n^{-\frac{3}{2}},X_{t_{k-1}^n}),
\end{align*}
and hence the second term of the right hand side in \eqref{asym:lem1} 
converges to $0$ in probability as $\varepsilon\to0$ and $n\to\infty$. From Theorems 3.2 and 3.4 in Hall and Heyde \cite{hall}, it is sufficient to show the following convergences: for $1\le l_1,l_2\le p$,
\begin{align*}
&\sum_{k=1}^{n}{\mathbb{E}\left[\zeta_{k,1}^{l_1}(\alpha_0)|\mathcal{G}_{k-1}^n\right]\overset{P}{\to}0}\\
&\sum_{k=1}^{n}{\mathbb{E}\left[\zeta_{k,1}^{l_1}(\alpha_0)\zeta_{k,1}^{l_2}(\alpha_0)|\mathcal{G}_{k-1}^n\right]\overset{P}{\to}4K_b^{l_1,l_2}(\theta_0)}\\
&\sum_{k=1}^{n}{\mathbb{E}\left[\zeta_{k,1}^{l_1}(\alpha_0)|\mathcal{G}_{k-1}^n\right]\mathbb{E}\left[\zeta_{k,1}^{l_2}(\alpha_0)|\mathcal{G}_{k-1}^n\right]\overset{P}{\to}0}\\
&\sum_{k=1}^{n}{\mathbb{E}\left[\left(\zeta_{k,1}^{l_1}(\alpha_0)\right)^4|\mathcal{G}_{k-1}^n\right]\overset{P}{\to}0}.
\end{align*}
The above convergences are obtained by Lemma 1 and 4 in Gloter and S\o rensen \cite{Gloter_2009}. We omit the detailed proof.
\end{proof}
%%%%%%%%%%%%%%%%%%%
\begin{proof}[\bf Proof of Theorem \ref{thm:1}]
{\bf 1st step.} We prove the consistency of $\tilde{\beta}_{\varepsilon,n}$. By the definition of the contrast function $U_{\varepsilon,n,v}^{(2)}(\beta|\bar{\alpha})$, we have
\small
\begin{align}
&\frac{1}{n}\left(U_{\varepsilon,n,v}^{(2)}(\beta|\tilde{\alpha}_{\varepsilon,n}^{(1)})-U_{\varepsilon,n,v}^{(2)}(\beta_0|\tilde{\alpha}_{\varepsilon,n}^{(1)})\right)\notag\\
&=
\frac{1}{n}\sum_{k=1}^{n}{\left\{\log\det\left([\sigma\sigma^\top](X_{t_{k-1}^n},\beta)[\sigma\sigma^\top]^{-1}(X_{t_{k-1}^n},\beta_0)\right)\right\}}\label{U2:1}\\
&\ +\varepsilon^{-2}\sum_{k=1}^{n}{P_{v,k}^\top(\alpha_0)\left([\sigma\sigma^\top]^{-1}(X_{t_{k-1}^n},\beta)-[\sigma\sigma^\top]^{-1}(X_{t_{k-1}^n},\beta_0)\right)P_{v,k}(\alpha_0)}\label{U2:2}\\
&\ +2\varepsilon^{-2}n^{-1}\sum_{k=1}^{n}{\left(nP_{v,k}(\tilde{\alpha}_{\varepsilon,n}^{(1)})-nP_{v,k}(\alpha_0)\right)^\top\left([\sigma\sigma^\top]^{-1}(X_{t_{k-1}^n},\beta)-[\sigma\sigma^\top]^{-1}(X_{t_{k-1}^n},\beta_0)\right)P_{v,k}(\alpha_0)}\label{U2:3}\\
&\ +(\varepsilon n)^{-2}\sum_{k=1}^{n}{\left(nP_{v,k}(\tilde{\alpha}_{\varepsilon,n}^{(1)})-nP_{v,k}(\alpha_0)\right)^\top\left([\sigma\sigma^\top]^{-1}(X_{t_{k-1}^n},\beta)-[\sigma\sigma^\top]^{-1}(X_{t_{k-1}^n},\beta_0)\right)\left(nP_{v,k}(\tilde{\alpha}_{\varepsilon,n}^{(1)})-nP_{v,k}(\alpha_0)\right)}\label{U2:4}.
\end{align}
\normalsize
It follows from Lemmas 4 and 5 in Gloter and S\o rensen \cite{Gloter_2009} 
that the sum of \eqref{U2:1} and \eqref{U2:2}
converges to $U_2(\beta,\beta_0)$ in probability uniformly in   $\beta$, where $U_2$ is defined by \eqref{U2:def-small}:
\small
\begin{align*}
U_2(\beta,\beta_0)=\int_{0}^{1}\left\{\log\det\left([\sigma\sigma^\top](X_s^0,\beta)[\sigma\sigma^\top]^{-1}(X_s^0,\beta_0)\right)+\mathrm{tr}\left([\sigma\sigma^\top]^{-1}(X_s^0,\beta)[\sigma\sigma^\top](X_s^0,\beta_0)\right)-d\right\}ds.
\end{align*}
\normalsize
Noting that
the condition [{\bf A2'}] leads to the Lipschitz continuity
of the functions $nP_{v,k}$ and $[\sigma\sigma^\top]^{-1}$, we deduce from Lemma 1-(5) in Gloter and S\o rensen \cite{Gloter_2009}  and Lemma \ref{lem:1} in this paper that  \eqref{U2:3} and \eqref{U2:4} converge to 0 in probability uniformly in   $\beta$.
Therefore, one deduces that
\begin{align}\label{conv:U2}
\frac{1}{n}\left(U_{\varepsilon,n,v}^{(2)}(\beta|\tilde{\alpha}_{\varepsilon,n}^{(1)})-U_{\varepsilon,n,v}^{(2)}(\beta_0|\tilde{\alpha}_{\varepsilon,n}^{(1)})\right)
\overset{P}{\to}U_2(\beta,\beta_0)\quad\text{uniformly in } \beta.
\end{align}
Let $\omega\in\Omega$ be fixed. It follows from the compactness of $\Theta$ and the consistency of $\tilde{\alpha}_{\varepsilon,n}^{(1)}$ that for any sequence $(\varepsilon_m,n_m)$, 
there exists a subsequence $(\varepsilon_m',n_m')$ such that
\begin{align}\label{conv:beta}
\left(\tilde{\alpha}_{\varepsilon_m',n_m'}^{(1)}(\omega), \tilde{\beta}_{\varepsilon_m',n_m'}(\omega)\right)\to(\alpha_0,\beta_\infty)\in\Theta\quad(\varepsilon_m'\to0,n_m'\to\infty).
\end{align}
From \eqref{conv:U2} , the continuity of $U_2$ and the definition of $\tilde{\beta}_{\varepsilon,n}$, 
we obtain that
\begin{align*}
0\ge\frac{1}{n}\left(U_{\varepsilon,n,v}^{(2)}(\tilde{\beta}_{\varepsilon_m',n_m'}(\omega)|\tilde{\alpha}_{\varepsilon_m',n_m'}^{(1)}(\omega))-U_{\varepsilon,n,v}^{(2)}(\beta_0|\tilde{\alpha}_{\varepsilon_m',n_m'}^{(1)}(\omega))\right)(\omega)\to U_2(\beta_\infty,\beta_0)\ge0.
\end{align*}
By [{\bf A3}] and the proof of Lemma 17 in Genon-Catalot and Jacod \cite{genon_1993}, we have $\beta_\infty=\beta_0$, and \eqref{conv:beta} means that 
$\tilde{\beta}_{\varepsilon,n}\overset{P}{\to}\beta_0$.

{\bf 2nd step.} Next we show the consistency of $\tilde{\alpha}_{\varepsilon,n}$. By the definition of the contrast function $U_{\varepsilon,n,v}^{(3)}(\alpha|\bar{\beta})$, we have
\fontsize{7.9pt}{0cm}\selectfont
\begin{align}
&\varepsilon^2\left(U_{\varepsilon,n,v}^{(3)}(\alpha|\tilde{\beta}_{\varepsilon,n})-U_{\varepsilon,n,v}^{(3)}(\alpha_0|\tilde{\beta}_{\varepsilon,n})\right)\notag\\
&=2n\sum_{k=1}^{n}{\left(P_{v,k}(\alpha)-P_{v,k}(\alpha_0)\right)^\top [\sigma\sigma^\top]^{-1}(X_{t_{k-1}^n},\tilde{\beta}_{\varepsilon,n})}P_{v,k}(\alpha_0)\notag\\
&\quad+n\sum_{k=1}^{n}{\left(P_{v,k}(\alpha)-P_{v,k}(\alpha_0)\right)^\top [\sigma\sigma^\top]^{-1}(X_{t_{k-1}^n},\beta_0)}\left(P_{v,k}(\alpha)-P_{v,k}(\alpha_0)\right)\notag\\
&\quad+n\sum_{k=1}^{n}{\left(P_{v,k}(\alpha)-P_{v,k}(\alpha_0)\right)^\top \left([\sigma\sigma^\top]^{-1}(X_{t_{k-1}^n},\tilde{\beta}_{\varepsilon,n})-[\sigma\sigma^\top]^{-1}(X_{t_{k-1}^n},\beta_0)\right)\left(P_{v,k}(\alpha)-P_{v,k}(\alpha_0)\right)}\notag\\
&=2\sum_{k=1}^{n}{\left(b(X_{t_{k-1}^n},\alpha_0)-b(X_{t_{k-1}^n},\alpha)+R_d(n^{-1},X_{t_{k-1}^n})\right)^\top [\sigma\sigma^\top]^{-1}(X_{t_{k-1}^n},\tilde{\beta}_{\varepsilon,n})}P_{v,k}(\alpha_0)\label{U3:1}\\
&\quad+\frac{1}{n}\sum_{k=1}^{n}{\left(b(X_{t_{k-1}^n},\alpha)-b(X_{t_{k-1}^n},\alpha_0)\right)^\top [\sigma\sigma^\top]^{-1}(X_{t_{k-1}^n},\beta_0)}\left(b(X_{t_{k-1}^n},\alpha)-b(X_{t_{k-1}^n},\alpha_0)\right)\label{U3:2}\\
&\quad+\frac{1}{n}\sum_{k=1}^{n}{\left(b(X_{t_{k-1}^n},\alpha)-b(X_{t_{k-1}^n},\alpha_0)\right)^\top \left([\sigma\sigma^\top]^{-1}(X_{t_{k-1}^n},\tilde{\beta}_{\varepsilon,n})-[\sigma\sigma^\top]^{-1}(X_{t_{k-1}^n},\beta_0)\right)\left(b(X_{t_{k-1}^n},\alpha)-b(X_{t_{k-1}^n},\alpha_0)\right)}\label{U3:3}\\
&\quad+\frac{1}{n}\sum_{k=1}^{n}{R(n^{-1},X_{t_{k-1}^n})}.\notag
\end{align}
\fontsize{10pt}{0cm}\selectfont
It follows from Lemma 4-(2) in Gloter and S\o rensen \cite{Gloter_2009} for \eqref{U3:1}, 
and the consistency of $\tilde{\beta}_{\varepsilon,n}$ for \eqref{U3:3} 
that the two terms converge to 0 in probability uniformly in   $\alpha$.
Moreover, from Lemma 4-(1) in Gloter and S\o rensen \cite{Gloter_2009}, 
 \eqref{U3:2} converges to $U_3(\alpha, \theta_0)$ in probability uniformly in  $\alpha$, where
\begin{align}
U_3(\alpha,\theta_0)=\int_{0}^{1}{\left(b(X_{s}^0,\alpha)-b(X_{s}^0,\alpha_0)\right)^\top [\sigma\sigma^\top]^{-1}(X_{s}^0,\beta_0)\left(b(X_{s}^0,\alpha)-b(X_{s}^0,\alpha_0)\right)}ds.
\end{align}
Therefore, we have
\begin{align}
\varepsilon^2\left(U_{\varepsilon,n,v}^{(3)}(\alpha|\tilde{\beta}_{\varepsilon,n})-U_{\varepsilon,n,v}^{(3)}(\alpha_0|\tilde{\beta}_{\varepsilon,n})\right)\overset{P}{\to}U_3(\alpha,\theta_0)\quad\text{uniformly in }\alpha.
\end{align}
In an analogous manner to the proof of the consistency for $\tilde{\alpha}_{\varepsilon,n}^{(1)}$ and $\tilde{\beta}_{\varepsilon,n}$, we have $\tilde{\alpha}_{\varepsilon,n}\overset{P}{\to}\alpha_0$.

{\bf 3rd step.} We prove the asymptotic normality for $\tilde{\theta}_{\varepsilon,n}$. From Taylor's theorem, we have the following expansions:
\begin{align*}
-\varepsilon\partial_{{\alpha}}U_{\varepsilon,n,v}^{(3)}(\alpha_0|\tilde{\beta}_{\varepsilon,n})&=\left(\varepsilon^2\int_{0}^{1}{\partial_{{\alpha}}}^2U_{\varepsilon,n,v}^{(3)}(\alpha_0+u(\tilde{\alpha}_{\varepsilon,n}-\alpha_0)|\tilde{\beta}_{\varepsilon,n})du\right)\varepsilon^{-1}(\tilde{\alpha}_{\varepsilon,n}-\alpha_0), \\
-\frac{1}{\sqrt{n}}\partial_{{\beta}}U_{\varepsilon,n,v}^{(2)}(\beta_0|\tilde{\alpha}_{\varepsilon,n}^{(1)})&=
\left(\frac{1}{n}\int_{0}^{1}{\partial_{{\beta}}}^2U_{\varepsilon,n,v}^{(2)}(\beta_0+u(\tilde{\beta}_{\varepsilon,n}-\beta_0)|\tilde{\alpha}_{\varepsilon,n}^{(1)})du\right)\sqrt{n}(\tilde{\beta}_{\varepsilon,n}-\beta_0), \\
\varepsilon\partial_{{\alpha}}U_{\varepsilon,n,v}^{(3)}(\alpha_0|\tilde{\beta}_{\varepsilon,n})-\varepsilon\partial_{{\alpha}}U_{\varepsilon,n,v}^{(3)}(\alpha_0|\beta_0)&=
\left(\frac{\varepsilon}{\sqrt{n}}\int_{0}^{1}{\partial_{\alpha\beta}^2U_{\varepsilon,n,v}^{(3)}(\alpha_0|\beta_0+u(\tilde{\beta}_{\varepsilon,n}-\beta_0))}du\right)\sqrt{n}(\tilde{\beta}_{\varepsilon,n}-\beta_0).
\end{align*}
Using these expressions, we calculate that 
\begin{align*}
\Gamma_{\varepsilon,n}^1=C_{\varepsilon,n}^1\Lambda_{\varepsilon,n}^1,
\end{align*}
where
\begin{align*}
\Gamma_{\varepsilon,n}^1:=\begin{pmatrix}
-\varepsilon\partial_{{\alpha}}U_{\varepsilon,n,v}^{(3)}(\alpha_0|\beta_0) \\
-\frac{1}{\sqrt{n}}\partial_{{\beta}}U_{\varepsilon,n,v}^{(2)}(\beta_0|\tilde{\alpha}_{\varepsilon,n}^{(1)})
\end{pmatrix},\quad
\Lambda_{\varepsilon,n}^1:=\begin{pmatrix}
\varepsilon^{-1}(\tilde{\alpha}_{\varepsilon,n}-\alpha_0)\\
\sqrt{n}(\tilde{\beta}_{\varepsilon,n}-\beta_0)
\end{pmatrix}, 
\end{align*}
\begin{align*}
C_{\varepsilon,n}^1:=\begin{pmatrix}
\varepsilon^2\int_{0}^{1}{\partial_{{\alpha}}}^2U_{\varepsilon,n,v}^{(3)}(\alpha_0+u(\tilde{\alpha}_{\varepsilon,n}-\alpha_0)|\tilde{\beta}_{\varepsilon,n})du  & \frac{\varepsilon}{\sqrt{n}}\int_{0}^{1}{\partial_{\alpha\beta}^2U_{\varepsilon,n,v}^{(3)}(\alpha_0|\beta_0+u(\tilde{\beta}_{\varepsilon,n}-\beta_0))}du\\
0 & \frac{1}{n}\int_{0}^{1}{\partial_{{\beta}}}^2U_{\varepsilon,n,v}^{(2)}(\beta_0+u(\tilde{\beta}_{\varepsilon,n}-\beta_0)|\tilde{\alpha}_{\varepsilon,n}^{(1)})du
\end{pmatrix}.
\end{align*}
In an analogous manner to the proof of Theorem 1 in S\o rensen and Uchida \cite{Sorensen_Uchida_2003}, it is sufficient to show the following convergences:
\begin{align}
&\sup_{u\in[0,1]}\left|\varepsilon^2\partial_{{\alpha}}^2U_{\varepsilon,n,v}^{(3)}(\alpha_0+u(\tilde{\alpha}_{\varepsilon,n}-\alpha_0)|\tilde{\beta}_{\varepsilon,n})-2I_b(\theta_0)\right|\overset{P}{\to}0, \label{thm1:1}\\
&\sup_{u\in[0,1]}\left|\frac{1}{n}\partial_{{\beta}}^2U_{\varepsilon,n,v}^{(2)}(\beta_0+u(\tilde{\beta}_{\varepsilon,n}-\beta_0)|\tilde{\alpha}_{\varepsilon,n}^{(1)})-2I_\sigma(\theta_0)\right|\overset{P}{\to}0, \label{thm1:2}\\
&\sup_{u\in[0,1]}\left|\frac{\varepsilon}{\sqrt{n}}\partial_{\alpha\beta}^2U_{\varepsilon,n,v}^{(3)}(\alpha_0|\beta_0+u(\tilde{\beta}_{\varepsilon,n}-\beta_0))\right|\overset{P}{\to}0, \label{thm1:3}
\end{align}
\begin{align}
&\quad\Gamma_{\varepsilon,n}^1\overset{d}{\to} N_{p+q}(0,4I(\theta_0)). \label{thm1:4}
\end{align}

Proof of \eqref{thm1:1}. 
By a simple calculation, it holds that for $1\le l_1,l_2\le p$,
\begin{align}
&\varepsilon^2 \partial_{{\alpha}_{l_1l_2}}^2 U_{\varepsilon,n,v}^{(3)}(\alpha|\tilde{\beta}_{\varepsilon,n})\notag\\
&=
2n\sum_{k=1}^{n}{\partial_{{\alpha}_{l_1l_2}}^2 P_{v,k}^\top(\alpha)[\sigma\sigma^\top]^{-1}(X_{t_{k-1}^n},\tilde{\beta}_{\varepsilon,n})P_{v,k}(\alpha_0)}\label{C1:1}\\
&\quad+2n\sum_{k=1}^{n}{\partial_{{\alpha}_{l_1l_2}}^2 P_{v,k}^\top(\alpha)[\sigma\sigma^\top]^{-1}(X_{t_{k-1}^n},\beta_0)\left(P_{v,k}(\alpha)-P_{v,k}(\alpha_0)\right)}\label{C1:2}\\
&\quad+2n\sum_{k=1}^{n}{\partial_{{\alpha}_{l_1l_2}}^2 P_{v,k}^\top(\alpha)\left([\sigma\sigma^\top]^{-1}(X_{t_{k-1}^n},\tilde{\beta}_{\varepsilon,n})-[\sigma\sigma^\top]^{-1}(X_{t_{k-1}^n},\beta_0)\right)\left(P_{v,k}(\alpha)-P_{v,k}(\alpha_0)\right)}\label{C1:3}\\
&\quad+2n\sum_{k=1}^{n}{\partial_{{\alpha}_{l_1}} P_{v,k}^\top(\alpha)[\sigma\sigma^\top]^{-1}(X_{t_{k-1}^n},\beta_0)\partial_{{\alpha}_{l_2}} P_{v,k}(\alpha)}\label{C1:4}\\
&\quad+2n\sum_{k=1}^{n}{\partial_{{\alpha}_{l_1}} P_{v,k}^\top(\alpha)\left([\sigma\sigma^\top]^{-1}(X_{t_{k-1}^n},\tilde{\beta}_{\varepsilon,n})-[\sigma\sigma^\top]^{-1}(X_{t_{k-1}^n},\beta_0)\right)\partial_{{\alpha}_{l_2}} P_{v,k}(\alpha)}\label{C1:5}.
\end{align}
By noting that $\partial_{{\alpha}_{l_1}} P_{v,k}(\alpha)=-n^{-1}\partial_{{\alpha}_{l_1}}b(X_{t_{k-1}^n},\alpha)+R_d(n^{-2},X_{t_{k-1}^n})=R_d(n^{-1},X_{t_{k-1}^n})$ and $
P_{v,k}(\alpha)-P_{v,k}(\alpha_0)$
$=n^{-1}\left(b(X_{t_{k-1}^n},\alpha_0)-b(X_{t_{k-1}^n},\alpha)\right)+R_d(n^{-2},X_{t_{k-1}^n})=R_d(n^{-1},X_{t_{k-1}^n})$, it follows from Lemma 4-(1) in Gloter and S\o rensen \cite{Gloter_2009} that the sum of \eqref{C1:2} and \eqref{C1:4} converges to $2B_2^{l_1,l_2}(\alpha,\theta_0)$ in probability uniformly in   $\alpha$, where
\begin{align*}
B_2^{l_1,l_2}(\alpha,\theta_0)&:=\int_{0}^{1}{\left(\partial_{{\alpha}_{l_1l_2}}^2b(X_s^0,\alpha)\right)^\top [\sigma\sigma^\top]^{-1}(X_s^0,\beta_0)\left(b(X_s^0,\alpha)-b(X_s^0,\alpha_0)\right)}ds\\
&\quad+\int_{0}^{1}{\left(\partial_{{\alpha}_{l_1}}b(X_s^0,\alpha)\right)^\top [\sigma\sigma^\top]^{-1}(X_s^0,\beta_0)\left(\partial_{{\alpha}_{l_2}}b(X_s^0,\alpha)\right)}ds.
\end{align*}
On the other hand, by the consistency of $\tilde{\beta}_{\varepsilon,n}$ and Lemma 4 in Gloter and S\o rensen \cite{Gloter_2009}, 
the remaining terms \eqref{C1:1}, \eqref{C1:3} and \eqref{C1:5} converge to 0 in probability. Hence, we have
\begin{align}\label{conv:B2}
\sup_{\alpha\in\Theta_\alpha}\left|\varepsilon^2 \partial_{{\alpha}_{l_1l_2}}^2 U_{\varepsilon,n,v}^{(3)}(\alpha|\tilde{\beta}_{\varepsilon,n})-2B_2^{l_1,l_2}(\alpha,\theta_0)\right|\overset{P}{\to}0.
\end{align}
By noting that $B_2(\alpha_0,\theta_0)=I_b(\theta_0)$, it follows from the consistency of $\tilde{\alpha}_{\varepsilon,n}$ and the uniform continuity of $B_2$ that for all $\delta>0$,
\begin{align*}
&P\left(\sup_{u\in[0,1]}\left|\varepsilon^2\partial_{{\alpha}}^2U_{\varepsilon,n,v}^{(3)}(\alpha_0+u(\tilde{\alpha}_{\varepsilon,n}-\alpha_0)|\tilde{\beta}_{\varepsilon,n})-2I_b(\theta_0)\right|>\delta\right)\\
&\le P\left(\sup_{\alpha\in\Theta_\alpha}\left|\varepsilon^2\partial_{{\alpha}}^2U_{\varepsilon,n,v}^{(3)}(\alpha|\tilde{\beta}_{\varepsilon,n})-2B_2(\alpha,\theta_0)\right|>\frac{\delta}{2}\right)\\
&\quad+P\left(\sup_{u\in[0,1]}\left|2B_2(\alpha_0+u(\tilde{\alpha}_{\varepsilon,n}-\alpha_0),\theta_0)-2I(\theta_0)\right|>\frac{\delta}{2}\right)\\
&\to 0\quad(\varepsilon\to0,n\to\infty).
\end{align*}
This implies \eqref{thm1:1}.

Proof of \eqref{thm1:2}. 
We deduce that for $1\le m_1,m_2\le q$, 
\begin{align}
&\frac{1}{n}\partial_{{\beta}_{m_1m_2}}^2U_{\varepsilon,n,v}^{(2)}(\beta|\tilde{\alpha}_{\varepsilon,n}^{(1)})\notag\\
&=\frac{1}{n}\sum_{k=1}^{n}{\partial_{{\beta}_{m_1m_2}}}\log\det [\sigma\sigma^\top](X_{t_{k-1}^n},\beta)\label{C2:1}\\
&\quad+\varepsilon^{-2}\sum_{k=1}^{n}{P_{v,k}(\alpha_0)^\top \left(\partial_{{\beta}_{m_1m_2}}^2[\sigma\sigma^\top]^{-1}(X_{t_{k-1}^n},\beta)\right)P_{v,k}(\alpha_0)}\label{C2:2}\\
&\quad+2\varepsilon^{-2}\sum_{k=1}^{n}{\left(P_{v,k}(\tilde{\alpha}_{\varepsilon,n}^{(1)})-P_{v,k}(\alpha_0)\right)^\top \left(\partial_{{\beta}_{m_1m_2}}^2[\sigma\sigma^\top]^{-1}(X_{t_{k-1}^n},\beta)\right)P_{v,k}(\alpha_0)}\label{C2:3}\\
&\quad+2\varepsilon^{-2}\sum_{k=1}^{n}{\left(P_{v,k}(\tilde{\alpha}_{\varepsilon,n}^{(1)})-P_{v,k}(\alpha_0)\right)^\top \left(\partial_{{\beta}_{m_1m_2}}^2[\sigma\sigma^\top]^{-1}(X_{t_{k-1}^n},\beta)\right)\left(P_{v,k}(\tilde{\alpha}_{\varepsilon,n}^{(1)})-P_{v,k}(\alpha_0)\right)}.\label{C2:4}
\end{align}
It follows from Lemma 4 and 5 in Gloter and S\o rensen \cite{Gloter_2009} that the sum of \eqref{C2:1} and \eqref{C2:2} converges to $B_3^{r_1,r_2}(\beta,\beta_0)$ in probability uniformly in   $\beta$, where
\begin{align*}
B_3^{m_1,m_2}(\beta,\beta_0)&:=-\int_{0}^{1}{\mathrm{tr}\left[\left([\sigma\sigma^\top]^{-1}\left(\partial_{{\beta}_{m_1}}[\sigma\sigma^\top]\right)[\sigma\sigma^\top]^{-1}\left(\partial_{{\beta}_{m_2}}[\sigma\sigma^\top]\right)\right)(X_s^0,\beta)\right]}ds\\
&\quad+\int_{0}^{1}{\mathrm{tr}\left[\left([\sigma\sigma^\top]^{-1}\left(\partial_{{\beta}_{m_1m_2}}^2[\sigma\sigma^\top]\right)\right)(X_s^0,\beta)\right]}ds\\
&\quad+2\int_{0}^{1}{\mathrm{tr}\left[\left([\sigma\sigma^\top]^{-1}\left(\partial_{{\beta}_{m_1}}[\sigma\sigma^\top]\right)[\sigma\sigma^\top]^{-1}\left(\partial_{{\beta}_{m_2}}[\sigma\sigma^\top]\right)[\sigma\sigma^\top]^{-1}\right)(X_s^0,\beta)[\sigma\sigma^\top](X_s^0,\beta_0)\right]}ds\\
&\quad-\int_{0}^{1}{\mathrm{tr}\left[\left([\sigma\sigma^\top]^{-1}\left(\partial_{{\beta}_{m_1m_2}}^2[\sigma\sigma^\top]\right)[\sigma\sigma^\top]^{-1}\right)(X_s^0,\beta)[\sigma\sigma^\top](X_s^0,\beta_0)\right]}ds.
\end{align*}
We can show 
from Lemma \ref{lem:1} in this paper and Lemma 1 in Gloter and S\o rensen \cite{Gloter_2009} that 
 \eqref{C2:3} and \eqref{C2:4} converge to 0 in probability uniformly in $\beta$. 
 Therefore, it holds that
\begin{align}\label{conv:B3}
\sup_{\beta\in\Theta_\beta}\left|\frac{1}{n}\partial_{{\beta}_{m_1m_2}}^2U_{\varepsilon,n,v}^{(2)}(\beta|\tilde{\alpha}_{\varepsilon,n}^{(1)})-B_3^{m_1,m_2}(\beta,\beta_0)\right|\overset{P}{\to}0.
\end{align}
Noting that $B_3(\beta_0,\beta_0)=2I_\sigma(\beta_0)$, in an analogous manner to the proof of \eqref{thm1:1}, 
we deduce that for  all $\delta>0$, 
\begin{align*}
P\left(\sup_{u\in[0,1]}\left|\frac{1}{n}\partial_{{\beta}}^2U_{\varepsilon,n,v}^{(2)}(\beta_0+u(\tilde{\beta}_{\varepsilon,n}-\beta_0)|\tilde{\alpha}_{\varepsilon,n}^{(1)})-2I_\sigma(\beta_0)\right|>\delta\right)\to0\quad(\varepsilon\to0,n\to\infty).
\end{align*}
This implies \eqref{thm1:2}.

Proof of \eqref{thm1:3}. From Lemma 4-(2)  in Gloter and S\o rensen \cite{Gloter_2009}, 
it holds that for $1\le l\le p$ and $1\le m\le q$,  
\begin{align*}
\frac{\varepsilon}{\sqrt{n}}\partial_{\alpha_l\beta_m}^2U_{\varepsilon,n,v}^{(3)}(\alpha_0|\beta)&=
2\varepsilon^{-1}\sqrt{n}\sum_{k=1}^{n}{\partial_{{\alpha}_l}P_{v,k}(\alpha_0)^\top \partial_{{\beta}_m}[\sigma\sigma^\top]^{-1}(X_{t_{k-1}^n},\beta)P_{v,k}(\alpha_0)}\\
&=\frac{2}{\sqrt{n}}\cdot \varepsilon^{-1}\sum_{k=1}^{n}{R_d(1,X_{t_{k-1}^n})^\top P_{v,k}(\alpha_0)}\\
&\overset{P}{\to}0 \quad\text{uniformly in }\beta.
\end{align*}
In particular, we have
\begin{align*}
\sup_{u\in[0,1]}\left|\frac{\varepsilon}{\sqrt{n}}\partial_{\alpha\beta}^2U_{\varepsilon,n,v}^{(3)}(\alpha_0|\beta_0+u(\tilde{\beta}_{\varepsilon,n}-\beta_0))\right|\overset{P}{\to}0.
\end{align*} 

Proof of \eqref{thm1:4}. By using Lemma 4-(2)  in Gloter and S\o rensen \cite{Gloter_2009},  
it holds that for $1\le l\le p$, 
\begin{align}
-\varepsilon\partial_{{\alpha}_l}U_{\varepsilon,n,v}^{(3)}(\alpha_0|\beta_0)&=-2\varepsilon^{-1}n\sum_{k=1}^{n}{\partial_{{\alpha}_l}}P_{v,k}^\top(\alpha_0)[\sigma\sigma^\top]^{-1}(X_{t_{k-1}^n},\beta_0)P_{v,k}(\alpha_0)\notag\\
&=\sum_{k=1}^{n}{\xi_{k,1}^l(\theta_0)}+o_P(1),\label{thm1:alpha}
\end{align}
where
\begin{align}
\xi_{k,1}^l(\theta_0):=2\varepsilon^{-1}\partial_{{\alpha}_l}b(X_{t_{k-1}^n},\alpha_0)^\top [\sigma\sigma^\top](X_{t_{k-1}^n},\beta_0)P_{v,k}(\alpha_0).
\end{align}
On the other hand, it follows from Taylor's theorem that
\begin{align*}
-\frac{1}{\sqrt{n}}\partial_{{\beta}}U_{\varepsilon,n,v}^{(2)}(\beta_0|\tilde{\alpha}_{\varepsilon,n}^{(1)})&=-\frac{1}{\sqrt{n}}\partial_{{\beta}}U_{\varepsilon,n,v}^{(2)}(\beta_0|\alpha_0)\\
&\quad-\left(\frac{\varepsilon}{\sqrt{n}}\int_{0}^{1}{\partial_{{\alpha\beta}}^2U_{\varepsilon,n,v}^{(2)}(\beta_0|\alpha_0+u(\tilde{\alpha}_{\varepsilon,n}^{(1)}-\alpha_0))}du\right)\varepsilon^{-1}(\tilde{\alpha}_{\varepsilon,n}^{(1)}-\alpha_0).
\end{align*}
In particular, it holds from Lemma \ref{lem:1} in this paper and Lemma 4-(2)  in Gloter and S\o rensen \cite{Gloter_2009} that
\begin{align*}
\frac{\varepsilon}{\sqrt{n}}\partial_{{\alpha\beta}}^2U_{\varepsilon,n,v}^{(2)}(\beta_0|\alpha)&=2\varepsilon^{-1}\sqrt{n}\sum_{k=1}^{n}{\partial_{{\alpha}_l}P_{v,k}(\alpha)^\top \partial_{{\beta}_m}[\sigma\sigma^\top]^{-1}(X_{t_{k-1}^n},\beta_0)P_{v,k}(\alpha_0)}\\
&\quad+2\varepsilon^{-1}\sqrt{n}\sum_{k=1}^{n}{\partial_{{\alpha}_l}P_{v,k}(\alpha)^\top \partial_{{\beta}_m}[\sigma\sigma^\top]^{-1}(X_{t_{k-1}^n},\beta_0)\left(P_{v,k}(\alpha)-P_{v,k}(\alpha_0)\right)}\\
&\overset{P}{\to}0\quad\text{uniformly in }\alpha,
\end{align*}
and we have
\begin{align*}
\sup_{u\in[0,1]}\left|\frac{\varepsilon}{\sqrt{n}}\partial_{{\alpha\beta}}^2U_{\varepsilon,n,v}^{(2)}(\beta_0|\alpha_0+u(\tilde{\alpha}_{\varepsilon,n}^{(1)}-\alpha_0))\right|\overset{P}{\to}0.
\end{align*}
Therefore, by Lemma \ref{lem:1}, 
we deduce that for $1\le m\le q$, 
\begin{align}
-\frac{1}{\sqrt{n}}\partial_{{\beta}_m}U_{\varepsilon,n,v}^{(2)}(\beta_0|\tilde{\alpha}_{\varepsilon,n}^{(1)})&=-\frac{1}{\sqrt{n}}\partial_{{\beta}_m}U_{\varepsilon,n,v}^{(2)}(\beta_0|\alpha_0)+o_P(1)\notag\\
&=\sum_{k=1}^{n}{\left(\eta_{k,1}^m(\beta_0)+\eta_{k,2}^m(\theta_0)\right)}+o_P(1),\label{thm1:beta}
\end{align} 
where
\begin{align}
&\eta_{k,1}^m(\beta_0):=-n^{-\frac{1}{2}}\mathrm{tr}\left[\left([\sigma\sigma^\top]^{-1}\partial_{{\beta}_m}[\sigma\sigma^\top]\right)(X_{t_{k-1}^n},\beta_0)\right], \\
&\eta_{k,2}^m(\theta_0):=\varepsilon^{-2}n^{\frac{1}{2}}P_{v,k}^\top(\alpha_0)\left(\left([\sigma\sigma^\top]^{-1}\left(\partial_{{\beta}_m}[\sigma\sigma^\top]\right)[\sigma\sigma^\top]^{-1}\right)(X_{t_{k-1}^n},\beta_0)\right)P_{v,k}(\alpha_0).
\end{align}
In order to show \eqref{thm1:4}, by Theorems 3.2 and 3.4 in Hall and Heyde \cite{hall}, it is sufficient to show the following convergences: for $1\le l_1,l_2\le p$ and $1\le m_1,m_2\le q$,
\begin{align*}
&\sum_{k=1}^{n}{\mathbb{E}\left[\xi_{k,1}^{l_1}(\theta_0)|\mathcal{G}_{k-1}^n\right]\overset{P}{\to}0}, \\
&\sum_{k=1}^{n}{\mathbb{E}\left[\eta_{k,1}^{m_1}(\beta_0)+\eta_{k,2}^{m_1}(\theta_0)|\mathcal{G}_{k-1}^n\right]\overset{P}{\to}0}, \\
&\sum_{k=1}^{n}{\mathbb{E}\left[\xi_{k,1}^{l_1}(\theta_0)\xi_{k,1}^{l_2}(\theta_0)|\mathcal{G}_{k-1}^n\right]}-\sum_{k=1}^{n}{\mathbb{E}\left[\xi_{k,1}^{l_1}(\theta_0)|\mathcal{G}_{k-1}^n\right]\mathbb{E}\left[\xi_{k,1}^{l_2}(\theta_0)|\mathcal{G}_{k-1}^n\right]}\overset{P}{\to}4I_b^{l_1,l_2}(\theta_0), \\
&\sum_{k=1}^{n}{\mathbb{E}\left[(\eta_{k,1}^{m_1}(\beta_0)+\eta_{k,2}^{m_1}(\theta_0))(\eta_{k,1}^{m_2}(\beta_0)+\eta_{k,2}^{m_2}(\theta_0))|\mathcal{G}_{k-1}^n\right]} \\
&\quad-\sum_{k=1}^{n}{\mathbb{E}\left[\eta_{k,1}^{m_1}(\beta_0)+\eta_{k,2}^{m_1}(\theta_0)|\mathcal{G}_{k-1}^n\right]\mathbb{E}\left[\eta_{k,1}^{m_2}(\beta_0)+\eta_{k,2}^{m_2}(\theta_0)|\mathcal{G}_{k-1}^n\right]}\overset{P}{\to}4I_\sigma^{m_1,m_2}(\beta_0), \\
&\sum_{k=1}^{n}{\mathbb{E}\left[\xi_{k,1}^{l_1}(\theta_0)(\eta_{k,1}^{m_1}(\beta_0)+\eta_{k,2}^{m_1}(\theta_0))|\mathcal{G}_{k-1}^n\right]}\\
&\quad-\sum_{k=1}^{n}{\mathbb{E}\left[\xi_{k,1}^{l_1}(\theta_0)|\mathcal{G}_{k-1}^n\right]\mathbb{E}\left[\eta_{k,1}^{m_1}(\beta_0)+\eta_{k,2}^{m_1}(\theta_0)|\mathcal{G}_{k-1}^n\right]}\overset{P}{\to}0, \\
&\sum_{k=1}^{n}{\mathbb{E}\left[\left(\xi_{k,1}^{l_1}(\theta_0)\right)^4|\mathcal{G}_{k-1}^n\right]\overset{P}{\to}0}, \\
&\sum_{k=1}^{n}{\mathbb{E}\left[\left(\eta_{k,1}^{m_1}(\beta_0)+\eta_{k,2}^{m_1}(\theta_0)\right)^4|\mathcal{G}_{k-1}^n\right]\overset{P}{\to}0}.
\end{align*}
The above properties are obtained by 
Lemmas 1 and 4 in Gloter and S\o rensen \cite{Gloter_2009}. 
We omit a detailed proof. Therefore, we have $\Lambda_{\varepsilon,n}^1\overset{d}{\to} N_d(0,I(\theta_0)^{-1})$, which completes the proof.
\end{proof}
%%%%%%%%%%%%%%%%%%%
\begin{proof}[\bf Proof of Lemma \ref{lem:2}]
First, we prove that $\hat{\alpha}_{\varepsilon,n}^{(l)}\overset{P}{\to}\alpha_0$ for $l=1,\ldots, v$. 
When $l=1$, a simple computation shows that 
\begin{align*}
\varepsilon^2\left(V_{\varepsilon,n,v}^{(1)}(\alpha)-V_{\varepsilon,n,v}^{(1)}(\alpha_0)\right)&=\frac{1}{n}\sum_{k=1}^{n}{\left(b(X_{t_{k-1}^n},\alpha_0)-b(X_{t_{k-1}^n},\alpha)\right)^\top\left(b(X_{t_{k-1}^n},\alpha_0)-b(X_{t_{k-1}^n},\alpha)\right)}\\
&\quad+2\sum_{k=1}^{n}{\left(b(X_{t_{k-1}^n},\alpha_0)-b(X_{t_{k-1}^n},\alpha)\right)^\top P_{1,k}(\alpha_0)}.
\end{align*}
By Lemma 2 in S\o rensen and Uchida \cite{Sorensen_Uchida_2003}, the first term of the right hand side converges in probability to $U_1(\alpha,\alpha_0)$ defined by \eqref{U1-small} and the second term converges to 0 in probability. Moreover, these convergences hold uniformly in $\alpha$, and we have
\begin{align}
\varepsilon^2\left(V_{\varepsilon,n,v}^{(1)}(\alpha)-V_{\varepsilon,n,v}^{(1)}(\alpha_0)\right)\overset{P}{\to}U_1(\alpha,\alpha_0)\quad \text{uniformly in }\alpha.
\end{align}
Therefore, in the same manner as the proof of the consistency of $\tilde{\alpha}_{\varepsilon,n}^{(1)}$ in Lemma \ref{lem:1}, we can show $\hat{\alpha}_{\varepsilon,n}^{(1)}\overset{P}{\to}\alpha_0$.
When $l\ge2$, assume that $\hat{\alpha}_{\varepsilon,n}^{(l-1)}$ is consistent. 
Noting that $Q_{l,k}(\bar{\alpha})=R_d(n^{-2},X_{t_{k-1}^n})$, we calculate that
\begin{align*}
\varepsilon^2\left(V_{\varepsilon,n,v}^{(l)}(\alpha|\hat{\alpha}_{\varepsilon,n}^{(l-1)})-V_{\varepsilon,n,v}^{(l)}(\alpha_0|\hat{\alpha}_{\varepsilon,n}^{(l-1)})\right)&=n\sum_{k=1}^{n}{\left(P_{1,k}(\alpha)-P_{1,k}(\alpha_0)\right)^\top\left(P_{1,k}(\alpha)-P_{1,k}(\alpha_0)\right)}\\
&\quad+2n\sum_{k=1}^{n}{\left(P_{1,k}(\alpha)-P_{1,k}(\alpha_0)\right)^\top\left(P_{1,k}(\alpha_0)-Q_{l,k}(\hat{\alpha}_{\varepsilon,n}^{(l-1)})\right)}\\
&=\frac{1}{n}\sum_{k=1}^{n}{\left(b(X_{t_{k-1}^n},\alpha_0)-b(X_{t_{k-1}^n},\alpha)\right)^\top\left(b(X_{t_{k-1}^n},\alpha_0)-b(X_{t_{k-1}^n},\alpha)\right)}\\
&\quad+2\sum_{k=1}^{n}{\left(b(X_{t_{k-1}^n},\alpha_0)-b(X_{t_{k-1}^n},\alpha)\right)^\top P_{1,k}(\alpha_0)}\\
&\quad+\frac{1}{n}\sum_{k=1}^{n}{R(n^{-1},X_{t_{k-1}^n})}.
\end{align*}
%Similarly, 
In a similar way to the proof of (\ref{conv:B3}),
we have
\begin{align}
\varepsilon^2\left(V_{\varepsilon,n,v}^{(l)}(\alpha|\hat{\alpha}_{\varepsilon,n}^{(l-1)})-V_{\varepsilon,n,v}^{(l)}(\alpha_0|\hat{\alpha}_{\varepsilon,n}^{(l-1)})\right)\overset{P}{\to}U_1(\alpha,\alpha_0)\quad \text{uniformly in }\alpha.
\end{align}
Hence, it follows from [{\bf A3}] and the consistency of $\hat{\alpha}_{\varepsilon,n}^{(l-1)}$ that $\hat{\alpha}_{\varepsilon,n}^{(l)}\overset{P}{\to}\alpha_0$.

Second, we show $\varepsilon^{-\frac{1}{v}}(\hat{\alpha}_{\varepsilon,n}^{(1)}-\alpha_0)\overset{P}{\to}0$. In the case of $v=1$, this statement is the same as 
that of Lemma \ref{lem:1}. Therefore, we assume $v\ge2$. It follows from Taylor's theorem that
\begin{align*}
-\varepsilon^{2-\frac{1}{v}}\partial_{{\alpha}}V_{\varepsilon,n,v}^{(1)}(\alpha_0)&=\left(\varepsilon^2\int_{0}^{1}{\partial_{{\alpha}}^2V_{\varepsilon,n,v}^{(1)}(\alpha_0+u(\hat{\alpha}_{\varepsilon,n}^{(1)}-\alpha_0))}du\right)\varepsilon^{-\frac{1}{v}}(\hat{\alpha}_{\varepsilon,n}^{(1)}-\alpha_0).
\end{align*}
In order to prove $\varepsilon^{-\frac{1}{v}}(\hat{\alpha}_{\varepsilon,n}^{(1)}-\alpha_0)\overset{P}{\to}0$, by [{\bf A4}]-(i), it is sufficient to show the following properties:
\begin{align}
&-\varepsilon^{2-\frac{1}{v}}\partial_{{\alpha}}V_{\varepsilon,n,v}^{(1)}(\alpha_0)\overset{P}{\to}0, \label{lem2:1-1}\\
&\sup_{u\in[0,1]}\left|\varepsilon^2\partial_{{\alpha}}^2V_{\varepsilon,n,v}^{(1)}(\alpha_0+u(\hat{\alpha}_{\varepsilon,n}^{(1)}-\alpha_0))-2J_b(\alpha_0)\right|\overset{P}{\to}0\label{lem2:1-2}.
\end{align}

Proof of \eqref{lem2:1-1}. By the definition of the contrast function $V_{\varepsilon,n,v}^{(1)}(\alpha)$, we have
\begin{align*}
-\varepsilon^{2-\frac{1}{v}}\partial_{{\alpha}}V_{\varepsilon,n,v}^{(1)}(\alpha_0)&=\sum_{k=1}^{n}{\psi_{k}^{(1)}(\alpha_0)},
\end{align*}
where
\begin{align*}
\psi_{k}^{(1)}(\alpha_0)&=2\varepsilon^{-\frac{1}{v}}{\partial_{{\alpha}}b(X_{t_{k-1}^n},\alpha_0)^\top P_{1,k}(\alpha_0)}.
\end{align*}
From Lemma 9 in Genon-Catalot and Jacod \cite{genon_1993}, it is sufficient to show the following two convergences.
\begin{align}
&\sum_{k=1}^{n}{\mathbb{E}\left[\psi^{(1)}_k(\alpha_0)|\mathcal{G}_{k-1}^n\right]}\overset{P}{\to}0, \label{uchida-1}\\
&\sum_{k=1}^{n}{\mathbb{E}\left[\left(\psi^{(1)}_k(\alpha_0)\right)^2|\mathcal{G}_{k-1}^n\right]}\overset{P}{\to}0. \label{uchida-2}
\end{align}
By using Lemmas 1 and 4 in Gloter and S\o rensen \cite{Gloter_2009}, 
the two terms on the left hand side in \eqref{uchida-1} and \eqref{uchida-2}
%are evaluated from above by 
are
$O_P(\varepsilon^{2-\frac{1}{v}})$ and $O_P(\varepsilon^{2(1-\frac{1}{v})})$, respectively. Therefore, it follows from $v\ge 2$ that the two properties hold.

Proof of \eqref{lem2:1-2}. By Lemma 2 in S\o rensen and Uchida \cite{Sorensen_Uchida_2003}, one deduces that
\begin{align*}
\varepsilon^2\partial_{{\alpha}}^2V_{\varepsilon,n,v}^{(1)}(\alpha)&=-\frac{2}{n}\sum_{k=1}^{n}{\left(\partial_{{\alpha}}^2b(X_{t_{k-1}^n},\alpha)\right)^\top \left(b(X_{t_{k-1}^n},\alpha_0)-b(X_{t_{k-1}^n},\alpha)\right)}\\
&\quad+\frac{2}{n}\sum_{k=1}^{n}{\left(\partial_{{\alpha}}b(X_{t_{k-1}^n},\alpha)\right)^\top\left(\partial_{{\alpha}}b(X_{t_{k-1}^n},\alpha)\right)}\\
&\quad-2\sum_{k=1}^{n}{\left(\partial_{{\alpha}}^2b(X_{t_{k-1}^n},\alpha)\right)^\top P_{1,k}(\alpha_0)}\\
&\overset{P}{\to} 2B_1(\alpha,\alpha_0)\quad\text{uniformly in }\alpha.
\end{align*}
By noting  that $B_1(\alpha_0,\alpha_0)=J_b(\alpha_0)$, it follows from the consistency of $\hat{\alpha}_{\varepsilon,n}^{(1)}$ and the uniform continuity of $B_1$ that for all $\delta>0$,
\begin{align*}
&P\left(\sup_{u\in[0,1]}\left|\varepsilon^2\partial_{{\alpha}}^2V_{\varepsilon,n,v}^{(1)}(\alpha_0+u(\hat{\alpha}_{\varepsilon,n}^{(1)}-\alpha_0))-2J_b(\alpha_0)\right|>\delta\right)\\
&\le P\left(\sup_{\alpha\in\Theta_\alpha}\left|\varepsilon^2\partial_{{\alpha}}^2V_{\varepsilon,n,v}^{(1)}(\alpha)-2B_1(\alpha,\alpha_0)\right|>\frac{\delta}{2}\right)\\
&\quad+P\left(\sup_{u\in[0,1]}\left|2B_1(\alpha_0+u(\hat{\alpha}_{\varepsilon,n}^{(1)}-\alpha_0),\alpha_0)-2J_b(\alpha_0)\right|>\frac{\delta}{2}\right)\\
&\to 0\quad(\varepsilon\to0,n\to\infty).
\end{align*}
This implies \eqref{lem2:1-2}.

Third, we assume $\varepsilon^{-\frac{l-1}{v}}(\hat{\alpha}_{\varepsilon,n}^{(l-1)}-\alpha_0)\overset{P}{\to}0$ and prove $\varepsilon^{-\frac{l}{v}}(\hat{\alpha}_{\varepsilon,n}^{(l)}-\alpha_0)\overset{P}{\to}0$ for $2\le l\le v-1$. From Taylor's theorem, we deduce that
\begin{align*}
-\varepsilon^{2-\frac{l}{v}}\partial_{{\alpha}}V_{\varepsilon,n,v}^{(l)}(\alpha_0|\hat{\alpha}_{\varepsilon,n}^{(l-1)})&=\left(\varepsilon^2\int_{0}^{1}{\partial_{{\alpha}}^2V_{\varepsilon,n,v}^{(l)}(\alpha_0+u(\hat{\alpha}_{\varepsilon,n}^{(l)}-\alpha_0)|\hat{\alpha}_{\varepsilon,n}^{(l-1)})}du\right)\varepsilon^{-\frac{l}{v}}(\hat{\alpha}_{\varepsilon,n}^{(l)}-\alpha_0), \\
\varepsilon^{2-\frac{l}{v}}\partial_{{\alpha}}V_{\varepsilon,n,v}^{(l)}(\alpha_0|\hat{\alpha}_{\varepsilon,n}^{(l-1)})&=\varepsilon^{2-\frac{l}{v}}\partial_{{\alpha}}V_{\varepsilon,n,v}^{(l)}(\alpha_0|\alpha_0)\\
&\quad+\left(\varepsilon^{2-\frac{1}{v}}\int_{0}^{1}{\partial_{{\alpha\bar{\alpha}}}^2V_{\varepsilon,n,v}^{(l)}(\alpha_0|\alpha_0+u(\hat{\alpha}_{\varepsilon,n}^{(l-1)}-\alpha_0))}du\right)\varepsilon^{-\frac{l-1}{v}}(\hat{\alpha}_{\varepsilon,n}^{(l-1)}-\alpha_0).
\end{align*}
Hence, one deduces that
\begin{align}\label{taylor:1}
&-\varepsilon^{2-\frac{l}{v}}\partial_{{\alpha}}V_{\varepsilon,n,v}^{(l)}(\alpha_0|\alpha_0)-\left(\varepsilon^{2-\frac{1}{v}}\int_{0}^{1}{\partial_{{\alpha\bar{\alpha}}}^2V_{\varepsilon,n,v}^{(l)}(\alpha_0|\alpha_0+u(\hat{\alpha}_{\varepsilon,n}^{(l-1)}-\alpha_0))}du\right)\varepsilon^{-\frac{l-1}{v}}(\hat{\alpha}_{\varepsilon,n}^{(l-1)}-\alpha_0)\notag\\
&=\left(\varepsilon^2\int_{0}^{1}{\partial_{{\alpha}}^2V_{\varepsilon,n,v}^{(l)}(\alpha_0+u(\hat{\alpha}_{\varepsilon,n}^{(l)}-\alpha_0)|\hat{\alpha}_{\varepsilon,n}^{(l-1)})}du\right)\varepsilon^{-\frac{l}{v}}(\hat{\alpha}_{\varepsilon,n}^{(l)}-\alpha_0). 
\end{align}
In an analogous manner to the proof of $\varepsilon^{-\frac{1}{v}}(\hat{\alpha}_{\varepsilon,n}^{(1)}-\alpha_0)\overset{P}{\to}0$, it is sufficient to show the following properties:
\begin{align}
&-\varepsilon^{2-\frac{l}{v}}\partial_{{\alpha}}V_{\varepsilon,n,v}^{(l)}(\alpha_0|\alpha_0)\overset{P}{\to}0, \label{lem2:2-1}\\
&\sup_{u\in[0,1]}\left|\varepsilon^2\partial_{{\alpha}}^2V_{\varepsilon,n,v}^{(l)}(\alpha_0+u(\hat{\alpha}_{\varepsilon,n}^{(l)}-\alpha_0)|\hat{\alpha}_{\varepsilon,n}^{(l-1)})-2J_b(\alpha_0)\right|\overset{P}{\to}0, \label{lem2:2-2}\\
&\sup_{u\in[0,1]}\left|\varepsilon^{2-\frac{1}{v}}\partial_{{\alpha\bar{\alpha}}}^2V_{\varepsilon,n,v}^{(l)}(\alpha_0|\alpha_0+u(\hat{\alpha}_{\varepsilon,n}^{(l-1)}-\alpha_0))\right|\overset{P}{\to}0\label{lem2:2-3}.
\end{align}

Proof of \eqref{lem2:2-1}. By the definition of the contrast function $V_{\varepsilon,n,v}^{(l)}(\alpha|\bar{\alpha})$, we have
\begin{align}\label{lem2:ex}
-\varepsilon^{2-\frac{l}{v}}\partial_{{\alpha}}V_{\varepsilon,n,v}^{(l)}(\alpha_0|\alpha_0)&=\sum_{k=1}^{n}{\psi_{k}^{(l)}(\alpha_0)}+\varepsilon^{2-\frac{l}{v}}\cdot \frac{1}{n}\sum_{k=1}^{n}{R_d(n^{-1},X_{t_{k-1}^n})},
\end{align}
where
\begin{align*}
\psi_{k}^{(l)}(\alpha_0)&=2\varepsilon^{-\frac{l}{v}}{\partial_{{\alpha}}b(X_{t_{k-1}^n},\alpha_0)^\top P_{1,k}(\alpha_0)}.
\end{align*}
Since the second term of the right hand side of \eqref{lem2:ex} converges to 0 in probability, it is sufficient to show the following two properties:
\begin{align}
&\sum_{k=1}^{n}{\mathbb{E}\left[\psi^{(l)}_k(\alpha_0)|\mathcal{G}_{k-1}^n\right]}\overset{P}{\to}0, \label{uchida-3} \\
&\sum_{k=1}^{n}{\mathbb{E}\left[\left(\psi^{(l)}_k(\alpha_0)\right)^2|\mathcal{G}_{k-1}^n\right]}\overset{P}{\to}0. \label{uchida-4}
\end{align}
In an analogous manner to the case of $\psi_k^{(1)}$, 
the two terms on the left hand side in \eqref{uchida-3} and \eqref{uchida-4} 
are 
%evaluated from above by 
$O_P(\varepsilon^{2-\frac{l}{v}})$ and $O_P(\varepsilon^{2(1-\frac{l}{v})})$, respectively. Therefore, it follows from $l<v$ that the two properties hold.

Proof of \eqref{lem2:2-2}. By Lemma 2 in S\o rensen and Uchida \cite{Sorensen_Uchida_2003}, one deduces that
\begin{align*}
\varepsilon^2\partial_{{\alpha}}^2V_{\varepsilon,n,v}^{(l)}(\alpha|\hat{\alpha}_{\varepsilon,n}^{(l-1)})&=-\frac{2}{n}\sum_{k=1}^{n}{\left(\partial_{{\alpha}}^2b(X_{t_{k-1}^n},\alpha)\right)^\top \left(b(X_{t_{k-1}^n},\alpha_0)-b(X_{t_{k-1}^n},\alpha)\right)}\\
&\quad+\frac{2}{n}\sum_{k=1}^{n}{\left(\partial_{{\alpha}}b(X_{t_{k-1}^n},\alpha)\right)^\top\left(\partial_{{\alpha}}b(X_{t_{k-1}^n},\alpha)\right)}\\
&\quad-2\sum_{k=1}^{n}{\left(\partial_{{\alpha}}^2b(X_{t_{k-1}^n},\alpha)\right)^\top P_{1,k}(\alpha_0)}\\
&\quad-2\sum_{k=1}^{n}{\left(\partial_{{\alpha}}^2b(X_{t_{k-1}^n},\alpha)\right)^\top Q_{l,k}(\hat{\alpha}_{\varepsilon,n}^{(l-1)})}\\
&\overset{P}{\to} 2B_1(\alpha,\alpha_0)\quad\text{uniformly in }\alpha.
\end{align*}
Hence, it follows from the consistency of $\hat{\alpha}_{\varepsilon,n}^{(l)}$ and the uniform continuity of $B_1$ that for all $\delta>0$,
\begin{align*}
&P\left(\sup_{u\in[0,1]}\left|\varepsilon^2\partial_{{\alpha}}^2V_{\varepsilon,n,v}^{(l)}(\alpha_0+u(\hat{\alpha}_{\varepsilon,n}^{(l)}-\alpha_0)|\hat{\alpha}_{\varepsilon,n}^{(l-1)})-2J_b(\alpha_0)\right|>\delta\right)\\
&\le P\left(\sup_{\alpha\in\Theta_\alpha}\left|\varepsilon^2\partial_{{\alpha}}^2V_{\varepsilon,n,v}^{(l)}(\alpha|\hat{\alpha}_{\varepsilon,n}^{(l-1)})-2B_1(\alpha,\alpha_0)\right|>\frac{\delta}{2}\right)\\
&\quad+P\left(\sup_{u\in[0,1]}\left|2B_1(\alpha_0+u(\hat{\alpha}_{\varepsilon,n}^{(l)}-\alpha_0),\alpha_0)-2J_b(\alpha_0)\right|>\frac{\delta}{2}\right)\\
&\to 0\quad(\varepsilon\to0,n\to\infty).
\end{align*}
This implies \eqref{lem2:2-2}.

Proof of \eqref{lem2:2-3}. By Lemma 2 in S\o rensen and Uchida \cite{Sorensen_Uchida_2003} and [{\bf B}], 
it holds that for $1\le l_1,l_2\le p$, 
\begin{align*}
\varepsilon^{2-\frac{1}{v}}\partial_{{\alpha_{l_1}\bar{\alpha}_{l_2}}}^2V_{\varepsilon,n,v}^{(l)}(\alpha|\bar{\alpha})
&=2\varepsilon^{-\frac{1}{v}}\sum_{k=1}^{n}{\partial_{{\alpha}_{l_1}}b(X_{t_{k-1}^n},\alpha)^\top\partial_{{\bar{\alpha}_{l_2}}}Q_{l,k}(\bar{\alpha})}\\
&=2n^{-(1-\frac{\rho}{v})}\cdot(\varepsilon n^\rho)^{-\frac{1}{v}}\cdot\frac{1}{n}\sum_{k=1}^{n}{R(1,X_{t_{k-1}^n})}\\
&\overset{P}{\to}0\quad\text{uniformly in }(\alpha,\bar{\alpha}).
\end{align*}
Therefore, we have \eqref{lem2:2-3}. 

Next, we prove $\varepsilon^{-1}(\hat{\alpha}_{\varepsilon,n}^{(v)}-\alpha_0)=O_P(1)$. 
In the same manner as the derivation of \eqref{taylor:1}, it holds from Taylor's theorem that
\begin{align*}
&-\varepsilon\partial_{{\alpha}}V_{\varepsilon,n,v}^{(v)}(\alpha_0|\alpha_0)-\left(\varepsilon^{2-\frac{1}{v}}\int_{0}^{1}{\partial_{{\alpha\bar{\alpha}}}^2V_{\varepsilon,n,v}^{(v)}(\alpha_0|\alpha_0+u(\hat{\alpha}_{\varepsilon,n}^{(v-1)}-\alpha_0))}du\right)\varepsilon^{-\frac{v-1}{v}}(\hat{\alpha}_{\varepsilon,n}^{(v-1)}-\alpha_0)\notag\\
&=\left(\varepsilon^2\int_{0}^{1}{\partial_{{\alpha}}^2V_{\varepsilon,n,v}^{(v)}(\alpha_0+u(\hat{\alpha}_{\varepsilon,n}^{(v)}-\alpha_0)|\hat{\alpha}_{\varepsilon,n}^{(v-1)})}du\right)\varepsilon^{-1}(\hat{\alpha}_{\varepsilon,n}^{(v)}-\alpha_0).
\end{align*}
In order to prove $\varepsilon^{-1}(\hat{\alpha}_{\varepsilon,n}^{(v)}-\alpha_0)=O_P(1)$, it is sufficient to show the following properties:
\begin{align}
&-\varepsilon\partial_{{\alpha}}V_{\varepsilon,n,v}^{(v)}(\alpha_0|\alpha_0)=O_P(1), \label{lem2:3-1}\\
&\sup_{u\in[0,1]}\left|\varepsilon^2\partial_{{\alpha}}^2V_{\varepsilon,n,v}^{(v)}(\alpha_0+u(\hat{\alpha}_{\varepsilon,n}^{(v)}-\alpha_0)|\hat{\alpha}_{\varepsilon,n}^{(v-1)})-2J_b(\alpha_0)\right|\overset{P}{\to}0, \label{lem2:3-2}\\
&\sup_{u\in[0,1]}\left|\varepsilon^{2-\frac{1}{v}}\partial_{{\alpha\bar{\alpha}}}^2V_{\varepsilon,n,v}^{(v)}(\alpha_0|\alpha_0+u(\hat{\alpha}_{\varepsilon,n}^{(v-1)}-\alpha_0))\right|\overset{P}{\to}0. \label{lem2:3-3}
\end{align}
In analogous manners to the proofs of \eqref{lem2:2-2} and \eqref{lem2:2-3},  we can show \eqref{lem2:3-2} and \eqref{lem2:3-3}. For \eqref{lem2:3-1}, it follows from Lemma 4-(2) in Gloter and S\o rensen \cite{Gloter_2009} that
\begin{align*}
-\varepsilon\partial_{{\alpha}}V_{\varepsilon,n,v}^{(v)}(\alpha_0|\alpha_0)&=2\varepsilon^{-1}\sum_{k=1}^{n}{\partial_{{\alpha}}b(X_{t_{k-1}^n},\alpha_0)^\top \left(P_{1,k}(\alpha_0)-Q_{v,k}(\alpha_0)\right)}\\
&=2\varepsilon^{-1}\sum_{k=1}^{n}{\partial_{{\alpha}}b(X_{t_{k-1}^n},\alpha_0)^\top P_{v,k}(\alpha_0)}\\
&=O_P(1).
\end{align*}
Finally, we show the asymptotic normality of $\hat{\alpha}_{\varepsilon,n}^{(v)}$ 
in the same manner as the proof of Lemma \ref{lem:1}.  
Noting that $-\varepsilon\partial_{{\alpha}_{l}}V_{\varepsilon,n,v}^{(v)}(\alpha_0|\alpha_0)=\sum_{k=1}^{n}{\zeta_{k,1}^{l}}(\alpha_0)$ for $1\le l\le p$, we can show $-\varepsilon\partial_{{\alpha}}V_{\varepsilon,n,v}^{(v)}(\alpha_0|\alpha_0)\overset{d}{\to}N_p(0,4K_b(\theta_0))$. From \eqref{lem2:3-2} and this convergence, we complete the proof.
\end{proof}
%%%%%%%%%%%%%%%%%%%
\begin{proof}[\bf Proof of Theorem \ref{thm:2}]
We prove this theorem in the same order as the proof of Theorem \ref{thm:1}.

{\bf 1st step.} We   show $\hat{\beta}_{\varepsilon,n}\overset{P}{\to}\beta_0$. 
Noting that $V_{\varepsilon,n,v}^{(v+1)}(\beta|\bar{\alpha})=U_{\varepsilon,n,v}^{(2)}(\beta|\bar{\alpha})$, we can utilize the proof of Theorem \ref{thm:1}. In particular, by replacing $U_{\varepsilon,n,v}^{(2)}$ and $(\tilde{\alpha}_{\varepsilon,n}^{(1)},\tilde{\beta}_{\varepsilon,n})$ with $V_{\varepsilon,n,v}^{(v+1)}$ and $(\hat{\alpha}_{\varepsilon,n}^{(v)},\hat{\beta}_{\varepsilon,n})$ 
in the 1st step of the proof of Theorem \ref{thm:1}, we obtain the result. 

{\bf 2nd step.} We prove $\hat{\alpha}_{\varepsilon,n}\overset{P}{\to}\alpha_0$. It follows from 
Lemma 4 in Gloter and S\o rensen  \cite{Gloter_2009} and the consistency of $\hat{\beta}_{\varepsilon,n}$ that
\small
\begin{align*}
&\varepsilon^2\left(V_{\varepsilon,n,v}^{(v+2)}(\alpha|\hat{\alpha}_{\varepsilon,n}^{(v)},\hat{\beta}_{\varepsilon,n})-V_{\varepsilon,n,v}^{(v+2)}(\alpha_0|\hat{\alpha}_{\varepsilon,n}^{(v)},\hat{\beta}_{\varepsilon,n})\right)\\
&=2\sum_{k=1}^{n}{\left(b(X_{t_{k-1}^n},\alpha_0)-b(X_{t_{k-1}^n},\alpha)\right)^\top [\sigma\sigma^\top]^{-1}(X_{t_{k-1}^n},\hat{\beta}_{\varepsilon,n})}P_{1,k}(\alpha_0)\\
&\quad-2\sum_{k=1}^{n}{\left(b(X_{t_{k-1}^n},\alpha_0)-b(X_{t_{k-1}^n},\alpha)\right)^\top [\sigma\sigma^\top]^{-1}(X_{t_{k-1}^n},\hat{\beta}_{\varepsilon,n})}Q_{v,k}(\hat{\alpha}_{\varepsilon,n}^{(v)})\\
&\quad+\frac{1}{n}\sum_{k=1}^{n}{\left(b(X_{t_{k-1}^n},\alpha)-b(X_{t_{k-1}^n},\alpha_0)\right)^\top [\sigma\sigma^\top]^{-1}(X_{t_{k-1}^n},\beta_0)}\left(b(X_{t_{k-1}^n},\alpha)-b(X_{t_{k-1}^n},\alpha_0)\right)\\
&\quad+\frac{1}{n}\sum_{k=1}^{n}{\left(b(X_{t_{k-1}^n},\alpha)-b(X_{t_{k-1}^n},\alpha_0)\right)^\top \left([\sigma\sigma^\top]^{-1}(X_{t_{k-1}^n},\hat{\beta}_{\varepsilon,n})-[\sigma\sigma^\top]^{-1}(X_{t_{k-1}^n},\beta_0)\right)\left(b(X_{t_{k-1}^n},\alpha)-b(X_{t_{k-1}^n},\alpha_0)\right)}\\
&\overset{P}{\to}U_3(\alpha,\theta_0)\quad\text{uniformly in }\alpha.
\end{align*}
\normalsize
Therefore, in an analogous manner to the proof of the consistency for $\tilde{\alpha}_{\varepsilon,n}$, we have $\hat{\alpha}_{\varepsilon,n}\overset{P}{\to}\alpha_0$.

{\bf 3rd step.} We prove the asymptotic normality for $\hat{\theta}_{\varepsilon,n}$. Using Taylor's theorem,  we have the following expansions:
\begin{align*}
&-\varepsilon\partial_{{\alpha}}V_{\varepsilon,n,v}^{(v+2)}(\alpha_0|\hat{\alpha}_{\varepsilon,n}^{(v)},\hat{\beta}_{\varepsilon,n})=\left(\varepsilon^2\int_{0}^{1}{\partial_{{\alpha}}}^2V_{\varepsilon,n,v}^{(v+2)}(\alpha_0+u(\hat{\alpha}_{\varepsilon,n}-\alpha_0)|\hat{\alpha}_{\varepsilon,n}^{(1)},\hat{\beta}_{\varepsilon,n})du\right)\varepsilon^{-1}(\hat{\alpha}_{\varepsilon,n}-\alpha_0), \\
&-\frac{1}{\sqrt{n}}\partial_{{\beta}}V_{\varepsilon,n,v}^{(v+1)}(\beta_0|\hat{\alpha}_{\varepsilon,n}^{(v)})=
\left(\frac{1}{n}\int_{0}^{1}{\partial_{{\beta}}}^2V_{\varepsilon,n,v}^{(v+1)}(\beta_0+u(\hat{\beta}_{\varepsilon,n}-\beta_0)|\hat{\alpha}_{\varepsilon,n}^{(v)})du\right)\sqrt{n}(\hat{\beta}_{\varepsilon,n}-\beta_0), \\
&\varepsilon\partial_{{\alpha}}V_{\varepsilon,n,v}^{(v+2)}(\alpha_0|\hat{\alpha}_{\varepsilon,n}^{(v)},\hat{\beta}_{\varepsilon,n})\\
&\quad=\varepsilon\partial_{{\alpha}}V_{\varepsilon,n,v}^{(v+2)}(\alpha_0|\hat{\alpha}_{\varepsilon,n}^{(v)},\beta_0)+
\left(\frac{\varepsilon}{\sqrt{n}}\int_{0}^{1}{\partial_{\alpha\beta}^2V_{\varepsilon,n,v}^{(v+2)}(\alpha_0|\hat{\alpha}_{\varepsilon,n}^{(v)},\beta_0+u(\hat{\beta}_{\varepsilon,n}-\beta_0))}du\right)\sqrt{n}(\hat{\beta}_{\varepsilon,n}-\beta_0).
\end{align*}
Using these expressions, we calculate that 
\begin{align*}
\Gamma_{\varepsilon,n}^2=C_{\varepsilon,n}^2\Lambda_{\varepsilon,n}^2,
\end{align*}
where
\begin{align*}
\Gamma_{\varepsilon,n}^2:=\begin{pmatrix}
-\varepsilon\partial_{{\alpha}}V_{\varepsilon,n,v}^{(v+2)}(\alpha_0|\hat{\alpha}_{\varepsilon,n}^{(v)},\beta_0)\\
-\frac{1}{\sqrt{n}}\partial_{{\beta}}V_{\varepsilon,n,v}^{(v+1)}(\beta_0|\hat{\alpha}_{\varepsilon,n}^{(v)})
\end{pmatrix},\quad
\Lambda_{\varepsilon,n}^2:=\begin{pmatrix}
\varepsilon^{-1}(\hat{\alpha}_{\varepsilon,n}-\alpha_0)\\
\sqrt{n}(\hat{\beta}_{\varepsilon,n}-\beta_0)
\end{pmatrix}, 
\end{align*}
\begin{align*}
C_{\varepsilon,n}^2:=\begin{pmatrix}
\varepsilon^2\int_{0}^{1}{\partial_{{\alpha}}}^2V_{\varepsilon,n,v}^{(v+2)}(\alpha_0+u(\hat{\alpha}_{\varepsilon,n}-\alpha_0)|\hat{\alpha}_{\varepsilon,n}^{(1)},\hat{\beta}_{\varepsilon,n})du &\frac{\varepsilon}{\sqrt{n}}\int_{0}^{1}{\partial_{\alpha\beta}^2V_{\varepsilon,n,v}^{(v+2)}(\alpha_0|\hat{\alpha}_{\varepsilon,n}^{(v)},\beta_0+u(\hat{\beta}_{\varepsilon,n}-\beta_0))}du\\
0 &\frac{1}{n}\int_{0}^{1}{\partial_{{\beta}}}^2V_{\varepsilon,n,v}^{(v+1)}(\beta_0+u(\hat{\beta}_{\varepsilon,n}-\beta_0)|\hat{\alpha}_{\varepsilon,n}^{(v)})du
\end{pmatrix}.
\end{align*}
In an analogous manner to the proof of Theorem 1 in S\o rensen and Uchida \cite{Sorensen_Uchida_2003}, it is sufficient to show the following convergences:
\begin{align}
&\sup_{u\in[0,1]}\left|\varepsilon^2\partial_{{\alpha}}^2V_{\varepsilon,n,v}^{(v+2)}(\alpha_0+u(\hat{\alpha}_{\varepsilon,n}-\alpha_0)|\hat{\alpha}_{\varepsilon,n}^{(v)},\hat{\beta}_{\varepsilon,n})-2I_b(\theta_0)\right|\overset{P}{\to}0, \label{thm2:1}\\
&\sup_{u\in[0,1]}\left|\frac{1}{n}\partial_{{\beta}}^2V_{\varepsilon,n,v}^{(v+1)}(\beta_0+u(\hat{\beta}_{\varepsilon,n}-\beta_0)|\hat{\alpha}_{\varepsilon,n}^{(v)})-2I_\sigma(\beta_0)\right|\overset{P}{\to}0, \label{thm2:2}\\
&\sup_{u\in[0,1]}\left|\frac{\varepsilon}{\sqrt{n}}\partial_{\alpha\beta}^2V_{\varepsilon,n,v}^{(v+2)}(\alpha_0|\hat{\alpha}_{\varepsilon,n}^{(v)},\beta_0+u(\hat{\beta}_{\varepsilon,n}-\beta_0))\right|\overset{P}{\to}0, \label{thm2:3}\\
&\quad-\varepsilon\partial_{{\alpha}}V_{\varepsilon,n,v}^{(v+2)}(\alpha_0|\hat{\alpha}_{\varepsilon,n}^{(v)},\beta_0)\overset{d}{\to} N_{p}(0,4I_b(\theta_0)), \label{thm2:4}\\
&\quad-\frac{1}{\sqrt{n}}\partial_{{\beta}}V_{\varepsilon,n,v}^{(v+1)}(\beta_0|\hat{\alpha}_{\varepsilon,n}^{(v)})
\overset{d}{\to} N_{q}(0,4I_\sigma(\beta_0)). \label{thm2:5}
\end{align}
%For \eqref{thm2:2} and \eqref{thm2:5}, 
By the definitions of $U_{\varepsilon,n,v}^{(2)}$ and $V_{\varepsilon,n,v}^{(v+1)}$, we have already shown \eqref{thm2:2} and \eqref{thm2:5} in the proof of Theorem \ref{thm:1}.

Proof of \eqref{thm2:1}. It follows from Lemma 2 in S\o rensen and Uchida \cite{Sorensen_Uchida_2003}, Lemma 4 in Gloter and S\o rensen \cite{Gloter_2009}
and the consistency of $\hat{\beta}_{\varepsilon,n}$ that
\begin{align*}
&\varepsilon^2 \partial_{{\alpha}}^2 V_{\varepsilon,n,v}^{(v+2)}(\alpha|\hat{\alpha}_{\varepsilon,n}^{(v)},\hat{\beta}_{\varepsilon,n})\notag\\
&=
-2\sum_{k=1}^{n}{\left(\partial_{{\alpha}}^2 b(X_{t_{k-1}^n},\alpha) \right)^\top[\sigma\sigma^\top]^{-1}(X_{t_{k-1}^n},\hat{\beta}_{\varepsilon,n})P_{1,k}(\alpha_0)}\\
&\quad+2\sum_{k=1}^{n}{\left(\partial_{{\alpha}}^2 b(X_{t_{k-1}^n},\alpha) \right)^\top[\sigma\sigma^\top]^{-1}(X_{t_{k-1}^n},\hat{\beta}_{\varepsilon,n})Q_{v,k}(\hat{\alpha}_{\varepsilon,n}^{(v)})}\\
&\quad+\frac{2}{n}\sum_{k=1}^{n}{ \left(\partial_{{\alpha}}^2 b(X_{t_{k-1}^n},\alpha) \right)^\top[\sigma\sigma^\top]^{-1}(X_{t_{k-1}^n},\beta_0)\left(b(X_{t_{k-1}^n},\alpha)-b(X_{t_{k-1}^n},\alpha_0)\right)}\\
&\quad+\frac{2}{n}\sum_{k=1}^{n}{\left(\partial_{{\alpha}}^2 b(X_{t_{k-1}^n},\alpha)\right)^\top\left([\sigma\sigma^\top]^{-1}(X_{t_{k-1}^n},\hat{\beta}_{\varepsilon,n})-[\sigma\sigma^\top]^{-1}(X_{t_{k-1}^n},\beta_0)\right)\left(b(X_{t_{k-1}^n},\alpha)-b(X_{t_{k-1}^n},\alpha_0)\right)}\\
&\quad+\frac{2}{n}\sum_{k=1}^{n}{\left(\partial_{{\alpha}} b(X_{t_{k-1}^n},\alpha)\right)^\top[\sigma\sigma^\top]^{-1}(X_{t_{k-1}^n},\beta_0)\left(\partial_{{\alpha}} b(X_{t_{k-1}^n},\alpha)\right)}\\
&\quad+\frac{2}{n}\sum_{k=1}^{n}{\left(\partial_{{\alpha}} b(X_{t_{k-1}^n},\alpha)\right)^\top\left([\sigma\sigma^\top]^{-1}(X_{t_{k-1}^n},\hat{\beta}_{\varepsilon,n})-[\sigma\sigma^\top]^{-1}(X_{t_{k-1}^n},\beta_0)\right)\left(\partial_{{\alpha}} b(X_{t_{k-1}^n},\alpha)\right)}\\
&\overset{P}{\to} 2B_2(\alpha,\theta_0)\quad\text{uniformly in }\alpha.
\end{align*}
In an analogous manner to the proof of \eqref{thm1:1}, it follows from the consistency for $\hat{\alpha}_{\varepsilon,n}$ and the uniform continuity of $B_2$ that for all $\delta>0$,
\begin{align*}
&P\left(\sup_{u\in[0,1]}\left|\varepsilon^2\partial_{{\alpha}}^2V_{\varepsilon,n,v}^{(v+2)}(\alpha_0+u(\hat{\alpha}_{\varepsilon,n}-\alpha_0)|\hat{\alpha}_{\varepsilon,n}^{(v)},\hat{\beta}_{\varepsilon,n})-2I_b(\theta_0)\right|>\delta\right)\\
&\le P\left(\sup_{\alpha\in\Theta_\alpha}\left|\varepsilon^2\partial_{{\alpha}}^2V_{\varepsilon,n,v}^{(v+2)}(\alpha|\hat{\alpha}_{\varepsilon,n}^{(v)},\hat{\beta}_{\varepsilon,n})-2B_2(\alpha,\theta_0)\right|>\frac{\delta}{2}\right)\\
&\quad+P\left(\sup_{u\in[0,1]}\left|2B_2(\alpha_0+u(\hat{\alpha}_{\varepsilon,n}-\alpha_0),\theta_0)-2I(\theta_0)\right|>\frac{\delta}{2}\right)\\
&\to 0\quad(\varepsilon\to0,n\to\infty).
\end{align*}
This implies \eqref{thm2:1}.

Proof of \eqref{thm2:3}.
By using the Lipschitz continuity of $n^2Q_{v,k}(\alpha)$, Lemma \ref{lem:2} in this paper and Lemma 4-(2)  in Gloter and S\o rensen \cite{Gloter_2009}, it holds that
\small
\begin{align*}
\frac{\varepsilon}{\sqrt{n}}\partial_{\alpha\beta}^2V_{\varepsilon,n,v}^{(v+2)}(\alpha_0|\hat{\alpha}_{\varepsilon,n}^{(v)},\beta)&=
-2(\varepsilon\sqrt{n})^{-1}\sum_{k=1}^{n}{\left(\partial_{{\alpha}} b(X_{t_{k-1}^n},\alpha)\right)^\top\left( \partial_{{\beta}}[\sigma\sigma^\top]^{-1}(X_{t_{k-1}^n},\beta)\right)P_{v,k}(\alpha_0)}\\
&\quad+2(\varepsilon\sqrt{n})^{-1}\sum_{k=1}^{n}{\left(\partial_{{\alpha}} b(X_{t_{k-1}^n},\alpha)\right)^\top \left(\partial_{{\beta}}[\sigma\sigma^\top]^{-1}(X_{t_{k-1}^n},\beta)\right)\left(Q_{v,k}(\hat{\alpha}_{\varepsilon,n}^{(v)})-Q_{v,k}(\alpha_0)\right)}\\
&\overset{P}{\to}0 \quad\text{uniformly in }\beta.
\end{align*}
\normalsize
Hence, we have \eqref{thm2:3}.

Proof of \eqref{thm2:4}.
For $1\le l\le p$, it follows from the Lipschitz continuity 
of $n^2Q_{v,k}(\alpha)$, Lemma \ref{lem:2} in this paper and Lemma 4-(2)  in Gloter and S\o rensen \cite{Gloter_2009} that
\begin{align*}
-\varepsilon\partial_{{\alpha}_l}V_{\varepsilon,n,v}^{(v+2)}(\alpha_0|\hat{\alpha}_{\varepsilon,n}^{(v)},\beta_0)&=2\varepsilon^{-1}\sum_{k=1}^{n}{\left(\partial_{{\alpha}_l} b(X_{t_{k-1}^n},\alpha_0)\right)^\top[\sigma\sigma^\top]^{-1}(X_{t_{k-1}^n},\beta_0)P_{v,k}(\alpha_0)}\\
&\quad-2\varepsilon^{-1}\sum_{k=1}^{n}{\left(\partial_{{\alpha}_l} b(X_{t_{k-1}^n},\alpha_0)\right)^\top[\sigma\sigma^\top]^{-1}(X_{t_{k-1}^n},\beta_0)\left(Q_{v,k}(\hat{\alpha}_{\varepsilon,n}^{(v)})-Q_{v,k}(\alpha_0)\right)}\\
&=\sum_{k=1}^{n}{\xi_{k,1}^l(\theta_0)}+o_P(1).
\end{align*}
Since the main term of $-\varepsilon\partial_{{\alpha}_l}V_{\varepsilon,n,v}^{(v+2)}(\alpha_0|\hat{\alpha}_{\varepsilon,n}^{(v)},\beta_0)$ is the same as that of $-\varepsilon\partial_{{\alpha}_l}U_{\varepsilon,n,v}^{(3)}(\alpha_0|\beta_0)$, 
we utilize the results in the proof of Theorem \ref{thm:1} and complete the proof. 
\end{proof}

%%%%%%%%%%%%%%%%%%%%%%%%%%%%%%%%%%%%%%%%%%%%%%
For the proof of theorem \ref{thm:3-small}, we first show the following lemma.

\begin{lemma}\label{lem:3-small}
Assume {\bf{[A1]}}-{\bf{[A4]}} and {\bf{[B]}}. 
Then it follows 
%under null hypothesis 
that
\begin{align*}
\varepsilon^{-1}(\tilde{\alpha}_{\varepsilon,n}^{(1)}-\tilde{\alpha}_{\varepsilon,n}^{(1),H_0})&\overset{d}{\to}
(J_b^{-1}(\alpha_0)-G_1^r(\alpha_0))Z
\quad(\mathrm{under\  } H_0^{(1)}),\\
\sqrt{n}(\tilde{\beta}_{\varepsilon,n}-\tilde{\beta}_{\varepsilon,n}^{H_0})&\overset{d}{\to}
(I_\sigma^{-1}(\beta_0)-G_2^s(\beta_0))Y
\quad(\mathrm{under\  } H_0^{(2)})
\end{align*}
as $\varepsilon\to0$ and $n\to\infty$, 
where $Z$ and $Y$ are random vectors which have normal distribution $N_p(0,K_b(\theta_0))$ and $N_q(0,I_\sigma(\beta_0))$, respectively.
\end{lemma}

\begin{proof}[\bf Proof of Lemma \ref{lem:3-small}]
It follows from Taylor's theorem that
\begin{align*}
\varepsilon\partial_{{\alpha}}U_{\varepsilon,n,v}^{(1)}(\tilde{\alpha}_{\varepsilon,n}^{(1),H_0})-\varepsilon\partial_{{\alpha}}U_{\varepsilon,n,v}^{(1)}(\alpha_0)=J_n(\tilde{\alpha}_{\varepsilon,n}^{(1),H_0},\alpha_0)\varepsilon^{-1}(\tilde{\alpha}_{\varepsilon,n}^{(1),H_0}-\alpha_0),
\end{align*} 
where
\begin{align*}
J_n(\alpha,\bar{\alpha})=\int_{0}^{1}{\varepsilon^2\partial_{{\alpha}}^2U_{\varepsilon,n,v}^{(1)}(\bar{\alpha}+u(\alpha-\bar{\alpha}))du}.
\end{align*}
From \eqref{main1:lem1} and \eqref{main2:lem1}, one has that under $H_0^{(1)}$, 
\begin{align*}
J_n(\tilde{\alpha}_{\varepsilon,n}^{(1),H_0},\alpha_0)\overset{P}{\to} 2 J_b(\alpha_0),\quad-\varepsilon\partial_{{\alpha}}U_{\varepsilon,n,v}^{(1)}(\alpha_0)\overset{d}{\to} 2Z.
\end{align*}
Noting that the first $r$ components of $(\tilde{\alpha}_{\varepsilon,n}^{(1),H_0}-\alpha_0)$ are all zero, we can calculate
\begin{align*}
G_1^r(\alpha_0)J_b(\alpha_0)\varepsilon^{-1}(\tilde{\alpha}_{\varepsilon,n}^{(1),H_0}-\alpha_0)=\varepsilon^{-1}(\tilde{\alpha}_{\varepsilon,n}^{(1),H_0}-\alpha_0).
\end{align*}
Hence, it follows from Lemma \ref{lem:1} and $G_1^r(\alpha_0)\partial_{{\alpha}}U_{\varepsilon,n,v}^{(1)}(\tilde{\alpha}_{\varepsilon,n}^{(1),H_0})=0$ that under $H_0^{(1)}$,
\begin{align*}
-\varepsilon G_1^r(\alpha_0) \partial_{{\alpha}}U_{\varepsilon,n,v}^{(1)}(\alpha_0)
&=G_1^r(\alpha_0)J_n(\tilde{\alpha}_{\varepsilon,n}^{(1),H_0},\alpha_0)\varepsilon^{-1}(\tilde{\alpha}_{\varepsilon,n}^{(1),H_0}-\alpha_0)\\
&=G_1^r(\alpha_0)(2J_b(\alpha_0)+o_P(1))\varepsilon^{-1}(\tilde{\alpha}_{\varepsilon,n}^{(1),H_0}-\alpha_0)\\
&=2\varepsilon^{-1}(\tilde{\alpha}_{\varepsilon,n}^{(1),H_0}-\alpha_0)+o_P(1).
\end{align*}
Therefore, one has that under $H_0^{(1)}$,
\begin{align*}
-\varepsilon  \partial_{{\alpha}}U_{\varepsilon,n,v}^{(1)}(\tilde{\alpha}_{\varepsilon,n}^{(1),H_0})
&=-\varepsilon  \partial_{{\alpha}}U_{\varepsilon,n,v}^{(1)}(\alpha_0)-J_n(\tilde{\alpha}_{\varepsilon,n}^{(1),H_0},\alpha_0)\varepsilon^{-1}(\tilde{\alpha}_{\varepsilon,n}^{(1),H_0}-\alpha_0)\\
&=-\varepsilon  \partial_{{\alpha}}U_{\varepsilon,n,v}^{(1)}(\alpha_0)-J_b(\alpha_0)\cdot2\varepsilon^{-1}(\tilde{\alpha}_{\varepsilon,n}^{(1),H_0}-\alpha_0)+o_P(1)\\
&=-\varepsilon  \partial_{{\alpha}}U_{\varepsilon,n,v}^{(1)}(\alpha_0)+ J_b(\alpha_0)G_1^r(\alpha_0)\varepsilon \partial_{{\alpha}}U_{\varepsilon,n,v}^{(1)}(\alpha_0)+o_P(1)\\
&=-\left(E_p-J_b(\alpha_0)G_1^r(\alpha_0)\right)\varepsilon  \partial_{{\alpha}}U_{\varepsilon,n,v}^{(1)}(\alpha_0) +o_P(1)\\
&\overset{d}{\to}2\left(E_p-J_b(\alpha_0)G_1^r(\alpha_0)\right)Z.
\end{align*}
On the other hand, it follows from Taylor's theorem and the consistency of estimator for $\alpha$ that
\begin{align*}
-\varepsilon\partial_{{\alpha}}U_{\varepsilon,n,v}^{(1)}(\tilde{\alpha}_{\varepsilon,n}^{(1),H_0})&=J_n(\tilde{\alpha}_{\varepsilon,n}^{(1),H_0},\tilde{\alpha}_{\varepsilon,n}^{(1)})\varepsilon^{-1}(\tilde{\alpha}_{\varepsilon,n}^{(1)}-\tilde{\alpha}_{\varepsilon,n}^{(1),H_0})\\
&=2J_b(\alpha_0)\varepsilon^{-1}(\tilde{\alpha}_{\varepsilon,n}^{(1)}-\tilde{\alpha}_{\varepsilon,n}^{(1),H_0})+o_P(1).
\end{align*}
 As a result, from {\bf{[A4]}}, it holds that
\begin{align*}
\varepsilon^{-1}(\tilde{\alpha}_{\varepsilon,n}^{(1)}-\tilde{\alpha}_{\varepsilon,n}^{(1),H_0})&\overset{d}{\to} J_b^{-1}(\alpha_0)\left(E_p-J_b(\alpha_0)G_1^r(\alpha_0)\right)Z\\
&=\left(J_b^{-1}(\alpha_0)-G_1^r(\alpha_0)\right)Z.
\end{align*}

For the second statement, it follows from Taylor's theorem with respect to $\beta$ that
\begin{align*}
\frac{1}{\sqrt{n}}\partial_{{\beta}}U_{\varepsilon,n,v}^{(2)}(\tilde{\beta}_{\varepsilon,n}^{H_0}|\tilde{\alpha}_{\varepsilon,n}^{(1)})-\frac{1}{\sqrt{n}}\partial_{{\beta}}U_{\varepsilon,n,v}^{(2)}(\beta_0|\tilde{\alpha}_{\varepsilon,n}^{(1)})=I_{n,\beta}(\tilde{\beta}_{\varepsilon,n}^{H_0},\beta_0|\tilde{\alpha}_{\varepsilon,n}^{(1)})\sqrt{n}(\tilde{\beta}_{\varepsilon,n}^{H_0}-\beta_0),
\end{align*} 
where
\begin{align*}
I_{n,\beta}(\beta,\bar{\beta}|\alpha)=\int_{0}^{1}{\frac{1}{n}\partial_{{\beta}}^2U_{\varepsilon,n,v}^{(2)}(\bar{\beta}+u(\beta-\bar{\beta})|\alpha)du}.
\end{align*}
Moreover,  It follows from Taylor's theorem with respect to $\alpha$ that
\begin{align*}
\frac{1}{\sqrt{n}}\partial_{{\beta}}U_{\varepsilon,n,v}^{(2)}(\beta_0|\tilde{\alpha}_{\varepsilon,n}^{(1)})=\frac{1}{\sqrt{n}}\partial_{{\beta}}U_{\varepsilon,n,v}^{(2)}(\beta_0|\alpha_0)+I_{n,\alpha\beta}(\beta_0|\tilde{\alpha}_{\varepsilon,n}^{(1)},\alpha_0)\varepsilon^{-1}(\tilde{\alpha}_{\varepsilon,n}^{(1)}-\alpha_0),
\end{align*} 
where
\begin{align*}
I_{n,\alpha\beta}(\beta|\alpha,\bar{\alpha})=\int_{0}^{1}{\frac{\varepsilon}{\sqrt{n}}\partial_{{\alpha\beta}}^2U_{\varepsilon,n,v}^{(2)}(\beta|\bar{\alpha}+u(\alpha-\bar{\alpha}))du}.
\end{align*}
From \eqref{thm1:2}-\eqref{thm1:4}, one has that under $H_0^{(2)}$, 
\begin{align*}
I_{n,\beta}(\tilde{\beta}_{\varepsilon,n}^{H_0},\beta_0|\tilde{\alpha}_{\varepsilon,n}^{(1)})\overset{P}{\to} 2 I_\sigma(\beta_0),\quad
I_{n,\alpha\beta}(\beta_0|\tilde{\alpha}_{\varepsilon,n}^{(1)},\alpha_0)\overset{P}{\to}0,\quad
-\frac{1}{\sqrt{n}}\partial_{{\alpha}}U_{\varepsilon,n,v}^{(2)}(\beta_0|\alpha_0)\overset{d}{\to} 2Y.
\end{align*}
We can calculate
\begin{align*}
G_2^s(\beta_0)I_\sigma(\beta_0)\sqrt{n}(\tilde{\beta}_{\varepsilon,n}^{H_0}-\beta_0)=\sqrt{n}(\tilde{\beta}_{\varepsilon,n}^{H_0}-\beta_0),
\end{align*}
%and it follows from the facts $\sqrt{n}(\tilde{\beta}_{\varepsilon,n}^{H_0}-\beta_0)=O_P(1)$ and $G_2^s(\beta_0)\partial_{{\alpha}}U_{\varepsilon,n,v}^{(2)}(\tilde{\beta}_{\varepsilon,n}^{H_0}|\tilde{\alpha}_{\varepsilon,n}^{(1)})=0$ that under $H_0^{(2)}$,
and from the facts that 
$\sqrt{n}(\tilde{\beta}_{\varepsilon,n}^{H_0}-\beta_0)=O_P(1)$ and $G_2^s(\beta_0)\partial_{{\alpha}}U_{\varepsilon,n,v}^{(2)}(\tilde{\beta}_{\varepsilon,n}^{H_0}|\tilde{\alpha}_{\varepsilon,n}^{(1)})=0$,  
it follows that under $H_0^{(2)}$,
 \begin{align*}
-\frac{1}{\sqrt{n}} G_2^s(\beta_0)\partial_{{\alpha}}U_{\varepsilon,n,v}^{(2)}(\beta_0|\alpha_0)
&=G_2^s(\beta_0)I_{n,\beta}(\tilde{\beta}_{\varepsilon,n}^{H_0},\beta_0|\tilde{\alpha}_{\varepsilon,n}^{(1)})\sqrt{n}(\tilde{\beta}_{\varepsilon,n}^{H_0}-\beta_0)+o_P(1)\\
&=2G_2^s(\beta_0)I_{\sigma}(\beta_0)\sqrt{n}(\tilde{\beta}_{\varepsilon,n}^{H_0}-\beta_0)+o_P(1)\\
&=2\sqrt{n}(\tilde{\beta}_{\varepsilon,n}^{H_0}-\beta_0)+o_P(1).
\end{align*}
Therefore, one has that under $H_0^{(2)}$,
\begin{align*}
-\frac{1}{\sqrt{n}}\partial_{{\beta}}U_{\varepsilon,n,v}^{(2)}(\tilde{\beta}_{\varepsilon,n}^{H_0}|\tilde{\alpha}_{\varepsilon,n}^{(1)})
&=-\frac{1}{\sqrt{n}}\partial_{{\beta}}U_{\varepsilon,n,v}^{(2)}(\beta_0|\alpha_0)
-I_{n,\beta}(\tilde{\beta}_{\varepsilon,n}^{H_0},\beta_0|\tilde{\alpha}_{\varepsilon,n}^{(1)})\sqrt{n}(\tilde{\beta}_{\varepsilon,n}^{H_0}-\beta_0)+o_P(1)\\
&=-\frac{1}{\sqrt{n}}\partial_{{\beta}}U_{\varepsilon,n,v}^{(2)}(\beta_0|\alpha_0)
-I_{\sigma}(\beta_0)\cdot2\sqrt{n}(\tilde{\beta}_{\varepsilon,n}^{H_0}-\beta_0)+o_P(1)\\
&=-\left(E_q-I_\sigma(\beta_0)G_2^s(\beta_0)\right)\frac{1}{\sqrt{n}}\partial_{{\beta}}U_{\varepsilon,n,v}^{(2)}(\beta_0|\alpha_0) +o_P(1)\\
&\overset{d}{\to}2\left(E_q-I_\sigma(\beta_0)G_2^s(\beta_0)\right)Y.
\end{align*}
On the other hand, it follows from Taylor's theorem and the consistency of estimator for $(\alpha,\beta)$ that
\begin{align*}
-\frac{1}{\sqrt{n}}\partial_{{\beta}}U_{\varepsilon,n,v}^{(2)}(\tilde{\beta}_{\varepsilon,n}^{H_0}|\tilde{\alpha}_{\varepsilon,n}^{(1)})&=I_{n,\beta}(\tilde{\beta}_{\varepsilon,n}^{H_0},\tilde{\beta}_{\varepsilon,n}|\tilde{\alpha}_{\varepsilon,n}^{(1)})\sqrt{n}(\tilde{\beta}_{\varepsilon,n}-\tilde{\beta}_{\varepsilon,n}^{H_0})\\
&=2I_\sigma(\beta_0)\sqrt{n}(\tilde{\beta}_{\varepsilon,n}-\tilde{\beta}_{\varepsilon,n}^{H_0})+o_P(1).
\end{align*}
 As a result, from {\bf{[A4]}}, it holds that
\begin{align*}
\sqrt{n}(\tilde{\beta}_{\varepsilon,n}-\tilde{\beta}_{\varepsilon,n}^{H_0})&\overset{d}{\to} I_\sigma^{-1}(\beta_0)\left(E_q-I_\sigma(\beta_0)G_2^s(\beta_0)\right)Y\\
&=\left(I_\sigma^{-1}(\beta_0)-G_2^s(\beta_0)\right)Y.
\end{align*}
\end{proof}
%%%%%%%%%%%%%%%%%%%%%%%%%%
\begin{proof}[\bf Proof of Theorem \ref{thm:3-small}]
One has from Taylor's theorem that
\begin{align*}
\tilde{\Lambda}_{n}^{(1)}&=U_{\varepsilon,n,v}^{(1)}\left(\tilde{\alpha}_{\varepsilon,n}^{(1),H_0}\right)-U_{\varepsilon,n,v}^{(1)}\left(\tilde{\alpha}_{\varepsilon,n}^{(1)}\right)\\
&=\left(\varepsilon^{-1}(\tilde{\alpha}_{\varepsilon,n}^{(1)}-\tilde{\alpha}_{\varepsilon,n}^{(1),H_0})\right)^\top
\left(\int_{0}^{1}{(1-u)\varepsilon^2\partial_{{\alpha}}^2}U_{\varepsilon,n,v}^{(1)}(\tilde{\alpha}_{\varepsilon,n}^{(1)}+u(\tilde{\alpha}_{\varepsilon,n}^{(1),H_0}-\tilde{\alpha}_{\varepsilon,n}^{(1)}))du\right) \varepsilon^{-1}(\tilde{\alpha}_{\varepsilon,n}^{(1)}-\tilde{\alpha}_{\varepsilon,n}^{(1),H_0}).
\end{align*}
From \eqref{main1:lem1}, 
we have that under $H_0^{(1)}$, 
\begin{align*}
\int_{0}^{1}{(1-u)\varepsilon^2\partial_{{\alpha}}^2}U_{\varepsilon,n,v}^{(1)}(\tilde{\alpha}_{\varepsilon,n}^{(1)}+u(\tilde{\alpha}_{\varepsilon,n}^{(1),H_0}-\tilde{\alpha}_{\varepsilon,n}^{(1)}))du\overset{P}{\to}J_b(\alpha_0).
\end{align*}
Therefore, Lemma  \ref{lem:3-small} and  the continuous mapping theorem 
yield that
\begin{align*}
\tilde{\Lambda}_{n}^{(1)}&\overset{d}{\to} Z^\top\left(J_b^{-1}(\alpha_0)-G_1^r(\alpha_0)\right)J_b(\alpha_0) \left(J_b^{-1}(\alpha_0)-G_1^r(\alpha_0)\right)Z\\
&=Z^\top \left(J_b^{-1}(\alpha_0)-G_1^r(\alpha_0)\right)Z\sim \pi_r.
\end{align*}

Next, Taylor's theorem implies that
\begin{align*}
\tilde{\Lambda}_{n}^{(2)}&=U_{\varepsilon,n,v}^{(2)}\left(\tilde{\beta}_{\varepsilon,n}^{H_0}|\tilde{\alpha}_{\varepsilon,n}^{(1)}\right)-U_{\varepsilon,n,v}^{(2)}\left(\tilde{\beta}_{\varepsilon,n}|\tilde{\alpha}_{\varepsilon,n}^{(1)}\right)\\
&=\left(\sqrt{n}(\tilde{\beta}_{\varepsilon,n}-\tilde{\beta}_{\varepsilon,n}^{H_0})\right)^\top
\left(\int_{0}^{1}{\frac{1-u}{n}\partial_{{\beta}}^2}U_{\varepsilon,n,v}^{(2)}(\tilde{\beta}_{\varepsilon,n}+u(\tilde{\beta}_{\varepsilon,n}^{H_0}-\tilde{\beta}_{\varepsilon,n}|\tilde{\alpha}_{\varepsilon,n}^{(1)}))du\right) \sqrt{n}(\tilde{\beta}_{\varepsilon,n}-\tilde{\beta}_{\varepsilon,n}^{H_0}).
\end{align*}
From \eqref{thm1:2}, 
we have that under $H_0^{(2)}$,
\begin{align*}
\int_{0}^{1}{\frac{1-u}{n}\partial_{{\beta}}^2}U_{\varepsilon,n,v}^{(2)}(\tilde{\beta}_{\varepsilon,n}+u(\tilde{\beta}_{\varepsilon,n}^{H_0}-\tilde{\beta}_{\varepsilon,n}|\tilde{\alpha}_{\varepsilon,n}^{(1)}))du\overset{P}{\to}I_\sigma(\beta_0).
\end{align*}
Therefore, we define $Y'$ as $Y=I_\sigma^{\frac{1}{2}}Y'$, and 
Lemma \ref{lem:3-small} yields that
\begin{align*}
\tilde{\Lambda}_{n}^{(2)}&\overset{d}{\to} Y^\top\left(I_\sigma^{-1}(\beta_0)-G_2^s(\beta_0)\right)I_\sigma(\beta_0) \left(I_\sigma^{-1}(\beta_0)-G_2^s(\beta_0)\right)Y\\
&=Y'^\top I_\sigma^{\frac{1}{2}}(\beta_0)\left(I_\sigma^{-1}(\beta_0)-G_2^s(\beta_0)\right)I_\sigma^{\frac{1}{2}}(\beta_0)Y'\\
&=Y'^\top(E_p-I_\sigma^{\frac{1}{2}}(\beta_0)G_2^s(\beta_0)I_\sigma^{\frac{1}{2}}(\beta_0))Y'\sim \chi^2_s.
\end{align*}
\end{proof}

%%%%%%%%%%%%%%%%%%%%%%%

In order to prove Theorem \ref{thm:4-small}, we show the following lemma.
\begin{lemma}\label{lem:4-small}
Assume {\bf{[A1]}}-{\bf{[A4]}}, {\bf{[B]}} and {\bf{[C]}}. 
Then it follows 
%under alternatives 
that
\begin{align*}
\tilde{\alpha}_{\varepsilon,n}^{(1),H_0}\overset{P}{\to}\alpha_0^{H_0}\quad(\mathrm{under\  } H_1^{(1)}),\quad
\tilde{\beta}_{\varepsilon,n}^{H_0}\overset{P}{\to}\beta_0^{H_0} \quad(\mathrm{under\  } H_1^{(2)})
\end{align*}
as $\varepsilon\to0$ and $n\to\infty$.
\end{lemma}

\begin{proof}[\bf Proof of Lemma \ref{lem:4-small}]
It follows from \eqref{U1_conv-small} that under $H_1^{(1)}$,
\begin{align*}
\sup_{\alpha\in\Theta_\alpha}\left|\varepsilon^2\left(U_{\varepsilon,n,v}^{(1)}(\alpha)-U_{\varepsilon,n,v}^{(1)}(\alpha_0)\right)-U_1(\alpha,\alpha_0)\right|\overset{P}{\to}0.
\end{align*}
Assumption {\bf{[C]}}-(i) implies the following: For any $\delta_1>0$, there exists $\delta_1'>0$ such that for any $\alpha\in\Theta_{\alpha}^{H_0}$,
\begin{align}\label{C1:1-small}
|\alpha-\alpha_0^{H_0}|\ge\delta_1\ \Longrightarrow\ U_1(\alpha,\alpha_0)-{U}_1(\alpha_0^{H_0},\alpha_0)>\delta_1'.
\end{align}
By the definition of $\tilde{\alpha}_{\varepsilon,n}^{(1),H_0}$ and \eqref{C1:1-small}, it holds that
\begin{align*}
P(|\tilde{\alpha}_{\varepsilon,n}^{(1),H_0}-\alpha_0^{H_0}|\ge\delta_1)&\le
P\left(U_1(\tilde{\alpha}_{\varepsilon,n}^{(1),H_0},\alpha_0)-{U}_1(\alpha_0^{H_0},\alpha_0)>\delta_1'\right)\\
&=P\left(\left\{U_1(\tilde{\alpha}_{\varepsilon,n}^{(1),H_0},\alpha_0)-\varepsilon^2\left(U_{\varepsilon,n,v}^{(1)}(\tilde{\alpha}_{\varepsilon,n}^{(1),H_0})-U_{\varepsilon,n,v}^{(1)}(\alpha_0)\right)\right\}\right.\\
&\quad\quad+\varepsilon^2\left(U_{\varepsilon,n,v}^{(1)}(\tilde{\alpha}_{\varepsilon,n}^{(1),H_0})-U_{\varepsilon,n,v}^{(1)}(\alpha_0^{H_0})\right)\\
&\quad\quad+\left. \left\{\varepsilon^2\left(U_{\varepsilon,n,v}^{(1)}(\alpha_0^{H_0})-U_{\varepsilon,n,v}^{(1)}(\alpha_0)\right)-{U}_1(\alpha_0^{H_0},\alpha_0)\right\}>\delta_1'\right)\\
&\le 2P\left(\sup_{\alpha\in\Theta_\alpha}\left|\varepsilon^2\left(U_{\varepsilon,n,v}^{(1)}(\alpha)-U_{\varepsilon,n,v}^{(1)}(\alpha_0)\right)-U_1(\alpha,\alpha_0)\right|>\frac{\delta_1'}{3}\right)\\
&\quad\quad+P\left(\varepsilon^2\left(U_{\varepsilon,n,v}^{(1)}(\tilde{\alpha}_{\varepsilon,n}^{(1),H_0})-U_{\varepsilon,n,v}^{(1)}(\alpha_0^{H_0})\right)>\frac{\delta_1'}{3}\right)\\
&\to 0,
\end{align*}
which implies $\tilde{\alpha}_{\varepsilon,n}^{(1),H_0}\overset{P}{\to}\alpha_0^{H_0}$ under $H_1^{(1)}$.

%Similarly, 
Next, it follows from \eqref{conv:U2} that under $H_1^{(2)}$,
\begin{align*}
\sup_{\beta\in\Theta_\beta}\left|\frac{1}{n}\left(U_{\varepsilon,n,v}^{(2)}(\beta|\tilde{\alpha}_{\varepsilon,n}^{(1)})-U_{\varepsilon,n,v}^{(2)}(\beta_0|\tilde{\alpha}_{\varepsilon,n}^{(1)})\right)
-U_2(\beta,\beta_0)\right|\overset{P}{\to}0
\end{align*}
Assumption {\bf{[C]}}-(ii) implies the following: For any $\delta_2>0$, there exists $\delta_2'>0$ such that for any $\beta\in\Theta_{\beta}^{H_0}$,
\begin{align}\label{C1:2-small}
|\beta-\beta_0^{H_0}|\ge\delta_2\ \Longrightarrow\ U_2(\beta;\beta_0)-{U}_2(\beta_0^{H_0},\beta_0)>\delta_2'.
\end{align}
By the definition of $\tilde{\beta}_{\varepsilon,n}^{H_0}$ and \eqref{C1:2-small}, it holds that
\begin{align*}
P(|\tilde{\beta}_{\varepsilon,n}^{H_0}-\beta_0^{H_0}|\ge\delta_2)&\le
P\left(U_2(\tilde{\beta}_{\varepsilon,n}^{H_0},\beta_0)-{U}_2(\beta_0^{H_0},\beta_0)>\delta_2'\right)\\
&=P\left(\left\{U_2(\tilde{\beta}_{\varepsilon,n}^{H_0},\beta_0)-\frac{1}{n}\left(U_{\varepsilon,n,v}^{(2)}(\tilde{\beta}_{\varepsilon,n}^{H_0}|\tilde{\alpha}_{\varepsilon,n}^{(1)})-U_{\varepsilon,n,v}^{(2)}(\beta_0|\tilde{\alpha}_{\varepsilon,n}^{(1)})\right)\right\}\right.\\
&\quad\quad+\frac{1}{n}\left(U_{\varepsilon,n,v}^{(2)}(\tilde{\beta}_{\varepsilon,n}^{H_0}|\tilde{\alpha}_{\varepsilon,n}^{(1)})-U_{\varepsilon,n,v}^{(2)}(\beta_0^{H_0}|\tilde{\alpha}_{\varepsilon,n}^{(1)})\right)\\
&\quad\quad+\left. \left\{\frac{1}{n}\left(U_{\varepsilon,n,v}^{(2)}(\beta_0^{H_0}|\tilde{\alpha}_{\varepsilon,n}^{(1)})-U_{\varepsilon,n,v}^{(2)}(\beta_0|\tilde{\alpha}_{\varepsilon,n}^{(1)})\right)-{U}_2(\beta_0^{H_0},\beta_0)\right\}>\delta_2'\right)\\
&\le 2P\left(\sup_{\beta\in\Theta_\beta}\left|\frac{1}{n}\left(U_{\varepsilon,n,v}^{(2)}(\beta|\tilde{\alpha}_{\varepsilon,n}^{(1)})-U_{\varepsilon,n,v}^{(2)}(\beta_0|\tilde{\alpha}_{\varepsilon,n}^{(1)})\right)
-U_2(\beta,\beta_0)\right|>\frac{\delta_2'}{3}\right)\\
&\quad\quad+P\left(\frac{1}{n}\left(U_{\varepsilon,n,v}^{(2)}(\tilde{\beta}_{\varepsilon,n}^{H_0}|\tilde{\alpha}_{\varepsilon,n}^{(1)})-U_{\varepsilon,n,v}^{(2)}(\beta_0^{H_0}|\tilde{\alpha}_{\varepsilon,n}^{(1)})\right)>\frac{\delta_2'}{3}\right)\\
&\to 0,
\end{align*}
which implies $\tilde{\beta}_{\varepsilon,n}^{H_0}\overset{P}{\to}\beta_0^{H_0}$ under $H_1^{(2)}$. 
\end{proof}

%%%%%%%%%%%%%%%%%%%%%%%%
\begin{proof}[\bf Proof of Theorem \ref{thm:4-small}]
Under $H_1^{(1)}$, it holds from the proof of Lemma \ref{lem:1} and Lemma \ref{lem:4-small} that
\begin{align*}
\tilde{\alpha}_{\varepsilon,n}^{(1)}\overset{P}{\to}\alpha_0,\quad 
\tilde{\alpha}_{\varepsilon,n}^{(1),H_0}\overset{P}{\to}\alpha_0^{H_0}\neq\alpha_0.
\end{align*}
Hence, one has from \eqref{U1_conv-small} that
\begin{align*}
\varepsilon^2 \Lambda_{n}^{(1)}&=\varepsilon^2\left\{U_{\varepsilon,n,v}^{(1)}\left(\tilde{\alpha}_{\varepsilon,n}^{(1),H_0}\right)-U_{\varepsilon,n,v}^{(1)}\left(\tilde{\alpha}_{\varepsilon,n}^{(1)}\right)\right\}\\
&=\varepsilon^2\left\{U_{\varepsilon,n,v}^{(1)}\left(\tilde{\alpha}_{\varepsilon,n}^{(1),H_0}\right)-U_{\varepsilon,n,v}^{(1)}\left(\alpha_0\right)\right\}-\varepsilon^2\left\{U_{\varepsilon,n,v}^{(1)}\left(\tilde{\alpha}_{\varepsilon,n}^{(1)}\right)-U_{\varepsilon,n,v}^{(1)}\left(\alpha_0\right)\right\}\\
&\overset{P}{\to} U_1(\alpha_0^{H_0},\alpha_0)-U_1(\alpha_0,\alpha_0)\\
&=U_1(\alpha_0^{H_0},\alpha_0).
\end{align*}
Since it follows from ${\bf{[A3]}}$ that $U_1(\alpha_0^{H_0},\alpha_0)>0$ under $H_1^{(1)}$, 
one has that for any $\delta\in(0,1)$,
\begin{align*}
P(\Lambda_n^{(1)}<\pi_r(\delta))=P(\varepsilon^2\Lambda_n^{(1)}<\varepsilon^2\pi_r(\delta))\to0.
\end{align*}

Next, by the proof of Theorem \ref{thm:1} and Lemma \ref{lem:4-small}, we obtain that
under $H_1^{(2)}$, 
\begin{align*}
\tilde{\beta}_{\varepsilon,n}\overset{P}{\to}\beta_0,\quad 
\tilde{\beta}_{\varepsilon,n}^{H_0}\overset{P}{\to}\beta_0^{H_0}\neq\beta_0.
\end{align*}
Therefore, by \eqref{conv:U2}, 
\begin{align*}
\frac{1}{n} \Lambda_{n}^{(2)}&=\frac{1}{n}\left\{U_{\varepsilon,n,v}^{(2)}\left(\tilde{\beta}_{\varepsilon,n}^{H_0}|\tilde{\alpha}_{\varepsilon,n}^{(1)}\right)-U_{\varepsilon,n,v}^{(2)}\left(\tilde{\beta}_{\varepsilon,n}|\tilde{\alpha}_{\varepsilon,n}^{(1)}\right)\right\}\\
&=\frac{1}{n}\left\{U_{\varepsilon,n,v}^{(2)}\left(\tilde{\beta}_{\varepsilon,n}^{H_0}|\tilde{\alpha}_{\varepsilon,n}^{(1)}\right)-U_{\varepsilon,n,v}^{(2)}\left(\beta_0|\tilde{\alpha}_{\varepsilon,n}^{(1)}\right)\right\}-\frac{1}{n}\left\{U_{\varepsilon,n,v}^{(2)}\left(\tilde{\beta}_{\varepsilon,n}|\tilde{\alpha}_{\varepsilon,n}^{(1)}\right)-U_{\varepsilon,n,v}^{(2)}\left(\beta_0|\tilde{\alpha}_{\varepsilon,n}^{(1)}\right)\right\}\\
&\overset{P}{\to} U_2(\beta_0^{H_0},\beta_0)-U_2(\beta_0,\beta_0)\\
&=U_2(\beta_0^{H_0},\beta_0).
\end{align*}
It holds from {\bf{[A3]}} that 
$U_2(\beta_0^{H_0},\beta_0)>0$ under $H_1^{(2)}$. 
Consequently, we have that for any $\delta\in(0,1)$,
\begin{align*}
P(\Lambda_n^{(2)}<\chi^2_s(\delta))=P\left(\frac{1}{n}\Lambda_n^{(1)}<\frac{1}{n}\chi^2_s(\delta)\right)\to0.
\end{align*}
This completes the proof.
\end{proof}

%\section*{Acknowledgements}
%This work was partially supported by JST CREST Grant Number JPMJCR14D7 and JSPS
%KAKENHI Grant Number JP17H01100.

%%%%%%%%%%%%%%%%%%% Reference  %%%%%%%%%%%%%%%%%%%%%%

\bibliographystyle{abbrv}
\bibliography{main.bbl}
%\bibliography{Reference}

\begin{comment}

\end{comment}

\end{document}